\documentclass[11pt]{article}
\usepackage[utf8]{inputenc}
\usepackage{graphicx} 

\usepackage[a4paper, margin=2cm]{geometry}
\usepackage{amsmath}
\usepackage{amssymb}
\usepackage{amsthm}
\usepackage{appendix}
\usepackage{bbm}
\usepackage{xcolor}
\usepackage{pgfplots}

\usepackage{tabularx}

\newcommand{\dd}{\mathrm{d}}
\newcommand{\ddiv}{\mathrm{div\,}}
\newcommand{\R}{\mathbb{R}}
\newcommand{\D}{\mathbb{D}}
\newcommand{\T}{\mathbb{T}}
\newcommand{\loc}{\mathrm{loc}}
\newcommand{\eq}[1]{\begin{equation}
\begin{split}
#1
\end{split}
\end{equation}}

\newcommand{\x}{\mathrm{x}}

\newtheorem{tw}{Theorem}
\newtheorem{lem}[tw]{Lemma}
\newtheorem{prop}[tw]{Proposition}
\newtheorem{rem}[tw]{Remark}
\newtheorem{df}[tw]{Definition}

\numberwithin{equation}{section}
\numberwithin{tw}{section}

\title{Construction  of weak solutions to a pressureless viscous model driven by nonlocal attraction--repulsion}
\author{Piotr B. Mucha$^*\;$, Maja Szlenk$^\dagger\;$, Ewelina Zatorska$^\S\;$}
\date{\today}

\begin{document}

\maketitle

 {
\footnotesize
\centerline{$^*\;$Faculty of Mathematics, Informatics, and Mechanics, University of Warsaw, }
\centerline{ul. Banacha 2, Warsaw 02-097, Poland}
\centerline{\small \texttt{p.mucha@mimuw.edu.pl}}

\bigbreak
\centerline{$^\dagger\;$Faculty of Mathematics, Informatics, and Mechanics, University of Warsaw, }
\centerline{ul. Banacha 2, Warsaw 02-097, Poland}
\centerline{\small \texttt{m.szlenk@uw.edu.pl}}

\bigbreak
\centerline{$^\S\;$Mathematics Institute, University of Warwick}
\centerline{Zeeman Building, Coventry CV4 7AL, United Kingdom}
\centerline{\small \texttt{ewelina.zatorska@warwick.ac.uk}}

}

\bigbreak

\begin{abstract}
    We analyze the pressureless Navier-Stokes system with nonlocal attraction--repulsion forces. Such systems appear in the context of models of collective behavior. We prove the existence of weak solutions on the whole space $\R^3$ in the case of density-dependent degenerate viscosity. For the nonlocal term it is assumed that the interaction kernel has the quadratic growth at infinity and almost quadratic singularity at zero. Under these assumptions, we derive the analog of the Bresch--Desjardins and Mellet--Vasseur estimates for the nonlocal system. In particular, we are able to adapt the approach of Vasseur and Yu \cite{vasseur-yu,vasseur-yu2} to construct a weak solution.
\end{abstract}

{\bf Keywords:}  weak solutions; hydrodynamic models of swarming; compressible Navier-Stokes equations; density-dependent viscosity.




\section{Introduction}

Hydrodynamical equations featuring nonlocal forces frequently emerge in the modeling of collective behavior. Such systems find notable applications in phenomena such as flocking and swarming observed in various animal species and bacteria (see e. g. \cite{bacteria,flocking}). On a microscopic scale, these interactions are characterized by $N$ systems of Ordinary Differential Equations (ODEs), each corresponding to one of the $N$ agents. These systems essentially represent Newton's equations, where the governing forces originate from agents attempting to adapt their positions and velocities in response to the presence of others. In contrast to the interactions between the gas particles, the interactions between the agents occur before the collision takes place, resulting in the nonlocal character of these forces.
In instances where the ensemble size is rather small, the natural choice is to adopt the ODE description, see, for instance \cite{Vic}.  Yet, when dealing with a substantial number of agents or particles, change of perspective is unavoidable. This brings us to the macroscopic scale description.
Via the mean-field limit, when the number of interacting agents is very large ($N\to\infty$), one can derive a significantly more concise description in terms of Partial Differential Equations (PDEs). The resulting systems of PDEs capture the evolution of the averaged quantities: the macroscopic density $\varrho$ and the macroscopic velocity $u$ of the individuals. The nonlocal interactions persevere the limit passage giving rise to the nolocal pressure-like forces in the macroscopic momentum equation. The examples of the above mathematical models and corresponding mean-field limits  can be found in \cite{Carrillo2010, carrillo2011,chuang-et-al}.

The purpose of this manuscript is to analyze one of such nonlocal hydrodynamic  models of collective motion described by the  degenerate pressureless Navier-Stokes equations:
\begin{equation}\label{main}
    \begin{aligned}
    \partial_t\varrho + \ddiv(\varrho u) &= 0 \\
    \partial_t(\varrho u) + \ddiv(\varrho u\otimes u) -\ddiv(\varrho\D u) + \varrho\nabla(K\ast\varrho) &=0
    \end{aligned}
    \quad \text{in} \quad [0,T]\times\R^3,
\end{equation}
where $\varrho\colon [0,T]\times\R^3\to\R_+$ is the density of the particles, and $u\colon [0,T]\times\R^3\to\R^3$ the velocity vector field, $\D u=\frac 12 (\nabla u + \nabla^T u)$ is the symmetric part of the gradient of the velocity. The convolution term 
$$\varrho\nabla(K\ast\varrho)(t,x)=
\varrho(t,x)\int_{\R^3} \nabla K(x-y)\varrho(t,y)\, \dd y,$$ 
corresponds to the attractive-repulsive forces modelling the collective behaviour.

In comparison to the classical compressible Navier-Stokes equations, our model experiences a twofold degeneration. Firstly, the conventional stress tensor incorporates density dependence, and it vanishes when the density equals $0$. Secondly, the system is pressureless, or rather, in place of the classical strongly repulsive pressure force $\nabla \varrho^\gamma$, we consider repulsion smeared around the origin and additional long-range attraction described by $\varrho \nabla K\ast\varrho$. The interaction potential $ K$ therefore  consists of two parts: 
\eq{\label{Def:K}
K(x) = \frac{1}{|x|^\alpha} + \frac{1}{2}|x|^2, }
where $\alpha\in (0,2)$.
The first part forces the particles to keep some small distance apart, and the second prevents them from spreading too far from each other, as illustrated in Fig. \ref{Fig}, below.

\begin{figure}[ht] 
  \centering
\begin{tikzpicture}
  \begin{axis}[
    axis lines=middle,
    xlabel={$|x|$},
    ylabel={$k(|x|)$},
    xmin=-0.5,
    xmax=5.5,
    ymin=0,
    ymax=9,
    width=11cm,
    height=7cm,
    thick, 
    samples=100 
  ]
   \addplot[blue, ultra thick, domain=0.25:2]{(2-x)^3+2};
  \addplot[blue, ultra thick, domain=2:5]{ 0.6*(x-2)^2+2}; \end{axis}

\foreach \x/\y in { 1.9/3.5, 7.3/3.4, 3.1/1.7, 
  5.25/1.8} {    \draw[fill=red] (\x,\y) circle (0.2);}

\draw[->, ultra thick] (2,3.2) -- (2.25,2.6);
\draw[->, ultra thick] (7.1,3.2) -- (6.7,2.8);
\draw[->, ultra thick] (3.35,1.55) -- (3.75,1.45);
\draw[->, ultra thick] (5.0,1.65) -- (4.55,1.4);
\end{tikzpicture}
\caption{Attractive-repulsive potential $k(|x|)=K(x)$.}
\label{Fig}

\end{figure}
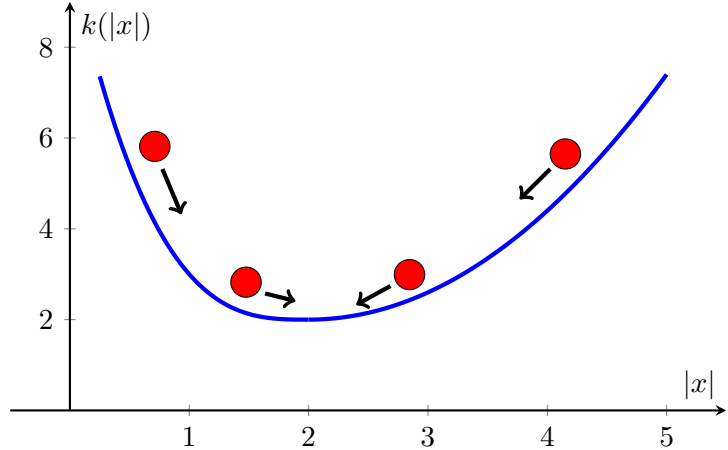

Note that the singular part of the potential $K$ up to a constant corresponds to a Riesz potential
\[ I_s(f)(x)=(-\Delta)^{-s/2}f(x) = \frac{1}{c_s}\int_{\R^3}\frac{f(y)}{|x-y|^{3-s}}\;\dd y, \quad c_s = \pi^{3/2}2^s\frac{\Gamma\left(\frac{s}{2}\right)}{\Gamma\left(\frac{3-s}{2}\right)} \]
for $s=3-\alpha$. The case $\alpha=1$ corresponds to the Newtonian potential and then equations (\ref{main}) become the Navier-Stokes-Poisson system. 

Our goal is to construct a weak solution to system \eqref{main} with potential $K$ of the form \eqref{Def:K}. Assuming formally that solutions are smooth and well behaving at infinity, testing the momentum equation by $u$ we get the basic energy law
\begin{equation}
    \frac{\dd}{\dd t} \int_{\R^3} \left( \frac 12 \varrho |u|^2 + \frac12 \rho K \ast \varrho\right)\, \dd x +
    \int_{\R^3} \mu(\varrho) |\D u|^2\, \dd x =0.
\end{equation}
There are two key elements different from the classical fluid mechanic systems:

\smallskip

$\clubsuit$ The part of energy corresponding to $K$ determines the potential energy which spurs the motion of the fluid. Mathematically, the weak solutions lack a-priori  estimates for the density in the $L^\gamma$ space corresponding to the pressure exponent.

\smallskip 

$\clubsuit$ The second point where we diverge from classical models is the form of the dissipative term. We  consider the viscosity coefficient to be equal to $\varrho$. In particular, the dissipation, and the  disappears if the density of the individuals is equal to 0. The effect of the magnitude of the density in the dissipation of the flow is  natural for the hydrodynamic systems of interacting agents.  It can be seen as macroscopic consequence of alignment forces, \cite[Chapter 5]{MMPZ}. From the mathematical perspective,  this kind of viscosity results in a different kind of regularity expected for weak solutions. 
\medskip

\smallskip 

The study of the  highlighted aspects has gained increased attention within the community over the past decade. In case of the potential proportional to $|x|^{-\alpha}$ for $\alpha\in(0,d)$, stationary solutions and their stability were thoroughly studied for example in \cite{carrillo2017,Calvez2017,carrillo2018}. In \cite{car-gwiazda} it was shown that the nonlocal Euler system admits infinitely many weak solutions. Concerning the existence of weak solutions, in \cite{car-wrob-zat} the authors proved the existence of weak solutions to the compressible Navier-Stokes system with damping, when the nonlocal term is sufficiently integrable. The existence  for Navier-Stokes-Poisson system on the torus in the same, degenerate setting as equation (\ref{main}), was shown in \cite{NS-P}.

The equation (\ref{main}) is a special case of a Navier-Stokes system 
\begin{equation}\label{density-dependent}
    \begin{aligned}
        \varrho_t+\ddiv(\varrho u) &= 0, \\
        \partial_t(\varrho u)+\ddiv(\varrho u\otimes u) -\ddiv(\mu(\varrho)\D u) - \nabla(\lambda(
        \varrho)\ddiv u) + \nabla P &= 0,
    \end{aligned}
\end{equation}
with viscosity coefficients satisfying the algebraic relation $\lambda(\varrho)=\varrho\mu'(\varrho)-\mu(\varrho)$. In two dimensions, equation (\ref{density-dependent}) with $\mu=\varrho$, $\lambda=0$ is used to describe shallow water flow. When $\mu(0)=0$ the system is degenerate, in the sense that the stress tensor does not provide us the typical $L^2$ integrability of the velocity gradient. An important tool to deal with this problem is then the Bresch-Desjardins inequality, first established in \cite{BD1}. It allows to 
get an estimate on $\nabla\varphi(\varrho)$ in $L^\infty(0,T;L^2)$ for some $\varphi$ depending on $\mu$ (in the case $\mu=\varrho, \lambda=0$,  one gets the estimate on $\nabla\sqrt{\varrho}$). With that information at hand, the authors of \cite{BD2,BD3} constructed the weak solutions to \eqref{density-dependent} with the additional regularizing terms in the momentum equation. Another interesting result, involving density-dependent viscosity, concerns the existence of weak solutions for quantum fluids, analysed for example in \cite{jungel}. To construct weak solutions to the degenerate Navier-Stokes system (\ref{density-dependent}) one needs another inequality, first used by Mellet and Vasseur in \cite{mellet-vasseur}. It provides the $L^\infty(0,T;L^1)$ estimate on the quantity $\varrho|u|^2\ln(1+|u|^2)$, which in consequence allows to derive compactness in $L^2$ of $\sqrt{\varrho}u$ (and in turn proves stability of solutions). Using this idea, Vasseur and Yu presented full, rigorous construction of global weak solutions to (\ref{density-dependent}) with $\mu=\varrho,\lambda=0$ in \cite{vasseur-yu,vasseur-yu2}, see also \cite{li-xin}. The case of more general viscosities was covered in \cite{BVY}, where the authors showed the existence of renormalized solutions, following the definition from \cite{L-VV}.

An important feature of the density-dependent viscosity case and the Bresch-Desjardins inequality is that it allows to derive a priori estimates for the density without the use of the pressure term. In the classical theory of weak solutions to Navier-Stokes equations with constant viscosities, developed by Lions \cite{lions} and Feireisl \cite{feireisl}, the construction of solutions is possible when $P\sim\varrho^\gamma$ for $\gamma>3/2$. Without the pressure term, or with too low value of $\gamma$, the density is not integrable enough to show the compactness of the approximating sequence. In the density-dependent viscosity case, by virtue of the Bresch-Desjardins inequality, we get the estimate on the gradient of the density, and then compactness follows straight from the Aubin-Lions lemma. The lack of pressure term is natural in the case of systems derived from the models of collective behaviour, since they describe the interactions of different nature than gases. The pressureless systems were obtained as a mean-field limit from the agent-based model for example in \cite{FigalliKang} in the presence of the nonlocal alignment forces. The Euler-Poisson system with quadratic confinement was also recently considered in a spherically symmetric multi-dimensional setting by Carrillo and Shu \cite{CS23} and we refer to this paper for up to date overview of results on that system in the context of continuous collective behaviour models.  In the context of our work, particularly interesting results were obtained even earlier in the one-dimensional setting \cite{car-choi-zatorska} and \cite{car-wrob-zat2}, where similar form of the nonlocal kernel was considered. In \cite{car-choi-zatorska} the authors analysed the asymptotics and critical thresholds for pressureless Euler system, whereas in \cite{car-wrob-zat2} they showed that these solutions can be approximated by the solutions of the corresponding Navier-Stokes type system with degenerate viscosity. In higher dimensions without symmetry assumptions, pressureless limits in the density-dependent viscosity framework were performed in \cite{Liang, CagDon}. In \cite{haspot}, a multidimensional version of result by Haspot and Zatorska \cite{HaZa} was proved, demonstrating that pressureless limit of  (\ref{density-dependent}) leads to the porous medium equation for "well prepared" data. However, according to our knowledge there are no corresponding results concerning the nonlocal systems.

\section{The main result}

Throughout the paper, by $\T_L^3$ we denote three-dimensional torus of size $2L$ (i.e. $\T_L^3\sim [-L,L]^3$). By $L^p(0,T;X(\Omega))$ we denote the Bochner space of functions on $[0,T]\times\Omega$, with the norm
\[ \|f\|_{L^p(0,T;X(\Omega))}^p=\int_0^T\|f(t,\cdot)\|_{X(\Omega)}^p\;\dd t. \]
In our case $\Omega$ is either $\T_L^3$ or $\R^3$ and for simplicity we will write $X$ instead of $X(\Omega)$, when from the context it is clear what the spacial domain is.

We supplement problem \eqref{main} with the initial data
\begin{equation}\label{A1} 
    \varrho(0,x)=\varrho_0(x),\quad (\varrho u)(0,x)=m_0(x),
    \mbox{ \ such that \ }\varrho_0\geq 0,\quad \sqrt{\varrho_0} \in H^1(\R^3).
\end{equation}
Moreover, for $F$ defined as
\eq{\label{def:F}
F(z)=\frac{1+z^2}{2}\ln(1+z^2),}
we assume that
\begin{equation}\label{A2} \int_{\R^3}\varrho_0 F(|u_0|)\;\dd x + \iint_{\R^3\times\R^3} F(|x-y|)\varrho_0(x)\varrho_0(y)\;\dd x\,\dd y <\infty, \end{equation}
where we define $u_0=\frac{m_0}{\varrho_0}$ on the set $\{x\in\R^3:\varrho_0(x)>0\}$. In particular, since $\frac{1}{2}z^2\leq F(z)$ for large $z$, from (\ref{A2}) it follows  that
\[ \int_{\R^3}\frac{1}{2}\varrho_0|u_0|^2\;\dd x + \iint_{\R^3\times\R^3}\frac{1}{2}|x-y|^2\varrho_0(x)\varrho_0(y)\;\dd x\,\dd y <\infty. \]

\begin{df}\label{Def:main}
We say that $(\varrho,u)$ is a weak finite energy solution to (\ref{main}) on $[0,T]\times\R^3$ with initial conditions $(\varrho_0,m_0)$, if
\[\begin{aligned} \varrho &\in L^\infty(0,T;L^1), \quad
\sqrt{\varrho} \in L^\infty(0,T;H^1), \\
\sqrt{\varrho}u &\in L^\infty(0,T;L^2), \quad
\overline{\varrho\D u} \in L^2(0,T;W^{-1,1}),
\end{aligned}\]
\[ \int_0^T\!\!\!\iint_{\R^3\times\R^3}|x-y|\varrho(t,x)\varrho(t,y)\;\dd x\,\dd y < \infty, \]
and for each $\varphi\in C_0^\infty([0,T)\times\R^3;\R)$ and $\psi\in C_0^\infty([0,T)\times\R^3;\R^3)$ we have
\[ -\int_0^T\!\!\!\int_{\R^3}\varrho\partial_t\varphi \;\dd x\,\dd t - \int_0^T\!\!\!\int_{\R^3}\varrho u\cdot\nabla\varphi \;\dd x\,\dd t = \int_{\R^3}\varrho_0\varphi(0,\cdot)\;\dd x \]
and
\begin{multline*} -\int_{\R^3} m_0\psi(0,\cdot)\;\dd x -\int_0^T\!\!\!\int_{\R^3} \sqrt\varrho \sqrt\varrho u\partial_t\psi \;\dd x\,\dd t - \int_0^T\!\!\!\int_{\R^3}(\sqrt{\varrho}u\otimes\sqrt{\varrho}u):\nabla\psi\;\dd x\,\dd t \\
+ \langle \overline{\varrho\D u},\nabla\psi \rangle  +\int_0^T\!\!\!\int_{\R^3}\varrho\nabla(K\ast\varrho)\cdot\psi \;\dd x\dd t = 0,
\end{multline*}
where we define
\[ \langle \overline{\varrho\D u},\nabla\psi\rangle := -\int_0^T\!\!\!\int_{\R^3} \varrho u \cdot(\Delta\psi+\nabla\ddiv\psi) \;\dd x\,\dd t - 2\int_0^T\!\!\!\int_{\R^3}(\nabla\sqrt{\varrho}\otimes \sqrt{\varrho}u):\nabla\psi\;\dd x\,\dd t. \]
\end{df}

\begin{rem}\label{Du_rem1}
Note that in the sense of Definition \ref{Def:main}, the velocity itself is not defined on the set where $\varrho=0$. Because of that, we operate with the variable $\sqrt{\varrho}u$ instead, and $u$ is defined only via $u(t,x)=\frac{(\sqrt{\varrho}u)(t,x)}{\sqrt{\varrho}(t,x)}$ for $(t,x)$ such that $\varrho(t,x)\neq 0$. In particular the gradient $\D u$ is not well defined as well. Because of that, we denote the stress tensor by $\overline{\varrho\D u}$ instead, which is defined using the relation
\[\begin{aligned} \overline{\varrho\D u} &= \D(\varrho u) - \nabla\varrho\otimes u = \D(\varrho u) - 2\nabla\sqrt{\varrho}\otimes \sqrt{\varrho}u. \end{aligned}\]
To avoid unnecessary complications of the notation, later on we will drop the bars and just write $\varrho\D u$, keeping in mind the above definition.
\end{rem}
 Under these assumptions, our main result states:

\begin{tw}\label{main_th}
    Let $(\varrho_0,m_0)$ satisfy (\ref{A1}-\ref{A2}). Then there exists a global in time weak solution $(\varrho,u)$ to (\ref{main}) in the sense of Definition \ref{Def:main}. In addition, this solution satisfies:\\
 (i) the  energy estimate
    \begin{equation}\label{main_energy} 
    \begin{aligned}
    \sup_{t\geq 0}\int_{\R^3}\frac{1}{2}\varrho|u|^2\;\dd x + \int_0^\infty\!\!\!\int_{\R^3}\varrho|\D u|^2\;\dd x\,\dd t + \sup_{t\geq 0}\iint_{\R^3\times\R^3} K(x-y)\varrho(t,x)\varrho(t,y)\;\dd x\,\dd y \\
    \leq \frac{1}{2}\int_{\R^3}\varrho_0|u_0|^2\;\dd x + \iint_{\R^3\times\R^3}K(x-y)\varrho_0(x)\varrho_0(y)\;\dd x\,\dd y; 
    \end{aligned}\end{equation}
   (ii) the Bresch-Desjardins estimate
    \begin{equation}\label{main_BD}
\sup_{t\in[0,T]}\int_{\R^3}|\nabla\sqrt{\varrho}|^2\;\dd x + \int_0^T\!\!\!\int_{\R^3}\varrho|\nabla u-\nabla^Tu|^2\;\dd x\, \dd t \leq C(T), \end{equation}
(iii) the Mellet-Vasseur estimate
\begin{equation}\label{main_MV}
\sup_{t\in[0,T]}\int_{\R^3}\varrho F(|u|)\;\dd x + \sup_{t\in [0,T]}\iint_{\R^3\times\R^3} F(|x-y|)\varrho(t,x)\varrho(t,y)\;\dd x\,\dd y \leq C(T).  
\end{equation}
\end{tw}
\begin{rem}
Note that while the energy estimate is global in time, the constants in the two other estimates are time dependent and are finite only when $T<\infty$.
\end{rem}

Below, we explain the overall strategy of the proof and discuss the main difficulties. The starting point is to find a solution to a certain approximation of (\ref{main}). In this construction, we follow the approach of Vasseur and Yu from \cite{vasseur-yu}. 
It is a multi-level construction with many approximation parameters regularising the solutions. Additionally, we localize the interaction kernel $K_L$, and consider the problem on the torus $\T_L^3\sim [-L,L]^3$. The final step of the construction is expansion  of the torus to the whole space and recovery of the solution to the original system \eqref{main} with \eqref{Def:K}. Similar approach to derive solutions on the whole space was proposed for the system (\ref{density-dependent}) by Li and Xin in \cite{li-xin}, and for quantum isothermal fluids by Carles, Carapatoso and Hillairet in \cite{CCH}.
Our approximate system reads as follows
\eq{\label{approx}
        &\partial_t\varrho + \ddiv(\varrho u) = \varepsilon\Delta\varrho,\\
        &\partial_t(\varrho u) + \ddiv(\varrho u\otimes u) - \ddiv(\varrho\D u) + \varrho\nabla(K_L\ast\varrho) \\ 
        &= -r_0u - r_1\varrho|u|^2u + \kappa\varrho\nabla\left(\frac{\Delta\sqrt{\varrho}}{\sqrt{\varrho}}\right) -\varepsilon\nabla\varrho\cdot\nabla u -\nu\Delta^2u + \eta\nabla\varrho^{-6} + \delta\varrho\nabla\Delta^3\varrho.
}

The outline of the paper follows the strategy of proof of Theorem \ref{main_th}, described below:
\begin{enumerate}
    \item Construction of the solution to the approximated system \eqref{approx} on $(0,T)\times\T_L^3$ via the Galerkin method and and the Schauder fixed point theorem. At this point, the artificial viscosity $\varepsilon\Delta\varrho$ in the continuity equation allows us to apply classical approach for the construction.
    
    \item Derivation of the approximate version of the Bresch-Desjardins inequality (\ref{main_BD}). To this purpose, one needs to test the momentum equation by $\nabla\log\varrho$. The terms $\eta\nabla\varrho^{-6} + \delta\varrho\nabla\Delta^3\varrho$ provide that the density is strictly positive and that $\nabla\log\varrho$ is sufficiently regular in space. On the other hand, the parameter $\nu$  allows to differentiate the continuity equation and to deduce that $\nabla\log\varrho$ is also sufficiently regular in time to be used as a test function.
    
    \item Passage to the limit with $\varepsilon,\nu,\eta$ and $\delta$. Having derived the estimate (\ref{main_BD}), the improved regularity of $\varrho$ allows to pass to the limit with consecutive regularizing parameters. The proof of Theorem \ref{main_th} up to this point is pretty standard, and is only sketched in Sections \ref{ssect_mu_eps}-\ref{ssect_eta_delta}.
    
    \item Derivation of the approximate Meller-Vasseur inequality (\ref{main_MV}), uniformly with the size of the torus. This is the key step of the proof. We employ the approximating procedure introduced in \cite{vasseur-yu2}, however due to the presence of nonlocal terms, we need different arguments to close the estimates. The estimate is derived by renormalization of the momentum equation. The overall idea is to test the equation by a function $F'(|u|)\frac{u}{|u|}=(1+\ln(1+|u|^2))u$. Since it is not an admissible test function, we introduce suitable approximation of $F$, and in place of $u$ we put $v=\phi_m^0(\varrho)\phi_k^\infty(\varrho)u$, where $\phi_m^0$ and $\phi_k^\infty$ cut off the density at zero and infinity respectively. Then, passing with $m$ and $k$ to $\infty$, we derive the desired estimate in the limit. Bounding together the parameters $\kappa$ and $k$, by deriving the estimate we simultaneously pass to the limit with $\kappa$ as well. 

    The biggest challenge here is to deal with the attractive part of the kernel $K$, since on the whole space it is not integrable with any power. Because of that we are not able to follow the arguments from \cite{vasseur-yu2}. Instead, we apply the weak version of Gronwall's lemma and use generalized Young inequality for convex functions.

    \item Passage to the limit with $r_0$ and $r_1$, contained in Section \ref{r_sect}. This is the final limit passage on the torus. The main issue is that although the density-dependent viscous stress tensor provides extra regularity for the density (via the Bresch-Desjardins estimate), it gives no information on $u$ itself on the set where $\varrho=0$. Without the extra friction terms, we end up with very little regularity of the velocity. However, having the estimate (\ref{main_MV}), following the arguments from \cite{mellet-vasseur} we are able to show strong convergence of $\sqrt{\varrho}u$, which combined with compactness properties of the density allows to still perform the limit passage.
    
    \item Expansion of the torus. In the previous steps we needed to restrict our domain to the torus $\T_L^3\sim [-L,L]^3$. The last part of the proof is to pass to the limit with $L\to\infty$ and in consequence to obtain the solutions on the whole space $\R^3$. The previously derived estimates are uniform in $L$ and thus allow to extend our solution. During this limit passage we also lose the compactness of the nonlocal term. Nonetheless, the energy inequality provides the estimate on a double second moment 
    \[ \iint_{\R^{3\times 3}}|x-y|^2\varrho(t,x)\varrho(t,y)\;\dd x\,\dd y, \]
    which allows to control the behaviour of the density far from the origin, and in consequence pass to the limit in the nonlocal term as well.
    
\end{enumerate}

For the reader's convenience, Table \ref{tabelka} contains a list with all parameters, together with its short descriptions.

\begin{table}[h]
    \centering
    \begin{tabularx}{\textwidth}{|l|X|}

    \hline
      &  \\[-8pt]
      $\varepsilon$ & Necessary for the Faedo-Galerkin approximation \\[4pt]

      \hline
      &  \\[-8pt]
      $\nu$ & Provides that $\partial_t\nabla\log\varrho\in L^2(0,T;L^2)$, which is needed  for computing the Bresch-Desjardins estimate \\[15pt]
      
      \hline
       &  \\[-8pt]       
       $\eta,\delta$  & Provide that $\frac{1}{\varrho}$ is bounded and that $\nabla\log\varrho\in L^2(0,T;H^2)$, i.e. is a suitable test function (together with time regularity) \\[17pt]

       \hline
       &  \\[-8pt]
       $\kappa$ & provides that $\sqrt{\varrho}\in L^2(0,T;H^2)$ and $\nabla\varrho^{1/4}\in L^4(0,T;L^4)$, which is necessary for successful renormalization of the momentum equation \\[14pt]

       \hline
       &  \\[-8pt]
       $r_0,r_1$ & Provide improved integrability of $u$; combined with estimates coming from $\kappa$ enable to renormalize the momentum equation \\[14pt]

       \hline
       &  \\[-8pt]
       $m,k$ & The cutoff function $\phi_m^0(\varrho)$ cuts the area when $\varrho<\frac{1}{m}$ and $\phi_k^\infty(\varrho)$ when $\varrho>k$. They appear in the proof of the Mellet-Vasseur inequality and are the additional levels of approximation needed for renormalization of the momentum equation \\[29pt]

       \hline
       &  \\[-8pt]
       $L$ & indicates the size of the torus. By taking $L\to\infty$, we obtain the solution on the whole space \\[14pt]

       \hline
       
    \end{tabularx}
    \caption{Description of  parameters appearing in the paper}
    \label{tabelka}
\end{table}

\section{Fundamental level of approximation}\label{appr_sect}
The aim of this section is to construct the solution $(\varrho_L,u_L)$ of system \eqref{approx} on 
 the torus $\T_L^3\sim [-L,L]^3$. This is done by means of Galerkin approximation and the fixed point theorem.
\subsection{Truncation to periodic domain}\label{trunc_sect}
We begin by modification of  kernel $K$ in a way that allows $K\ast\varrho$ to be well-defined on the torus.

Let $\phi_L\geq 0$ be a radial, decreasing cut-off function such that $\mathrm{supp\;}\phi_L\subset B(0,L)$, $\phi_L(x)\equiv 1$ for $|x|<\frac{L}{2}$ and
\eq{\label{prop:phi}
|\nabla\phi_L|\leq \frac{C}{L}, \quad |\Delta\phi_L|\leq \frac{C}{L^2}. }
Then we simply put 
\[ K_L = K\phi_L. \]
In a similar way we  prepare the initial conditions. We put
\[ \sqrt{\varrho_{0,L}} = \sqrt{\varrho_0}\phi_L, \]
where $\phi_L$ is defined above, then periodize. In consequence we obtain the initial condition on the torus $\T_L^3$, satisfying:

\begin{lem}\label{initial}
    The function $\sqrt{\varrho_{0,L}}$ satisfies the following properties:
    \[ \|\nabla\sqrt{\varrho_{0,L}}\|_{L^2(\T_L^3)} \leq \|\nabla\sqrt{\varrho_0}\|_{L^2(\R^3)} + \frac{C}{L}\|\varrho_0\|_{L^1(\R^3)}^{1/2}, \]

    \[ \iint_{\T_L^3\times\T_L^3}\varrho_{0,L}(x)\varrho_{0,L}(y)K_L(x-y) \;\dd x\,\dd y \leq \iint_{\R^3\times\R^3}\varrho_0(x)\varrho_0(y)K(x-y) \;\dd x\,\dd y \]
    and
    \[ \varrho_{0,L} \to \varrho_0 \quad \text{in} \quad L^1(\R^3). \]
\end{lem}
\begin{proof}
    The proof follows straight from the definition of $\varrho_{0,L}$. First, we have
    \[ \nabla\sqrt{\varrho_{0,L}} = \nabla\sqrt{\varrho_0}\phi_L + \sqrt{\varrho_0}\nabla\phi_L \]
    and thus
    \[ \|\nabla\sqrt{\varrho_{0,L}}\|_{L^2(\T_L^3)} \leq \|\nabla\sqrt{\varrho_0}\|_{L^2(\R^3)} + \frac{C}{L}\|\varrho_0\|_{L^1(\R^3)}^{1/2}. \]
    The next estimate follows immediately, since
    \[ \iint_{\T_L^3\times\T_L^3}\varrho_0(x)\varrho_0(y)K(x-y)\phi_L^2(x)\phi_L^2(y)\phi_L(x-y) \;\dd x\,\dd y \leq \iint_{\R^3\times\R^3}\varrho_0(x)\varrho_0(y)K(x-y) \;\dd x\,\dd y \]
    by integrating over a larger domain and estimating $\phi_L$ by $1$.
    The convergence in $L^1$ follows immediately from the dominated convergence theorem, since $\phi_L\to 1$ pointwise.
\end{proof}
With the definition of $\varrho_{0,L}$, we can also define properly the initial conditions on $u$. Defining
\[ u_{0,L}(x)=0 \quad \text{for} \quad \varrho_{0,L}(x)=0; \quad u_{0,L}(x)=u_0(x) \quad \text{otherwise}, \]
we can periodize it in the same way as $\varrho_{0,L}$, and moreover
\[ \int_{\T_L^3}\varrho_{0,L}|u_{0,L}|^2\;\dd x \leq \int_{\R^3}\varrho_0|u_0|^2\;\dd x. \]
The full approximated system on the torus is then given by \eqref{approx} and it is supplemented with the initial condition:
\eq{u_{|_{t=0}}=u_{0,L},\quad \varrho_{|_{t=0}}=\tilde\varrho_0 := \varrho_{0,L}\ast\xi_{\bar\delta} + \frac{1}{m_1},}
where $m_1>0$, $\xi_{\bar\delta}$ is the standard mollifier on the torus, and we choose $\bar\delta$ depending on $\delta$ such that $\delta\|\nabla\Delta\tilde\varrho_0\|_{L^2(\T_L^3)}^2\to 0$. 

\subsection{The Galerkin approximation}
We solve the system (\ref{approx}) using the Galerkin approximation. Let $(e_i)_{i\in\mathbb{N}}$ be a suitable basis of $H^2(\T_L^3)$ and set $X_N:= \{e_1,\dots,e_N\}$. We put 
$u_N(t,x) = \sum_{i=1}^N \lambda_i(t)e_i(x).$
Moreover, let $S\colon C(0,T;X_N)\to C(0,T;C^k)$ be such that $\varrho=S(u)$ solves
\[ \varrho_t+\ddiv(\varrho u) -\varepsilon\Delta\varrho = 0, \quad \varrho_{|_{t=0}}=\varrho_0. \]
Then, we construct the solution by applying the Schauder fixed point theorem for the operator
\[ \mathfrak{M}^{-1}[S(u_N)](t)\left(\mathfrak{M}[\varrho_0](u_0)+\int_0^T\mathfrak{N}(S(u_N),u_N)(s)\dd s\right), \]
where $\mathfrak{M}[\varrho]\colon X_N\to X^*_N$ is given by
$\displaystyle \langle \mathfrak{M}[\varrho]u,w\rangle = \int_{\T_L^3}\varrho u\cdot w \;\dd x $
and 
\begin{multline*} \mathfrak{N}(\varrho,u) = -\ddiv(\varrho u\otimes u) + \ddiv(\varrho\D u) - \varrho\nabla(K_L\ast\varrho) \\
-r_0u - r_1\varrho|u|^2u + \kappa\varrho\nabla\left(\frac{\Delta\sqrt{\varrho}}{\sqrt{\varrho}}\right) -\varepsilon\nabla\varrho\cdot\nabla u -\nu\Delta^2u + \eta\nabla\varrho^{-6} + \delta\varrho\nabla\Delta^3\varrho. \end{multline*}
As a result, we obtain a smooth solution $(\varrho_N,u_N)$, corresponding to initial  $\tilde\varrho_0$ and
$ u_{0,N} = \sum_{i=1}^N \langle u_0,e_i\rangle. $

\subsection{Energy estimate}
Testing the momentum equation of (\ref{approx}) by $u_N$ and using the approximate continuity equation, we get that the following equality is satisfied uniformly in $N$
\eq{\label{energy0}
    \frac{\dd}{\dd t}E(\varrho_N,u_N) &+ \nu\int_{\T_L^3}|\Delta u_N|^2\dd x + \int_{\T_L^3}\varrho_N|\D u_N|^2\dd x + \varepsilon\delta\int_{\T_L^3}|\Delta^2\varrho_N|^2\dd x\\
   & + \frac{2}{3}\varepsilon\eta\int_{\T_L^3}|\nabla\varrho_N^{-3}|^2\dd x 
    + r_0\int_{\T_L^3}|u_N|^2\dd x + r_1\int_{\T_L^3}\varrho_N|u_N|^4\dd x \\
    &+ \kappa\varepsilon\int_{\T_L^3}\varrho_N|\nabla^2\log\varrho_N|^2\dd x +\varepsilon\int_{\T_L^3}\nabla(K_L\ast\varrho_N)\cdot\nabla\varrho_N \dd x = 0,
}
where
\[ E(\varrho,u) = \int_{\T_L^3}\left(\frac{1}{2}\varrho|u|^2 + \frac{1}{2}\varrho(K_L\ast\varrho) + \frac{\eta}{7}\varrho^{-6} + \frac{\kappa}{2}|\nabla\sqrt{\varrho}|^2\dd x +\frac{\delta}{2}|\nabla\Delta\varrho|^2 \right) \dd x.  \]

To deduce useful bounds from this equality, we extract certain estimates from the nonlocal term $\varepsilon\int_{\T_L^3}\nabla(K_L\ast\varrho_N)\cdot\nabla\varrho_N \dd x$, which are a consequence of the following Lemma:
\begin{lem}\label{KL_lem1}
For a sufficiently smooth $\varrho$, we have
\begin{equation}\label{oszac_KL2}
\begin{aligned} \int_{\T_L^3}\nabla(K_L\ast\varrho)\cdot\nabla\varrho \;\dd x \geq -C\|\varrho\|_{L^1(\T_L^3)}^2
\end{aligned}\end{equation}
for some $C>0$ not depending on $L$.
\end{lem}
\begin{proof}
We consider two cases, depending on $\alpha$.\\ 

\noindent{\emph{Case 1.}  If $\alpha\leq 1$, note that
\[ \int_{\T_L^3}\nabla(K_L\ast\varrho)\cdot\nabla\varrho \;\dd x = -\iint_{\T_L^3\times\T_L^3}\varrho(x)\varrho(y)\Delta(K\phi_L)(x-y)\dd x\,\dd y.
\]
We have $\Delta(K\phi_L) = \Delta K\phi_L + 2\nabla K\cdot\nabla\phi_L + K\Delta\phi_L. $
Further note that $\Delta\left(\frac{1}{2}|x|^2\right)=3$, and 
\[ \Delta\left(\frac{1}{|x|^\alpha}\right)=-\frac{\alpha(1-\alpha)}{|x|^{\alpha+2}}, \quad \text{for} \quad \alpha<1 \quad \text{and} \quad \Delta\left(\frac{1}{|x|}\right)=-4\pi\delta_0 \quad \text{for} \quad \alpha=1. \]
Putting it all together, for $\alpha < 1$ we get
%
\begin{equation}\label{oszac_KL}
\begin{aligned} \int_{\T_L^3}\nabla(K_L\ast\varrho)\cdot\nabla\varrho \;\dd x &= \alpha(1-\alpha)\iint_{\T_L^3\times\T_L^3}\varrho(x)\varrho(y)\frac{\phi_L(x-y)}{|x-y|^{\alpha+2}}\;\dd x\,\dd y \\
&- 3\iint_{\T_L^3\times\T_L^3}\varrho(x)\varrho(y)\phi_L(x-y)\,\dd x\,\dd y \\
&- 2\iint_{T_L^3\times\T_L^3}\varrho(x)\varrho(y)\left[-\alpha\frac{x-y}{|x-y|^{\alpha+2}} + x-y\right]\nabla\phi_L(x-y)\,\dd x\,\dd y \\
&- \iint_{\T_L^3\times\T_L^3}\varrho(x)\varrho(y)\left[\frac{1}{|x-y|^\alpha}+\frac{|x-y|^2}{2}\right]\Delta\phi_L(x-y)\, \dd x\,\dd y,
\end{aligned}\end{equation}
whereas if $\alpha=1$ the first term gets replaced by $\displaystyle 4\pi\int_{\T_L^3}\varrho^2\;\dd x$.
From the assumptions (\ref{prop:phi}) on $\phi_L$, we have $|\phi_L|\leq 1$, 
\[ \left|-\alpha\frac{x-y}{|x-y|^{\alpha+2}}+x-y\right||\nabla\phi_L(x-y)| \leq C\mathbbm{1}_{\{\frac{L}{2}<|x-y|<L\}}\left(\frac{1}{L^{\alpha+2}}+1\right) \leq C \]
and 
\[ \left(\frac{1}{|x-y|^\alpha}+\frac{|x-y|^2}{2}\right)|\Delta\phi_L(x-y)| \leq C\mathbbm{1}_{\{\frac{L}{2}<|x-y|<L\}}\left(\frac{1}{L^{\alpha+2}}+1\right) \leq C. \]
Applying these estimates to the last three terms in (\ref{oszac_KL}), we derive (\ref{oszac_KL2}).\\

\noindent{\emph Case 2.} In the case $\alpha>1$, we use the fact that $\mathcal{F}(\frac{\phi_L}{|x|^\alpha})$ is positive, where by $\mathcal{F}$ we denote the Fourier transform of $f$ on $\R^3$ or $\T_L^3$ respectively, i. e.
\[ \mathcal{F}(f)(\xi)=\hat{f}(\xi)=\int_{\R^3}e^{-2\pi i\xi\cdot x}f(x)\;\dd x, \;\xi\in\R^3 \quad \text{or} \quad \hat{f}(k)=\int_{\T_L^3} e^{-2\pi k\cdot x}f(x)\;\dd x, \; k\in\mathbb{Z}^3. \]
We have the following proposition.
\begin{prop}\label{fourier}
    Let $F\in L^p_{loc}(\R^3)$ be positive, such that $F(x)=f(|x|)$ with $rf(r)$ decreasing for $r>0$ and $\lim_{r\to\infty}rf(r)=0$. Then $\hat{F}$ is positive.
\end{prop}
\begin{proof}
    Using spherical coordinates, we get
    \[ \hat{F}(\xi) = \int_0^\infty\!\!\!\int_0^{2\pi}\!\!\!\int_0^\pi e^{-ir|\xi|\cos\theta}r^2f(r)\sin\theta\;\dd\theta\,\dd\varphi\,\dd r = \frac{4\pi}{|\xi|}\int_0^\infty rf(r)\sin(r|\xi|)\;\dd r, \]
    and the integral $\int_0^\infty rf(r)\sin(r|\xi|)\;\dd r$ is convergent and positive from the assumptions on $rf(r)$.
\end{proof}
The positivity of Fourier transform allows us in turn to show the positivity of the integral operator.
\begin{prop}
    If $K$ is a radially symmetric kernel with support in $[-L,L]^d$ and positive Fourier transform, then for any sufficiently regular function with period $2L$ we have
    \[ \int_{\T_L^3}f(x)\int_{\R^3}f(y)K(x-y)\;\dd x\,\dd y \geq 0. \]
\end{prop}
\begin{proof}
    The assertion follows straight from the identity
    \[ \int_{\T_L^3}f(x)g(x)\;\dd x = (f(-\cdot)\ast g)(0)=\sum_{k\in\mathbb{Z}^3}\mathcal{F}(f(-\cdot)\ast g)(k) = \sum_{k\in\mathbb{Z}^3}\hat{f}(-k)\hat{g}(k). \]
    In our case, it translates to
    \[\begin{aligned}
        \int_{\T_L^3}f(x)\int_{\R^3}f(y)K(x-y)\;\dd y\,\dd x &= \sum_{k\in\mathbb{Z}^d}\hat{f}(-k)\mathcal{F}(K\ast f)(k) \\
        &= \sum_{k\in\mathbb{Z}^d}\hat{K}(k)\hat{f}(-k)\hat{f}(k) = \sum_{k\in\mathbb{Z}^d}\hat{K}(k)|\hat{f}(k)|^2 \geq 0.
    \end{aligned}\]
\end{proof}
By virtue of Proposition \ref{fourier} and the assumptions on $\phi_L$, the kernel $\frac{\phi_L(x)}{|x|^\alpha}$ has positive Fourier transform. Then in particular
\[ \int_{\T_L^3}\nabla\varrho\cdot\nabla\left(\frac{\phi_L(\cdot)}{|\cdot|^\alpha}\ast\varrho\right) \;\dd x = \int_{\T_L^3}\nabla\varrho(x)\int_{\R^3}\nabla\varrho(y)\frac{\phi_L(x-y)}{|x-y|^\alpha}\;\dd y\;\dd x \geq 0. \]
Dealing with the quadratic part of $K_L$ in the same way as in the case $\alpha\leq 1$, we end the proof of Lemma \ref{KL_lem1}.}
\end{proof}
\begin{rem}
Note that the BD estimate holds for all $\alpha\in (0,3)$. Moreover, for $\alpha>1$ one can derive the following estimate
\begin{equation}
\begin{aligned}
\int_{\T_L^3}\nabla(K_L\ast\varrho)\cdot\nabla\varrho\;\dd x \geq & \frac{\alpha(\alpha-1)}{2}\|\varrho\|_{\dot H^{\frac{\alpha-1}{2}}(\T_L^3)}^2 -\frac{C}{L^{\alpha+2}}\|\varrho\|_{L^2(\T_L^3)}^2 -3\|\varrho\|_{L^1(\T_L^3)}^2 \\
&- C\iint_{\T_L^3\times\T_L^3}\varrho(x)\varrho(y)\mathbbm{1}_{\{(x,y): \frac{L}{2}<|x-y|<L\}} \dd x\,\dd y,
\end{aligned}\end{equation}
where
\[ \|\varrho\|_{\dot{H}^{\frac{\alpha-1}{2}}(\T_L^3)}^2 = \iint_{\T_L^3\times\T_L^3}\frac{|\varrho(x)-\varrho(y)|^2}{|x-y|^{\alpha+2}}. \]
However, in the case $\alpha\geq 2$ we are not able to close the Mellet-Vasseur estimate that appears later in the proof.
\end{rem}

Thanks to Lemma \ref{KL_lem1}, we get from \eqref{energy0} the following energy estimate:
\eq{\label{energy_ap}
    \sup_{t\in[0,T]} E(\varrho_N,u_N) &+ \nu\int_0^T\!\!\!\int_{\T_L^3}|\Delta u_N|^2\dd x\,\dd t + \int_0^T\!\!\!\int_{\T_L^3}\varrho_N|\D u_N|^2\dd x\,\dd t + \varepsilon\delta\int_0^T\!\!\!\int_{\T_L^3}|\Delta^2\varrho_N|^2\dd x\, \dd t \\
    &+ \frac{2}{3}\varepsilon\eta\int_0^T\!\!\!\int_{\T_L^3}|\nabla\varrho_N^{-3}|^2\dd x\,\dd t + r_0\int_0^T\!\!\!\int_{\T_L^3}|u_N|^2\dd x\,\dd t + r_1\int_0^T\!\!\!\int_{\T_L^3}\varrho|u_N|^4\dd x\,\dd t \\
    &+ \kappa\varepsilon\int_0^T\!\!\!\int_{\T_L^3}\varrho_N|\nabla^2\log\varrho_N|^2\dd x\,\dd t \leq E(\varrho_0,u_0) + C\varepsilon T\|\varrho_0\|_{L^1(\T_L^3)}^2.
}

As the estimates in (\ref{energy_ap}) are satisfied for any $T<\infty$, we can extend the solution to the whole interval $[0,T]$ for any $T<\infty$. From (\ref{energy_ap}) we also extract the estimates to pass to the limit with $N\to\infty$. 
To converge in the term involving negative powers of the density, we use the following version of the Sobolev inequality:
\begin{lem} We have:
    \begin{equation}\label{sobolev} \|\varrho^{-1}\|_{L^\infty(\T_L^3)} \leq C(1+\|\varrho\|_{H^3(\T_L^3)})^2(1+\|\varrho^{-1}\|_{L^6(\T_L^3)})^3. \end{equation}
\end{lem}
\begin{proof}
    We have
    \[ \nabla^2\varrho^{-1} = -\frac{1}{\varrho^2}\nabla^2\varrho + \frac{2}{\varrho^3}\nabla\varrho\otimes\nabla\varrho. \]
    Therefore
    \[\begin{aligned} \|\nabla^2\varrho^{-1}\|_{L^2(\Omega)} &\leq \|\varrho^{-2}\|_{L^3(\T_L^3)}\|\nabla^2\varrho\|_{L^6(\T_L^3)} + 2\|\nabla\varrho\|_{L^\infty(\T_L^3)}^2\|\varrho^{-3}\|_{L^2(\T_L^3)} \\
    &\leq C\|\varrho^{-1}\|_{L^6(\T_L^3)}^2\|\varrho\|_{H^3(\T_L^3)} + C\|\varrho\|_{H^3(\T_L^3)}^2\|\varrho^{-1}\|_{L^6(\T_L^3)}^3 \\
    &\leq C(1+\|\varrho\|_{H^3(\T_L^3)})^2(1+\|\varrho^{-1}\|_{L^6(\T_L^3)})^3,
    \end{aligned}\]
    which ends the proof by virtue of Sobolev embeddings.
\end{proof}
From (\ref{sobolev}) and (\ref{energy_ap}), we get that
\[ \varrho_N \geq C(\delta,\eta)>0. \]
With this argument at hand, the limit passage in the terms in \eqref{approx} is a straightforward application of Aubin--Lions lemma in suitable spaces, and thus we omit the details here. We refer to \cite{vasseur-yu} for the complete proof of convergence.

By the weak lower semicontinuity of convex functions, the limit  $(\varrho,u)$ also satisfies the energy estimate
\begin{equation}\label{energy}
\begin{aligned}
    \sup_{t\in[0,T]} E(\varrho,u) +& \nu\int_0^T\!\!\!\int_{\T_L^3}|\Delta u|^2\dd x\,\dd t + \int_0^T\!\!\!\int_{\T_L^3}\varrho|\D u|^2\dd x\,\dd t + \varepsilon\delta\int_0^T\!\!\!\int_{\T_L^3}|\Delta^2\varrho|^2\dd x\,\dd t \\
    +& \frac{2}{3}\varepsilon\eta\int_0^T\!\!\!\int_{\T_L^3}|\nabla\varrho^{-3}|^2\dd x + r_0\int_0^T\!\!\!\int_{\T_L^3}|u|^2\dd x\,\dd t + r_1\int_0^T\!\!\!\int_{\T_L^3}\varrho|u|^4\dd x\,\dd t \\
    +& \frac{\kappa\varepsilon}{2}\int_0^T\!\!\!\int_{\T_L^3}\varrho|\nabla^2\log\varrho|^2\dd x\,\dd t \leq E(\varrho_0,u_0) + C\varepsilon T\|\varrho_0\|_{L^1(\T_L^3)}^2.
    \end{aligned}
\end{equation}

\section{The BD inequality and limit passages with $\varepsilon,\nu,\delta,\eta\to 0$.}

Before we pass to the limit with the approximating parameters, we derive the Bresch--Desjardins inequality. To do that, we test the momentum equation by $\nabla\log\varrho$ and combine it with the energy inequality (\ref{energy}). In consequence, for
\[ E_{BD}(\varrho,u) = \int_{\T_L^3} \left(\frac{1}{2}\varrho\left|u+\frac{1}{\varrho}\nabla\varrho\right|^2 + \varrho(K_L\ast\varrho) + \frac{\delta}{2}|\nabla\Delta\varrho|^2 + \frac{\kappa}{2}|\nabla\sqrt{\varrho}|^2 + \frac{\eta}{7}\varrho^{-6}\right)\dd x \]
we obtain

\begin{equation}\label{BD1}
\begin{aligned}
        E_{BD} & (\varrho,u) + \frac{2}{3}\eta(1+\varepsilon)\int_0^T\!\!\!\int_{\T_L^3}|\nabla\varrho^{-3}|^2\dd x\,\dd t + \delta(1+\varepsilon)\int_0^T\!\!\!\int_{\T_L^3}|\Delta^2\varrho|^2\dd x\,\dd t \\
        &+ \frac{1}{4}\int_0^T\!\!\!\int_{\T_L^3}\varrho|\nabla u-\nabla^Tu|^2\dd x\,\dd t + \nu\int_0^T\!\!\!\int_{\T_L^3}|\Delta u|^2\dd x\,\dd t \\
        &+ r_0\int_0^T\!\!\!\int_{\T_L^3}|u|^2\;\dd x\,\dd t + r_1\int_0^T\!\!\!\int_{\T_L^3}\varrho|u|^4\;\dd x\,\dd t \\
        &+ \frac{\kappa(1+\varepsilon)}{2}\int_0^T\!\!\!\int_{\T_L^3}\varrho|\nabla^2\log\varrho|^2\dd x\,\dd t + \varepsilon\int_0^T\!\!\!\int_{\T_L^3}\frac{|\Delta\varrho|^2}{\varrho}\dd x + \int_0^T\!\!\!\int_{\T_L^3}\nabla(K_L\ast\varrho)\cdot\nabla\varrho \;\dd x \\
        \leq & E_{BD}(\varrho_0,u_0) + C\varepsilon T\|\varrho_0\|_{L^1(\T_L^3)}^2 \\
        &+ \varepsilon\int_0^T\!\!\!\int_{\T_L^3}\left(\nabla\varrho\cdot\nabla u\cdot\nabla\log\varrho + \Delta\varrho\frac{|\nabla\log\varrho|^2}{2} - \ddiv(\varrho u)\frac{1}{\varrho}\Delta\varrho\right)\dd x \\
        &- \nu\int_0^T\!\!\!\int_{\T_L^3}\Delta u\cdot\nabla\Delta\log\varrho \;\dd x - r_1\int_0^T\!\!\!\int_{\T_L^3}|u|^2u\nabla\varrho\;\dd x - r_0\int_0^T\!\!\!\int_{\T_L^3}\frac{u\cdot\nabla\varrho}{\varrho}\;\dd x \\
        =& E_{BD}(\varrho_0,u_0) + C\varepsilon T\|\varrho_0\|_{L^1(\T_L^3)}^2 + R_1+R_2+R_3+R_4+R_5+R_6.
        \end{aligned} \end{equation}
    The necessary calculations to derive (\ref{BD1}) are performed in the Appendix \ref{BD_app}.

    The terms $R_1$--$R_4$ go to $0$ as $\varepsilon,\nu\to 0$. For $R_5$, we have
\[\begin{aligned} 
R_5 =& -r_1\int_0^T\!\!\!\int_{\T_L^3}|u|^2u\cdot\nabla\varrho\;\dd x =
r_1\int_0^T\!\!\!\int_{\T_L^3}\varrho\ddiv(|u|^2u) \;\dd x \\
\leq & Cr_1\int_0^T\!\!\!\int_{\T_L^3}\varrho|u|^2|\nabla u|\;\dd x \leq Cr_1^2\int_0^T\!\!\!\int_{\T_L^3}\varrho|u|^4\dd x + \frac{1}{4}\int_0^T\!\!\!\int_{\T_L^3}\varrho|\nabla u|^2\dd x. 
\end{aligned}\]
Since 
\[ \int_0^T\!\!\!\int_{\T_L^3}\varrho|\nabla u|^2\;\dd x\,\dd t \leq \int_0^T\!\!\!\int_{\T_L^3}\varrho|\D u|^2\;\dd x + \frac{1}{2}\int_0^T\!\!\!\int_{\T_L^3}\varrho|\nabla u-\nabla^T u|^2\;\dd x\,\dd t, \]
the last term is further estimated by 
\[ \frac{1}{8}\int_0^T\!\!\!\int_{\T_L^3}\varrho|\nabla u-\nabla^Tu|^2\;\dd x\,\dd t + E(\varrho_0,u_0)+C\varepsilon T\|\varrho_0\|_{L^1(\T_L^3)}^2. \]
For $R_6$, we write
\[\begin{aligned} R_6 = -r_0\int_0^T\!\!\!\int_{\T_L^3}\frac{\ddiv(\varrho u)-\varrho\ddiv u}{\varrho}\;\dd x\,\dd t = r_0\int_0^T\!\!\!\int_{\T_L^3}\partial_t\log\varrho \;\dd x -\varepsilon r_0\int_0^T\!\!\!\int_{\T_L^3}\frac{\Delta\varrho}{\varrho}. \end{aligned}\]
Since $\varrho$ is bounded in $L^\infty(0,T;L^1)$, we have
$ r_0\int_0^T\!\!\!\int_{\T_L^3}\log_+\varrho \;\dd x \leq C. $
For the second term of $R_6$, we get
\[ \left|\varepsilon r_0\int_0^T\!\!\!\int_{\T_L^3}\frac{\Delta\varrho}{\varrho}\dd x\right| \leq CT\varepsilon r_0\|\varrho\|_{L^\infty(0,T;H^2)}\|\varrho^{-1}\|_{L^\infty([0,T]\times\T_L^3)}, \]
which also tends to $0$ as $\varepsilon\to 0$. In consequence, using additionally Lemma \ref{KL_lem1}, we get
\begin{equation}\label{BD}
    \begin{aligned}
       E_{BD}(\varrho,u) - & r_0\int_{\T_L^3}\log\varrho \;\dd x \\
        &+ \frac{2}{3}\eta(1+\varepsilon)\int_0^T\!\!\!\int_{\T_L^3}|\nabla\varrho^{-3}|^2\dd x\,\dd t + \delta(1+\varepsilon)\int_0^T\!\!\!\int_{\T_L^3}|\Delta^2\varrho|^2\dd x\,\dd t \\
        &+ \frac{1}{8}\int_0^T\!\!\!\int_{\T_L^3}\varrho|\nabla u-\nabla^Tu|^2\dd x\,\dd t + \nu\int_0^T\!\!\!\int_{\T_L^3}|\Delta u|^2\dd x\,\dd t \\
        &+ r_0\int_0^T\!\!\!\int_{\T_L^3}|u|^2\;\dd x\,\dd t + r_1\int_0^T\!\!\!\int_{\T_L^3}\varrho|u|^4\;\dd x\,\dd t \\
        &+ \frac{\kappa(1+\varepsilon)}{2}\int_0^T\!\!\!\int_{\T_L^3}\varrho|\nabla^2\log\varrho|^2\dd x\,\dd t + \varepsilon\int_{\T_L^3}\frac{|\Delta\varrho|^2}{\varrho}\dd x \\
        \leq & \sum_{i=1}^4 R_i + Cr_1^2\int_0^T\!\!\!\int_{\T_L^3}\varrho |u|^4\dd x\,\dd t +CT\varepsilon r_0\|\varrho\|_{L^\infty(0,T;H^2)}\|\varrho^{-1}\|_{L^\infty([0,T]\times\T_L^3)} \\
        &+ E_{BD}(\varrho_0,u_0) -r_0\int_{\T_L^3}\log\varrho_0\;\dd x + E(\varrho_0,u_0) + C T\|\varrho_0\|_{L^1(\T_L^3)}^2.
    \end{aligned}
\end{equation}

\subsection{Convergence lemmas}\label{conv_sect}
In this section, we present the tools which enable us to perform the limit passages in the following parts of the paper.

\begin{lem}\label{convergence1}
    Assume the sequence of functions $(\varrho_n,u_n)$ defined on $[0,T]\times\T_L^3$ satisfies
    \[ \|\partial_t\varrho_n\|_{L^2(0,T;L^{6/5})} + \|\partial_t(\varrho_n u_n)\|_{L^2(0,T;H^{-m})} \leq C \]
    for some $m\geq 1$,
    
    \begin{equation}\label{energy_pom}
        \sup_{t\in[0,T]}\int_{\T_L^3}\varrho_n|u_n|^2\;\dd x + \int_0^T\!\!\!\int_{\T_L^3} \varrho_n|\nabla u_n|^2\;\dd x\,\dd t + \int_0^T\!\!\!\int_{\T_L^3}|u_n|^2\;\dd x\,\dd t + \int_0^T\!\!\!\int_{\T_L^3}\varrho_n|u_n|^4\;\dd x\,\dd t \leq C,
    \end{equation}
    and
    \begin{equation}\label{BD_pom}
        \|\sqrt{\varrho_n}\|_{L^\infty(0,T;H^1)} \leq C
    \end{equation}
    uniformly in $n$. Then up to a subsequence
    \[\begin{aligned}
        \sqrt{\varrho_n} \rightharpoonup \sqrt{\varrho} &\quad \text{in} \quad L^2(0,T;H^1), \qquad
        u_n \rightharpoonup u &\quad \text{in} \quad L^2(0,T;L^2),
    \end{aligned}\]
    and
    \[\begin{aligned}
        \varrho_n\to\varrho &\quad \text{in} \quad C(0,T;L^{3/2}), \qquad
        \varrho_nu_n\to \varrho u &\quad \text{in} \quad L^2(0,T;L^{3/2}).
    \end{aligned}\]
    
    Moreover, if additionally $\partial_t\sqrt{\varrho_n}$ is bounded in $L^2(0,T;L^2)$ and 
    \[ \int_0^T\!\!\!\int_{\T_L^3}\varrho_n|\nabla^2\log\varrho_n|^2\;\dd x\,\dd t\leq C, \]
    then 
    \[ \sqrt{\varrho_n} \rightharpoonup \sqrt{\varrho} \quad \text{in} \quad L^2(0,T;H^2) \quad and \quad 
    \sqrt{\varrho_n} \to \sqrt{\varrho} \quad \text{in} \quad L^2(0,T;H^1). \]
\end{lem}

\begin{proof}
The weak convergence of $\sqrt{\varrho_n}$ and $u_n$ follows straight from the Banach-Alaoglu Theorem. To prove strong convergence, we use the Aubin-Lions lemma. Note that $\nabla\varrho_n = 2\sqrt{\varrho_n}\,\nabla\!\sqrt{\varrho_n} $ 
and therefore
\[ \|\nabla\varrho_n\|_{L^\infty(0,T;L^{3/2})} \leq \|\sqrt{\varrho_n}\|_{L^\infty(0,T;L^6)}\|\nabla\sqrt{\varrho_n}\|_{L^\infty(0,T;L^2)} \leq C \]
by (\ref{BD_pom}) and the Sobolev embedding. Therefore from the Aubin-Lions-Simon lemma (see e.g. \cite{Simon})
\[ \varrho_n \to \varrho \quad \text{in} \quad C(0,T;L^{3/2}). \]
Since
\[
    \nabla(\varrho_nu_n) = \nabla\varrho_n\otimes u_n + \varrho_n\nabla u_n = 2\varrho_n^{1/4}\nabla\sqrt{\varrho_n}\otimes\varrho_n^{1/4}u_n + \sqrt{\varrho_n}\sqrt{\varrho_n}\nabla u_n,
\]
we have
\[\begin{aligned} \|\nabla(\varrho_nu_n)\|_{L^2(0,T;L^{6/5})} \leq & C\|\varrho_n\|_{L^\infty(0,T;L^3)}^{1/4}\|\nabla\sqrt{\varrho_n}\|_{L^\infty(0,T;L^2)}\|\varrho_n^{1/4}u_n\|_{L^4(0,T;L^4)} \\
&+ C\|\sqrt{\varrho_n}\|_{L^\infty(0,T;L^6)}\|\sqrt{\varrho_n}\nabla u_n\|_{L^2(0,T;L^2)} \leq C. \end{aligned}\]
Moreover,
\[ \|\varrho_nu_n\|_{L^\infty(0,T;L^{3/2})} \leq \|\sqrt{\varrho_n}\|_{L^\infty(0,T;L^6)}\|\sqrt{\varrho_n}u_n\|_{L^\infty(0,T;L^2)} \leq C \]
and thus again from the Aubin-Lions lemma
\[ \varrho_nu_n\to \varrho u \quad \text{in} \quad L^2(0,T;L^{3/2}). \]

    For the second part of the lemma, the estimate on $\varrho_n|\nabla^2\log\varrho_n|^2$ in particular yields
    \[ \|\sqrt{\varrho_n}\|_{L^2(0,T;H^2)} + \|\nabla\varrho_n^{1/4}\|_{L^4(0,T;L^4)} \leq C. \]
    This is the consequence of the following Proposition, proved in \cite{jungel}:
    \begin{prop}\label{prop_jungel}
For smooth $\varrho$, we have
\[ \int_\Omega \varrho|\nabla^2\log\varrho|^2\dd x \geq \frac{1}{7}\int_\Omega |\nabla^2\sqrt{\varrho}|^2\dd x  
\mbox{ \ \ and  \ \ }
 \int_\Omega \varrho|\nabla^2\log\varrho|^2\dd x \geq \frac{1}{8}\int_\Omega |\nabla\varrho^{1/4}|^4\dd x. \]
\end{prop}
    
    Then from the estimates on $\sqrt{\varrho_n}$, again by Aubin-Lions lemma we get
    \[ \sqrt{\varrho_n} \to \sqrt{\varrho} \quad \text{in} \quad L^2(0,T;H^1). \]
\end{proof}

\begin{lem}[Limit in the nonlinear damping]\label{convergence2}
    If $\sqrt{\varrho_n}\to \sqrt{\varrho}$ in $L^2(0,T;H^1_{loc})$, $\varrho_n u_n\to \varrho u$ in $L^2(0,T;L^{3/2}_{loc})$ and $u_n\rightharpoonup u$ in $L^2(0,T;L^2)$, and additionally
    \[ \int_0^T\!\!\!\int_\Omega \varrho_n|u_n|^4\dd x\,\dd t \leq C \]
    uniformly in $n$, then
    \[ \varrho_n|u_n|^2u_n \to \varrho|u|^2u \quad \text{in} \quad L^1(0,T;L^1). \]
\end{lem}

\begin{proof}
    First, note that $\varrho_n^{1/4}u_n\rightharpoonup \varrho^{1/4}u$ in $L^4(0,T;L^4)$. Therefore, from lower semicontinuity of the norm and Fatou's lemma,
    \[ \int_0^T\!\!\!\int_\Omega \varrho|u|^4 \dd x\,\dd t \leq \int_0^T\!\!\!\int_\Omega \liminf_{n\to\infty}\varrho_n|u_n|^4\dd x\,\dd t \leq \liminf_{n\to\infty}\int_0^T\!\!\!\int_\Omega \varrho_n|u_n|^4\dd x\,\dd t \leq C.  \]
    From the strong convergence of $\sqrt{\varrho_n}$ and $\varrho_nu_n$, we know that
    \[ \varrho_n(t,x), (\varrho_nu_n)(t,x) \to \varrho(t,x), (\varrho u)(t,x) \quad \text{a.e.}, \]
    up to a subsequence. Therefore, for almost every $(t,x)$ such that $\varrho_n(t,x)\not\to 0$, we have
    \[ u_n = \frac{\varrho_nu_n}{\varrho_n} \to u. \]
 For the points where $\varrho_n \to 0$, we write
    \[ \varrho_n|T_M(u_n)|^3 \leq M^3\varrho_n \to 0 = \varrho|T_M(u)|^3, \]
   where $T_M$ is the truncation operator defined as
    \begin{equation}\label{truncation} T_M(u) = \left\{\begin{aligned}
        u, &\quad |u|\leq M, \\
        M\frac{u}{|u|}, &\quad |u|>M,
    \end{aligned}\right. \end{equation}

    Therefore from the dominated convergence theorem,
    \[ \varrho_n|T_M(u_n)|^2T_M(u_n) \to \varrho|T_M(u)|^2T_M(u) \quad \text{in} \quad L^1(0,T;L^1) \]
    for any fixed $M>0$. Moreover, we have
    \[\begin{aligned} \int_0^T\!\!\!\int_{\T_L^3} \Big|\varrho_n|u_n|^2u_n-\varrho|u|^2u\Big|\dd x\,\dd t \leq & \int_0^T\!\!\!\int_{\T_L^3} \Big|\varrho_n|T_M(u_n)|^2T_M(u_n)-\varrho|T_M(u)|^2T_M(u)\Big|\dd x\,\dd t \\
    &+ 2\int_0^T\!\!\!\int_{\T_L^3} \varrho_n|u_n|^3\mathbbm{1}_{|u_n|>M}\dd x\,\dd t + 2\int_0^T\!\!\!\int_{\T_L^3} \varrho|u|^3\mathbbm{1}_{|u|>M}\dd x\,\dd t \\
    \leq & \int_0^T\!\!\!\int_{\T_L^3} \Big|\varrho_n|u_n|^2u_n\mathbbm{1}_{|u_n|\leq M}-\varrho|u|^2u\mathbbm{1}_{|u|\leq M}\Big|\dd x\,\dd t \\
    &+ \frac{2}{M}\int_0^T\!\!\!\int_{\T_L^3} \varrho_n|u_n|^4\dd x\,\dd t + \frac{2}{M}\int_0^T\!\!\!\int_{\T_L^3} \varrho|u|^4\dd x\,\dd t.
    \end{aligned}\]
    Therefore 
    \[ \limsup_{n\to\infty} \int_0^T\!\!\!\int_{\T_L^3} \Big|\varrho_n|u_n|^2u_n-\varrho|u|^2u\Big|\dd x\,\dd t \leq \frac{C}{M}. \]
    Letting $M\to\infty$, we obtain the desired convergence.
\end{proof}

\subsection{Limit passage with $\nu,\varepsilon\to 0$}\label{ssect_mu_eps}

Now we pass to the limit with $\nu,\varepsilon \to 0$. Note that the inequality (\ref{BD}), together with the energy estimate (\ref{energy}) again provides us the estimates required in Lemmas \ref{convergence1}-\ref{convergence2} uniformly in $\varepsilon$ and $\nu$, where the estimate on $\varrho|\nabla u|^2$ comes from the estimate on the symmetric gradient $\varrho|\D u|^2$ in (\ref{energy}) and the antisymmetric part $\varrho|\nabla u-\nabla^Tu|^2$ in (\ref{BD}).

 For the terms depending on $\varepsilon$ and $\nu$, in the weak formulation we have
\[ \varepsilon\int_0^T\!\!\!\int_{\T_L^3}\nabla\varrho_{\nu,\varepsilon}\cdot\nabla\varphi \;\dd x\,\dd t \leq \sqrt{\varepsilon}\sqrt{\varepsilon}\|\nabla\varrho_{\nu,\varepsilon}\|_{L^2(0,T;L^2)}\|\nabla\varphi\|_{L^2(0,T;L^2)} \to 0, \]

\[ \varepsilon\int_0^T\!\!\!\int_{\T_L^3} \nabla\varrho_{\nu,\varepsilon} \nabla u_{\nu,\varepsilon}\varphi \;\dd x\,\dd t \leq \varepsilon\|\nabla\sqrt{\varrho_{\nu,\varepsilon}}\|_{L^\infty(0,T;L^2)}\|\sqrt{\varrho_{\nu,\varepsilon}}\nabla u_{\nu,\varepsilon}\|_{L^2(0,T;L^2)}\|\varphi\|_{L^2(0,T;L^\infty)} \to 0
 \]
 and
 \[ \nu\int_0^T\!\!\!\int_{\T_L^3} \Delta u_{\nu,\varepsilon}\Delta\varphi \dd x\,\dd t \leq \sqrt{\nu}\sqrt{\nu}\|\Delta u_{\nu,\varepsilon}\|_{L^2(0,T;L^2)}\|\Delta\varphi\|_{L^2(0,T;L^2)} \to 0. \]
 In consequence, performing limit passages  as before, we obtain the solutions to the system
\eq{\label{approx_nonlocal}
        &\partial_t\varrho + \ddiv(\varrho u) = 0, \\
        &\partial_t(\varrho u) + \ddiv(\varrho u\otimes u) - \ddiv(\varrho\D u) + \varrho\nabla(K_L\ast\varrho) \\ 
        &= -r_0u - r_1\varrho|u|^2u + \kappa\varrho\nabla\left(\frac{\Delta\sqrt{\varrho}}{\sqrt{\varrho}}\right)  + \eta\nabla\varrho^{-6} + \delta\varrho\nabla\Delta^3\varrho.
}
\noindent
Passing to the limit in (\ref{energy}) and (\ref{BD}), again by the lower semicontinuity of convex functions, we get
\begin{multline}\label{energy2}
    \sup_{t\in[0,T]} E(\varrho,u) + \int_0^T\!\!\!\int_{\T_L^3}\varrho|\D u|^2\dd x\,\dd t + r_0\int_0^T\!\!\!\int_{\T_L^3}|u|^2\dd x\,\dd t + r_1\int_0^T\!\!\!\int_{\T_L^3}\varrho|u|^4\dd x\,\dd t \\
    \leq E(\varrho_0,u_0)
\end{multline}
%
and for the Bresch--Desjardins inequality
\begin{equation}\label{BD2}
    \begin{aligned}
        E_{BD}(\varrho,u)- & r_0\int_{\T_L^3}\log\varrho\;\dd x+ \eta\int_0^T\!\!\!\int_{\T_L^3}|\nabla\varrho^{-3}|^2\;\dd x\,\dd t + 2\delta\int_0^T\!\!\!\int_{\T_L^3}|\Delta^2\varrho|^2\dd x\,\dd t \\
        &+ \frac{1}{8}\int_0^T\!\!\!\int_{\T_L^3}\varrho|\nabla u-\nabla^Tu|^2\dd x\,\dd t + \kappa\int_0^T\!\!\!\int_{\T_L^3}\varrho|\nabla^2\log\varrho|^2\dd x\,\dd t \\
        \leq & E_{BD}(\varrho_0,u_0) -r_0\int_{\T_L^3}\log\varrho_0\;\dd x + E(\varrho_0,u_0) +CT\|\varrho_0\|_{L^1(\T_L^3)}^2.
    \end{aligned}
\end{equation}

\subsection{Limit passage with $\eta,\delta\to 0$.}\label{ssect_eta_delta}
We will first pass to the limit with $\eta$ and then with $\delta$. Note that the inequalities (\ref{energy2}) and (\ref{BD2}) provide us again the estimates required in Lemmas \ref{convergence1} and \ref{convergence2}, this time uniformly in $\eta$ and $\delta$ (provided that $\delta\|\nabla\Delta\tilde\varrho_0\|_{L^2(\T_L^3)}^2\to 0$). We need to pass to the limit only with the terms $\eta\nabla\varrho^{-6}$ and $\delta\varrho\nabla\Delta^3\varrho$, since the remaining terms are treated in the same way as before. Note that since we lose the information on $\D u$ on the set where $\varrho=0$, we pass to the limit in the stress tensor using the relation
\[ \varrho\D u = \D(\varrho u) - \nabla\varrho\otimes u. \]

From (\ref{energy2}) and (\ref{BD2}) we also have the estimates
\begin{equation}
\eta\|\varrho_{\eta,\delta}^{-6}\|_{L^\infty(0,T;L^1)}, \sqrt{\eta}\|\varrho_{\eta,\delta}^{-3}\|_{L^2(0,T;H^1)} \leq C(T).
\end{equation}
In consequence, using the interpolation between $L^\infty(0,T;L^1)$ and $L^1(0,T;L^3)$, we have
\[ \|\eta\varrho_{\eta,\delta}^{-6}\|_{L^{5/3}(0,T;L^{5/3})} \leq C(T) \]
as well.

The above estimates allow us to show
\begin{lem}
    If $\varrho_{\eta,\delta}$ satisfies the estimates following from (\ref{energy2}) and (\ref{BD2}), then
    \[ \eta\int_0^T\!\!\!\int_{\T_L^3}\varrho_{\eta,\delta}^{-6}\dd x\,\dd t \to 0 \quad \text{as} \quad \eta\to 0. \]
\end{lem}
\begin{proof}
 From (\ref{BD2}) we know that
    \[ r_0\int_{\T_L^3}\log_+\left(\frac{1}{\varrho_{\eta,\delta}}\right)\;\dd x \leq C(T). \]
    As $\sqrt{\varrho_{\eta,\delta}}\to \sqrt{\varrho_\delta}$ in $L^2(0,T;H^1)$, then $\varrho_{\eta,\delta}\to \varrho_\delta$ a. e. and using the convexity of a function $y\mapsto \log_+\left(\frac{1}{y}\right)$ from Fatou's Lemma
    \[\begin{aligned} \int_{\T_L^3}\log_+\left(\frac{1}{\varrho_\delta}\right)\dd x \leq  \int_{\T_L^3}\liminf_{\eta\to 0}\log_+\left(\frac{1}{\varrho_{\eta,\delta}}\right)\dd x 
\leq  \liminf_{\eta\to 0}\int_{\T_L^3}\log_+\left(\frac{1}{\varrho_{\eta,\delta}}\right)\dd x \leq C. \end{aligned}\]
    Therefore $ |\{x: \varrho_\delta(t,x)=0\}|=0$ for almost every $t$.
    Then, as $\varrho_{\eta.\delta}\to \varrho_\delta$ a. e., we get 
    $\eta\varrho_{\eta,\delta}^{-6}\to 0$  a. e..
    As $\eta\varrho_{\eta,\delta}^{-6}$ is uniformly bounded in $L^{5/3}(0,T;L^{5/3})$, it follows that
    \[ \eta\varrho_{\eta,\delta}^{-6}\to 0 \quad \text{in} \quad L^1(0,T;L^1). \]
\end{proof}

Now we pass to the limit with $\delta$:
\begin{lem}
    For any $\varphi\in C_0^\infty([0,T]\times\T_L^3)$ we have
    \[ \delta\int_0^T\!\!\!\int_{\T_L^3}\varrho_\delta\nabla\Delta^3\varrho_\delta \varphi\;\dd x\,\dd t \to 0,\quad as\quad \delta\to 0. \]
\end{lem}
\begin{proof}
    We have
    \[ \delta\int_0^T\!\!\!\int_{\T_L^3}\varrho_\delta\nabla\Delta^3\varrho_\delta\varphi \;\dd x\,\dd t = -\delta\int_0^T\!\!\!\int_{\T_L^3}\Delta\ddiv(\varrho_\delta\varphi)\Delta^2\varrho_\delta \;\dd x\,\dd t. \]
    The inequalities (\ref{energy2}) and (\ref{BD2}) give
    \[ \sqrt{\delta}\|\varrho_\delta\|_{L^\infty(0,T;H^3)}, \sqrt{\delta}\|\varrho_\delta\|_{L^2(0,T;H^{4})} \leq C(T). \]
    Moreover, from the uniform estimate on $\nabla\sqrt{\varrho_\delta}$ in $L^\infty(0,T;L^2)$, we also have
    \[ \|\varrho_\delta\|_{L^\infty(0,T;L^3)} \leq C(T). \]
    Using the Gagliardo-Nirenberg inequality
    $ \|\nabla^3\varrho_\delta\|_{L^3}\leq C\|\nabla^{4}\varrho_\delta\|_{L^2}^{\frac{6}{7}}\|\varrho_\delta\|_{L^3}^{\frac{1}{7}},$
    we get
    \[ \delta\int_0^T\|\nabla^3\varrho_\delta\|_{L^3}^{\frac{7}{3}}\dd t \leq C\sup_{t\in[0,T]}\|\varrho_\delta\|_{L^3}^{\frac{1}{3}}\int_0^T\delta\|\nabla^{4}\varrho_\delta\|_{L^2}^2\dd t \]
    and in consequence
    $\delta^{\frac{3}{7}}\|\nabla^3\varrho_\delta\|_{L^{\frac{7}{3}}(0,T;L^3)}\leq C(T). $
    Then
    \[ \left|\delta\int_0^T\!\!\!\int_{\T_L^3}\Delta\nabla\varrho_\delta \cdot \Delta^2\varrho_\delta\varphi \;\dd x\,\dd t\right| \leq C(\varphi)\delta^{\frac{1}{14}}\|\sqrt{\delta}\nabla^{4}\varrho_\delta\|_{L^2(0,T;L^2)}\|\delta^{\frac{3}{7}}\nabla^3\varrho_\delta\|_{L^{\frac{7}{3}}(0,T;L^3)} \to 0 \]
    as $\delta\to 0$. Applying the same arguments to the rest of the terms from
    \[ \delta\int_0^T\!\!\!\int_{\T_L^3}\Delta\ddiv(\varrho_\delta\varphi)\Delta^2\varrho_\delta\dd x\,\dd t, \]
    we finish the proof of the Lemma.
\end{proof}

\begin{rem}\label{Du_rem2}
Note that in the limit passage with $\eta$ we lost any information on $\nabla u$. However, from the uniform estimates we know that up to a subsequence
\[ \sqrt{\varrho_{\eta,\delta}}\nabla u_{\eta,\delta} \rightharpoonup \overline{\sqrt{\varrho_\delta}\nabla u_\delta} \quad \text{in} \quad L^2((0,T)\times\T_L^3). \]
Using the relation
\[ \sqrt{\varrho}\nabla u = \nabla(\sqrt{\varrho} u) - \nabla\sqrt{\varrho}\otimes u, \]
from the strong convergence of $\varrho_{\eta,\delta}$ and $\nabla\sqrt{\varrho_{\eta,\delta}}$ and weak convergence of $u_{\eta,\delta}$ we get
\[ \overline{\sqrt{\varrho_\delta}\nabla u_\delta} = \nabla(\sqrt{\varrho_\delta} u_\delta) - \nabla\sqrt{\varrho_\delta}\otimes u_\delta. \]
Proceeding analogously, after passing to the limit with $\delta\to 0$ we get as well
\begin{equation}\label{Du} \overline{\sqrt{\varrho}\nabla u} = \nabla(\sqrt{\varrho} u) - \nabla\sqrt{\varrho}\otimes u. \end{equation}
In the analogous way we can also define $\overline{\sqrt{\varrho}\D u}$, $\overline{\sqrt{\varrho}\ddiv u}$ etc. In the next sections we will again omit the bars, keeping in mind the relation (\ref{Du}).
\end{rem}

\section{Mellet -- Vasseur estimates}\label{MV_sect}

Before we pass to the limit with the remaining parameters, we need to extract another estimate from the system. In the previous section, we showed the existence of a weak solution to the system
\begin{equation}\label{main2}
    \begin{aligned}
    \partial_t\varrho + \ddiv(\varrho u) &= 0,
    \\
    \partial_t(\varrho u) + \ddiv(\varrho u\otimes u) -\ddiv(\varrho\D u) + \varrho\nabla(K_L\ast\varrho) &= -r_0u - r_1\varrho|u|^2u + \kappa\varrho\nabla\left(\frac{\Delta\sqrt{\varrho}}{\sqrt{\varrho}}\right)
    \end{aligned}
\end{equation}
on $[0,T]\times\T_L^3$ with the initial conditions 
\[ \varrho_{|_{t=0}}:=\tilde\varrho_{0,L}=\varrho_{0,L}+\frac{1}{m_1}, \quad u_{|_{t=0}}=u_{0,L}, \]
where $m_1>0$ and $\varrho_{0,L},u_{0,L}$ are like in Section \ref{trunc_sect}. Similarly as in the Definition \ref{Def:main}, this means that for each $\varphi\in C_0^\infty([0,T)\times\T_L^3;\R)$ and $\psi\in C_0^\infty([0,T)\times\T_L^3;\R^3)$ it holds
\[ -\int_0^T\!\!\!\int_{\T_L^3}\varrho\partial_t\varphi \;\dd x\,\dd t - \int_0^T\!\!\!\int_{\T_L^3}\sqrt\varrho \sqrt\varrho u\cdot\nabla\varphi \;\dd x\,\dd t = \int_{\T_L^3}\tilde\varrho_{0,L}\varphi(0,\cdot)\;\dd x \]
and
\[ \begin{aligned} -\int_{\T_L^3} \tilde\varrho_{0,L}u_{0,L}\psi(0,\cdot)\;\dd x & -\int_0^T\!\!\!\int_{\T_L^3} \sqrt\varrho \sqrt\varrho u\partial_t\psi \;\dd x\,\dd t - \int_0^T\!\!\!\int_{\T_L^3}(\sqrt{\varrho}u\otimes\sqrt{\varrho}u):\nabla\psi\;\dd x\,\dd t \\
+ & \langle \varrho\D u,\nabla\psi \rangle  +\int_0^T\!\!\!\int_{\T_L^3}\varrho\nabla(K_L\ast\varrho)\cdot\psi \;\dd x \\
&= -r_0\int_0^T\!\!\!\int_{\T_L^3} u\cdot\psi \;\dd x\,\dd t -r_1\int_0^T\!\!\!\int_{\T_L^3} \varrho|u|^2u\cdot\psi\;\dd x\,\dd t \\
&\;\;\;\; -\kappa\int_0^T\!\!\!\int_{\T_L^3}\Delta\sqrt{\varrho}\sqrt{\varrho}\ddiv\psi\;\dd x\,\dd t - 2\kappa\int_0^T\!\!\!\int_{\T_L^3}\Delta\sqrt{\varrho}\nabla\sqrt{\varrho}\cdot\psi\;\dd x\,\dd t.
\end{aligned} \]
The solution satisfies the following estimates:

\begin{multline}\label{energy3}
    \sup_{t\in[0,T]} \frac{1}{2}\int_{\T_L^3}\left(\varrho|u|^2 + \varrho(K_L\ast\varrho) + \kappa|\nabla\sqrt{\varrho}|^2\right)\dd x + \int_0^T\!\!\!\int_{\T_L^3}\varrho|\D u|^2\dd x\,\dd t \\
    + r_0\int_0^T\!\!\!\int_{\T_L^3}|u|^2\dd x\,\dd t + r_1\int_0^T\!\!\!\int_{\T_L^3}\varrho|u|^4\dd x\,\dd t
    \leq E(\tilde\varrho_{0,L},u_{0,L}),
\end{multline}
where
\[ E(\tilde\varrho_{0,L},u_{0,L}) = \frac{1}{2}\int_{\T_L^3}\left(\tilde\varrho_{0,L}|u_{0,L}|^2+\tilde\varrho_{0,L}(K_L\ast\tilde\varrho_{0,L}) + \kappa|\nabla\sqrt{\tilde\varrho_{0,L}}|^2\right) \;\dd x, \]
and
\begin{multline}\label{BD3}
        \int_{\T_L^3} \left(|\nabla\sqrt{\varrho}|^2 -r_0\log\varrho \right)\;\dd x + \frac{1}{8}\int_0^T\!\!\!\int_{\T_L^3}\varrho|\nabla u-\nabla^Tu|^2\dd x + \kappa\int_0^T\!\!\!\int_{\T_L^3}\varrho|\nabla^2\log\varrho|^2\dd x \\
        \leq 2E(\tilde\varrho_{0,L},u_{0,L}) + \int_{\T_L^3} \left(|\nabla\sqrt{\tilde\varrho_{0,L}}|^2 -r_0\log\tilde\varrho_{0,L} \right)\dd x + CT\|\tilde\varrho_{0,L}\|_{L^1(\T_L^3)}^2.
\end{multline}

\noindent
From (\ref{BD3}) and Proposition \ref{prop_jungel}, it also follows that 
\begin{equation}\label{main_kappa} \kappa^{1/2}\|\sqrt{\varrho}\|_{L^2(0,T;H^2)} + \kappa^{1/4}\|\nabla\varrho^{1/4}\|_{L^4(0,T;L^4)} \leq C. \end{equation}
For the time regularity we have
\[ \|\partial_t\sqrt{\varrho}\|_{L^2(0,T;L^2)} \leq \frac{1}{2}\|\sqrt{\varrho}\ddiv u\|_{L^2(0,T;L^2)} + \frac{1}{2}\|\nabla\varrho^{1/4}\|_{L^4(0,T;L^4)}\|\varrho^{1/4}u\|_{L^4(0,T;L^4)}, \]
and since 
\[ \partial_t\varrho = -2\nabla\sqrt{\varrho}\cdot\varrho^{1/4}u\cdot\varrho^{1/4} - \sqrt{\varrho}\sqrt{\varrho}\ddiv u, \]
from (\ref{energy3}) and (\ref{BD3}) we get
\[ \|\partial_t\varrho\|_{L^2(0,T;L^{6/5})} \leq C. \]
 
In this section we perform the limit with $\kappa\to 0$ and simultaneously derive another estimate. The main result states 
\begin{lem}\label{MV_lem}
    There exists a solution to system (\ref{main2}) with $\kappa=0$, which satisfies the estimate
    \begin{multline}\label{MV} 
    \sup_{t\in[0,T]}\Bigg(\int_{\T_L^3}\varrho F(|u|) \;\dd x +  \iint_{\T_L^3\times\T_L^3} F(|x-y|)\varrho(x)\varrho(y)\;\dd x\,\dd y\Bigg) \\
    \leq C\left(\int_{\T_L^3}\varrho_0 F(|u_0|) + \iint_{\T_L^3\times\T_L^3} F(|x-y|)\varrho_0(x)\varrho_0(y)\;\dd x\,\dd y\right) + C + \frac{C}{L^2}
    \end{multline}
    for $F(z)=\frac{1+z^2}{2}\ln(1+z^2)$, where $C$ does not depend on $r_0$ and $r_1$. 
    Moreover, estimates (\ref{energy3}) and (\ref{BD3})  are valid with $\kappa =0$.
\end{lem}

\begin{proof} The strategy of the proof is based on the approach from \cite{vasseur-yu2}, however the need to incorporate the nonlocal term imposes some key differences in the method. Let
\[ F(z)=\frac{1+z^2}{2}\ln(1+z^2), \quad \psi(z)=\frac{1}{z}F'(z)=1+\ln(1+z^2) \]
for $z>0$. To get (\ref{MV}), we would like to test the momentum equation by the function $F'(|u|^2)\frac{u}{|u|}$. However, the regularity of the solution coming from the estimates (\ref{main_energy}) and (\ref{main_BD}) does not allow that (in particular, they do not provide any Sobolev regularity on $u$ itself). Instead, we will introduce a suitable approximation of the function $F'(|u|^2)\frac{u}{|u|}$, which will allow us to perform the needed renormalization of the momentum equation. Then (\ref{MV}) is obtained by passing to the limit. 

\paragraph{Preparation of initial data.} Note that our approximation of the initial data satisfies in particular
\[ \tilde\varrho_{0,L}\geq \frac{1}{m_1}. \]
In order to get the suitable continuity of $\varrho$ and $u$, we need to further truncate the initial data. Hence we will first derive the desired inequality assuming that
\begin{equation}\label{initial_modified} \tilde\varrho_{0,L}|u_{0,L}|^2\in L^\infty(\T_L^3). \end{equation}
Under this additional assumption we show that
\begin{prop} For $(\varrho,u)$ solving (\ref{main2}), we have
    \[ \varrho \in C(0,T;L^2) 
  \mbox{ \ \   and \ \ }
 \sqrt{\varrho}u \in C(0,T;L^2). \]
\end{prop}
\begin{proof}
    Since $\varrho\in L^\infty(0,T;H^1)$,
    \[ \partial_t\sqrt{\varrho}\in L^2(0,T;L^2) \quad \text{and} \quad \sqrt{\varrho}\in L^2(0,T;H^2), \]
    we have
    \[ \sqrt{\varrho}\in C(0,T;L^2) \quad \text{and} \quad \nabla\sqrt{\varrho}\in C(0,T;L^2). \]
   (the last convergence is shown by computing $\frac{\dd}{\dd t}\|\nabla\sqrt{\varrho}-\nabla\sqrt{\varrho_0}\|_{L^2}^2$).  As $\sqrt{\varrho}\in L^\infty(0,T;L^6)$, we therefore get
    \[ \sqrt{\varrho} \in C(0,T;L^p) \quad \text{for} \quad 2\leq p <6. \]
In particular, it also follows that $\varrho\in C(0,T;L^2).$
    Now we show that $\sqrt{\varrho}u\in C(0,T;L^2)$.
    By the estimates on $\partial_t(\varrho u)$ and $\varrho u$, we know that  $\varrho u\in C_{\mathrm{weak}}(0,T;L^{3/2}).$
We  estimate $|\sqrt\varrho u-\sqrt{\tilde\varrho_{0,L}}u_{0,L}|^2$ using the continuity properties of energy. Since the function $t\mapsto \int_{\T_L^3}\frac{|m(t,x)|^2}{\varrho(t,x)}\mathbbm{1}_{\{\varrho>0\}}\;\dd x$ is lower-semicontinuous (see also Lemma 7.19 in \cite{novotny-straskraba}), we have
\begin{multline*} \int_{\T_L^3} \tilde\varrho_{0,L}|u_{0,L}|^2 + \tilde\varrho_{0,L}(K_L\ast\tilde\varrho_{0,L})+\kappa|\nabla\sqrt{\tilde\varrho_{0,L}}|^2\;\dd x
\leq \liminf_{t\to 0}\int_{\T_L^3}\varrho|u|^2+\varrho(K_L\ast\varrho)+\kappa|\nabla\sqrt{\varrho}|^2\;\dd x. \end{multline*}
Combining that with the energy inequality, we get
    \begin{multline*} \lim_{t\to 0}\int_{\T_L^3} \left(\varrho|u|^2 + \varrho(K_L\ast\varrho) + \kappa|\nabla\sqrt{\varrho}|^2\right)\dd x 
    = \int_{\T_L^3} \left(\tilde\varrho_{0,L}|u_{0,L}|^2+\tilde\varrho_{,L}(K_L\ast\tilde\varrho_{0,L}) + \kappa|\nabla\sqrt{\tilde\varrho_{0,L}}|^2\right)\dd x. \end{multline*} 
    With this information at hand, we write
    \[\begin{aligned}
    \int_{\T_L^3}|\sqrt{\varrho}u-\sqrt{\tilde\varrho_{0,L}}u_{0,L}|^2\dd x =& \int_{\T_L^3} \left(\varrho|u|^2 + \varrho(K_L\ast\varrho) + \kappa|\nabla\sqrt{\varrho}|^2\right)\dd x \\
        &- \int_{\T_L^3}\left(\tilde\varrho_{0,L}|u_{0,L}|^2+\tilde\varrho_{0,L}(K_L\ast\tilde\varrho_{0,L}) + \kappa|\nabla\sqrt{\tilde\varrho_{0,L}}|^2\right)\dd x \\
        &+2 \int_{\T_L^3}\sqrt{\tilde\varrho_{0,L}}u_{0,L}(\sqrt{\tilde\varrho_{0,L}}u_{0,L}-\sqrt{\varrho}u)\;\dd x \\
        &+ \int_{\T_L^3} (\tilde\varrho_{0,L}(K_L\ast\tilde\varrho_{0,L})-\varrho(K_L\ast\varrho)) \;\dd x \\
        &- \kappa\int_{\T_L^3} |\nabla\sqrt{\tilde\varrho_{0,L}}-\nabla\sqrt{\varrho}|^2\dd x \\
        &+ 2\kappa\int_{\T_L^3} \nabla\sqrt{\tilde\varrho_{0,L}}\big(\nabla\sqrt{\tilde\varrho_{0,L}}-\nabla\sqrt{\varrho}\big)\dd x.
    \end{aligned}\]
    From the continuity of $\nabla\sqrt{\varrho}$, we have
    \[ \lim_{t\to 0}\int_{\T_L^3} |\nabla\sqrt{\tilde\varrho_{0,L}}-\nabla\sqrt{\varrho}|^2\dd x = 0 \]
    and
    \[ \lim_{t\to 0}\int_{\T_L^3} \nabla\sqrt{\tilde\varrho_{0,L}}\big(\nabla\sqrt{\tilde\varrho_{0,L}}-\sqrt{\varrho}\big)\dd x = 0. \]
    Since $K_L\in L^p(\T_L^3)$ for $p<3/\alpha$, from the continuity of $\varrho$ it follows that $K_L\ast\varrho\in C(0,T;L^q)$ for $q<\frac{3}{\alpha-2}$ (or $\infty$ if $\alpha\leq 2$) and therefore
    \[ \lim_{t\to 0}\int_{\T_L^3} (\tilde\varrho_{0,L}(K_L\ast\tilde\varrho_{0,L})-\varrho(K_L\ast\varrho)) \dd x \to 0. \]
    In consequence,
     \begin{multline}\label{calc_rho_cont} \mathrm{ess}\limsup_{t\to 0}\int_{\T_L^3} |\sqrt{\varrho}u-\sqrt{\tilde\varrho_{0,L}}u_{0,L}|^2\dd x 
     = 2\mathrm{ess}\limsup_{t\to 0}\int_{\T_L^3} \sqrt{\tilde\varrho_{0,L}}u_{0,L}(\sqrt{\tilde\varrho_{0,L}}u_{0,L}-\sqrt{\varrho}u)\;\dd x.
    \end{multline}
    Now, let $\phi_{m_1}$ be a smooth cutoff function such that $\phi_{m_1}(\varrho)=1$ for $\varrho\geq\frac{1}{m_1}$ and $\phi_{m_1}(\varrho)=0$ for $\varrho\leq\frac{1}{2m_1}$. Then we write
    \[\begin{aligned} \int_{\T_L^3} \sqrt{\tilde\varrho_{0,L}}u_{0,L}(\sqrt{\tilde\varrho_{0,L}}u_{0,L}-\sqrt{\varrho}u)\dd x =& \int_{\T_L^3}\sqrt{\tilde\varrho_{0,L}}u_{0,L}(\sqrt{\tilde\varrho_{0,L}}u_{0,L}-\phi_{m_1}(\varrho)\sqrt{\varrho}u) \dd x \\
    &- \int_{\T_L^3}\sqrt{\tilde\varrho_{0,L}}u_{0,L}(1-\phi_{m_1}(\varrho))\sqrt{\varrho}u\;\dd x \\
    =& B_1(t)+B_2(t).
    \end{aligned}\]
    For the term $B_1$, we use the relation
    \[\begin{aligned} \sqrt{\tilde\varrho_{0,L}}u_{0,L}(\sqrt{\tilde\varrho_{0,L}}u_{0,L}-\phi_{m_1}(\varrho)\sqrt{\varrho}u) =& \sqrt{\tilde\varrho_{0,L}}u_{0,L}\frac{\phi_{m_1}(\varrho)}{\sqrt{\varrho}}(\tilde\varrho_{0,L}u_{0,L}-\varrho u) \\
    &+ \tilde\varrho_{0,L}|u_{0,L}|^2\left(1-\frac{\sqrt{\tilde\varrho_{0,L}}}{\sqrt{\varrho}}\phi_{m_1}(\varrho)\right).
    \end{aligned}\]
    Since $\tilde\varrho_{0,L}\geq \frac{1}{m_1}$, we know that in particular $\phi_{m_1}(\tilde\varrho_{0,L})=1$ and then
    \[\begin{aligned} \lim_{t\to 0}B_1(t) =& \lim_{t\to 0}\int_{\T_L^3}\sqrt{\tilde\varrho_{0,L}}u_{0,L}\frac{\phi_{m_1}(\varrho)}{\sqrt{\varrho}}(\tilde\varrho_{0,L} u_{0,L}-\varrho u)\dd x \\
    &+ \lim_{t\to 0}\int_{\T_L^3}\sqrt{\tilde\varrho_{0,L}}|u_{0,L}|^2\sqrt{\tilde\varrho_{0,L}}\left(\frac{\phi_{m_1}(\tilde\varrho_{0,L}}{\sqrt{\tilde\varrho_{0,L}}}-\frac{\phi_{m_1}(\varrho)}{\sqrt{\varrho}}\right)\;\dd x = 0,
    \end{aligned}\]
    by the weak continuity of $\varrho u$ and strong continuity of $\varrho$. For $B_2$, by (\ref{initial_modified}) we have
    \[ |B_2(t)|\leq \|\sqrt{\tilde\varrho_{0,L}}u_{0,L}\|_{L^\infty(\T_L^3)}\|\sqrt{\varrho}u\|_{L^\infty(0,T;L^2)}\|1-\phi_{m_1}(\varrho)\|_{L^2(\T_L^3)}, \]
    which goes to zero as $t\to 0$ again from the strong continuity of $\varrho$.

    In consequence, 
    $\sqrt{\varrho}u \in C(0,T;L^2)$
    as we needed to prove.
\end{proof}

\paragraph{Preparation of the test function.} Let $\phi_m^0$ and $\phi_k^\infty$ be smooth cutoff functions at zero at infinity respectively, such that 
\[ \phi_m^0(\varrho) = 1 \quad \text{for} \quad \varrho>\frac{1}{m}, \quad \phi_m^0(\varrho)=0 \quad \text{for} \quad \varrho<\frac{1}{2m}, \quad |(\phi_m^0)'|\leq 2m, \]
and
\[ \phi_k^\infty(\varrho) = 1 \quad \text{for} \quad \varrho<k, \quad \phi_k^\infty(\varrho)=0, \quad \text{for} \quad \varrho>2k, \quad |(\phi_k^\infty)'|\leq \frac{2}{k}. \]
Then we define
\[ v_{m,k} = \phi_{m,k}(\varrho)u \quad \text{for} \quad \phi_{m,k}(\varrho)=\phi_m^0(\varrho)\phi_k^\infty(\varrho). \]
To simplify the notation, we will just write $v=v_{m,k}$ and $\phi=\phi_{m,k}$ when it does not raise confusion.
It turns out that $v$ has the $W^{1,2}$ regularity missing for $u$:

\begin{prop} We have
    \[ \nabla v \in L^2((0,T)\times\T_L^3). \]
\end{prop}
\begin{proof}
   By straightforward calculations,
   \[\begin{aligned} \nabla v = \phi'(\varrho)\nabla\varrho\otimes u + \phi(\varrho)\nabla u 
   = 4\phi'(\varrho)\sqrt{\varrho}\nabla\varrho^{1/4}\otimes\varrho^{1/4}u + \frac{\phi(\varrho)}{\sqrt{\varrho}}\sqrt{\varrho}\nabla u.
   \end{aligned}\]
   From the definition of $\phi$, 
   \[ \frac{\phi(\varrho)}{\sqrt{\varrho}}\leq \sqrt{2m} \quad \text{and} \quad |\phi'(\varrho)\sqrt{\varrho}|\leq \max\left(2m\cdot\frac{1}{\sqrt{m}},\frac{2}{k}\cdot\sqrt{2k}\right) \]
   and therefore
   \begin{multline*} \|\nabla v\|_{L^2((0,T)\times\T_L^3)}
   \leq C(m,k)\left(\|\nabla\varrho^{1/4}\|_{L^4((0,T)\times\T_L^3)}\|\varrho^{1/4}u\|_{L^4((0,T)\times\T_L^3)} + \|\sqrt{\varrho}\nabla u\|_{L^2((0,T)\times\T_L^3)}\right). \end{multline*}
\end{proof}

In order to construct a suitable test function, we need to approximate the functions $F$ and $\psi$ as well. Let
 \[ F_n(z) = \left\{\begin{aligned}
     \frac{1+z^2}{2}\ln(1+z^2), &\quad z\leq n, \\
     \left(nz+\frac{1-n^2}{2}\right)\ln(1+z^2), &\quad z>n
 \end{aligned}\right. \]
 and 
 \[ \psi_n(z) = \frac{1}{z}F_n'(z) = \left\{\begin{aligned}
     1+\ln(1+z^2), &\quad z\leq n, \\
     \frac{n}{z}\ln(1+z^2)+\frac{2nz+1-n^2}{1+z^2}, &\quad z>n
 \end{aligned}\right.\]
 That way 
 \[ F_n(z) \leq C_n|z|^{1+\delta} \quad 
 \text{and}
 \quad \psi_n(z)z = F_n'(z) \leq C_n|z|^{\delta}, \]
 for any $\delta\in (0,1)$ and some $C_n>0$. Since $F_n\leq \frac{1+z^2}{2}\ln(1+z^2)$ and $\psi_n\leq 1+\ln(1+z^2)$, we also have the estimates
 \[ F_n(z) \leq C+C|z|^{2+\delta}, \quad \psi_n(z)z \leq C+C|z|^{1+\delta}, \]
 where this time $C>0$ does not depend on $n$.
 
 Finally, we have the estimates for the second derivative: 
 \[ F_n''(z) = \psi_n'(z)z+\psi_n(z) = \left\{\begin{aligned}
     1+\ln(1+z^2)+\frac{2z^2}{1+z^2}, &\quad z\leq n, \\
     \frac{2nz}{1+z^2} + \frac{(n^2-1)z^2+4nz-(n^2-1)}{(1+z^2)^2}, &\quad z>n
 \end{aligned}\right.\]
 and thus it is positive and bounded. 
 
 Having the suitable approximations, we can state the first step towards the proof of Lemma \ref{MV_lem}.

\begin{lem}
    For any nonnegative $\xi(t)\in C_0^\infty(0,+\infty)$ we have
    \begin{multline}\label{renorm1}
    -\int_0^T\!\!\!\int_{\T_L^3} \xi'(t)\varrho F_n(|v|) \;\dd x\,\dd t + \int_0^T\!\!\!\int_{\T_L^3}\xi(t)\psi_n(|v|)v\cdot G \;\dd x\,\dd t \\
    + \int_0^T\!\!\!\int_{\T_L^3}\xi(t)S:\nabla(\psi_n(|v|)v) \;\dd x\,\dd t = 0, \end{multline}
    where 
    \[ S = \varrho\phi(\varrho)\left(\D u+\kappa\frac{\Delta\sqrt{\varrho}}{\sqrt{\varrho}}\mathbb{I}\right) \]
    and
    \[\begin{aligned}
        G =& \varrho^2u\phi'(\varrho)\ddiv u + \varrho\nabla\phi(\varrho)\D u + \phi(\varrho)\varrho\nabla(K_L\ast\varrho) + r_0u\phi(\varrho) \\
        &+ r_1\varrho|u|^2u\phi(\varrho) + \kappa\sqrt{\varrho}\nabla\phi(\varrho)\Delta\sqrt{\varrho} + 2\kappa\phi(\varrho)\nabla\sqrt{\varrho}\Delta\sqrt{\varrho}.
    \end{aligned}\]
\end{lem}
\begin{proof}
    Testing the momentum equation by $\phi(\varrho)\varphi$ for $\varphi\in C_0^\infty((0,T)\times\T_L^3)$, we obtain
    \begin{equation}\label{local_moment}
\begin{aligned} \partial_t(\varrho v) - \varrho u\phi'(\varrho)\partial_t\varrho + & \ddiv(\varrho u\otimes v)-\varrho u\otimes u\nabla\phi(\varrho) \\
    &-\ddiv(\phi(\varrho)\varrho\D u) +\varrho\D u\nabla\phi(\varrho) + \phi(\varrho)\varrho\nabla(K\ast\varrho) \\
    =& -r_0 v -r_1\phi(\varrho)\varrho|u|^2u \\
    &-\kappa\Big(\sqrt{\varrho}\Delta\sqrt{\varrho}\nabla\phi(\varrho)+ 2\phi(\varrho)\Delta\sqrt{\varrho}\nabla\sqrt{\varrho} - \nabla\big(\phi(\varrho)\sqrt{\varrho}\Delta\sqrt{\varrho}\big)\Big) \end{aligned}\end{equation}
    in the sense of distributions. Since
    \[ -\varrho u\phi'(\varrho)\partial_t\varrho -\varrho u\otimes u\nabla\phi(\varrho) = \varrho^2u\phi'(\varrho)\ddiv u, \]
    we can rewrite (\ref{local_moment}) as
    \begin{equation}\label{v} \partial_t(\varrho v) + \ddiv(\varrho u\otimes v) -\ddiv S + G = 0. \end{equation}

    Now, let us take $\xi\in C_0^\infty(0,+\infty)$. We test equation (\ref{v}) by $\Phi=(\xi(t)\psi_n(|v_\varepsilon|) v_\varepsilon)_\varepsilon$, where $f_\varepsilon=f\ast\eta_\varepsilon$ denotes the mollification over time and space. Note that since $\xi$ has compact support, for sufficiently small $\varepsilon$ the function $\Phi(t,\cdot)$ is well defined on $(0,\infty)$.
    Then we get
    \[ \int_0^T\!\!\!\int_{\T_L^3} \xi(t)\psi_n(|v_\varepsilon|)v_\varepsilon\cdot\big(\partial_t(\varrho v) + \ddiv(\varrho u\otimes v) -\ddiv S + G\big)_\varepsilon \;\dd x\,\dd t = 0. \]
    Let us rewrite the first two terms of the above. We have
    \[\begin{aligned} \int_0^T\!\!\!\int_{\T_L^3}\xi(t)\psi_n(|v_\varepsilon|)v_\varepsilon\cdot (\partial_t(\varrho v))_\varepsilon\;\dd x\,\dd t 
&= \int_0^T\!\!\!\int_{\T_L^3} \xi(t)\psi_n(|v_\varepsilon|)v_\varepsilon\partial_t(\varrho v_\varepsilon) \;\dd x\,\dd t + R_1 \\
    &= \int_0^T\!\!\!\int_{\T_L^3}\xi(t)\Big( \partial_t\varrho\psi_n(|v_\varepsilon|)|v_\varepsilon|^2 + \varrho\partial_t F_n(|v_\varepsilon|)\Big) \;\dd x\,\dd t + R_1,
    \end{aligned}\]
    where 
    \[ R_1 = \int_0^T\!\!\!\int_{\R^3} \xi(t)\psi_n(|v_\varepsilon|)v_\varepsilon\Big(\big(\partial_t(\varrho v)\big)_\varepsilon - \partial_t(\varrho v_\varepsilon)\Big) \;\dd x\,\dd t. \]
    Furthermore,
    \[\begin{aligned}
        \int_0^T\!\!\!\int_{\T_L^3} \xi(t)\psi_n(|v_\varepsilon|)v_\varepsilon\cdot & (\ddiv(\varrho u \otimes v))_\varepsilon \;\dd x\,\dd t \\
 &= \int_0^T\!\!\!\int_{\T_L^3} \xi(t)\psi_n(|v_\varepsilon|)v_\varepsilon\ddiv(\varrho u\otimes v_\varepsilon) \;\dd x\,\dd t + R_2 \\
        &= \int_0^T\!\!\!\int_{\T_L^3}\xi(t)\Big(-\partial_t\varrho\psi_n(|v_\varepsilon|)|v_\varepsilon|^2 + \partial_t\varrho F_n(|v_\varepsilon|)\Big) \;\dd x\,\dd t + R_2,
    \end{aligned}\]
    where
    \[ R_2 = \int_0^T\!\!\!\int_{\T_L^3}\xi(t)\psi_n(|v_\varepsilon|)v_\varepsilon\Big(\big(\ddiv(\varrho u\otimes v)\big)_\varepsilon-\ddiv(\varrho u\otimes v_\varepsilon)\Big) \;\dd x\,\dd t. \]
    In conclusion, we get
    \begin{multline*} \int_0^T\!\!\!\int_{\T_L^3} \xi(t)\partial_t(\varrho F_n(|v_\varepsilon|)) \;\dd x\,\dd t + R_1+R_2 \\
    -\int_0^T\!\!\!\int_{\T_L^3}\xi(t)\psi_n(|v_\varepsilon|)v_\varepsilon\big(\ddiv S\big)_\varepsilon \;\dd x\,\dd t + \int_0^T\!\!\!\int_{\T_L^3}\xi(t)\psi_n(|v_\varepsilon|)v_\varepsilon G_\varepsilon \;\dd x\,\dd t = 0. \end{multline*}
    Since $v\in L^2(0,T;H^1)$, $v_\varepsilon\to v$ in $L^2(0,T;H^1)$. In particular, up to a subsequence $v_\varepsilon\to v$ almost everywhere. For $1<p$ and $\delta>0$ such that $\frac{p}{3}+\frac{p(1+\delta)}{2}\leq 1$ and $p(1+\delta)<2$, by the definition of $F_n$ we have
    \[\begin{aligned}\|\xi'(t)\varrho F_n(|v_\varepsilon|)\|_{L^p((0,T)\times\T_L^3)}^p &\leq C_n\|\xi'\|_{L^\infty(0,T)}\int_0^T\!\!\!\int_{\T_L^3} \varrho^p|v_\varepsilon|^{p(1+\delta)}\;\dd x\,\dd t \\
&\leq C_n\|\xi'\|_{L^\infty(0,T)}\int_0^T\|\varrho\|_{L^3}^p\|v_\varepsilon\|_{L^2}^{p(1+\delta)} \;\dd t \\
    &\leq C(n,T)\|\varrho\|_{L^\infty(0,T;L^3)}^p\|v\|_{L^2((0,T)\times\T_L^3)}^{p(1+\delta)}. \end{aligned}\]
    Therefore $\xi'(t)\varrho F_n(|v_\varepsilon|)$ converges in $L^1((0,T)\times\T_L^3)$ and we have
    \[ \lim_{\varepsilon\to 0}\int_0^T\!\!\!\int_{\T_L^3} \xi'(t)\varrho F_n(|v_\varepsilon|) \;\dd x\,\dd t = \int_0^T\!\!\!\int_{\T_L^3} \xi'(t)\varrho F_n(|v|) \;\dd x\,\dd t. \]
    Since $G\in L^{4/3}((0,T)\times\T_L^3)$ and $\psi_n(|v_\varepsilon|)|v_\varepsilon|\leq C_n|v_\varepsilon|^\delta$, we similarly have the convergence
    \[ \lim_{\varepsilon\to 0}\int_0^T\!\!\!\int_{\T_L^3}\xi(t)\psi_n(|v_\varepsilon|)v_\varepsilon G_\varepsilon \;\dd x\,\dd t = \int_0^T\!\!\!\int_{\T_L^3}\xi(t)\psi_n(|v|)v\cdot G \;\dd x\,\dd t. \]
    Moreover, we have
    \[\begin{aligned} \int_0^T\!\!\!\int_{\T_L^3}\xi(t)\psi_n(|v_\varepsilon|)v_\varepsilon (\ddiv S)_\varepsilon \;\dd x\,\dd t &= -\int_0^T\!\!\!\int_{\T_L^3}\xi(t)S_\varepsilon:\nabla(\psi_n(|v_\varepsilon|)v_\varepsilon) \;\dd x\,\dd t \\
    &= -\int_0^T\!\!\!\int_{\T_L^3}\xi(t)(\psi_n'(|v_\varepsilon|)|v_\varepsilon|+\psi_n(|v_\varepsilon|))S_\varepsilon:\nabla v_\varepsilon \;\dd x\,\dd t.
    \end{aligned}\]
By virtue of the estimates on $\sqrt\varrho\D u$ and $\Delta\sqrt\varrho$, the function $S$ belongs to $L^2((0,T)\times\T_L^3$ and thus $S_\varepsilon\to S$ in $L^2((0,T)\times\T_L^3)$. Moreover $\nabla v_\varepsilon$ converges strongly in $L^2((0,T)\times\T_L^3)$ to $\nabla v$, and $\psi_n'(|v_\varepsilon|)|v_\varepsilon|+\psi_n(|v_\varepsilon|)$ is uniformly bounded in $L^\infty((0,T)\times\T_L^3)$. Therefore we have the convergence 
    \[ \lim_{\varepsilon\to 0}\int_0^T\!\!\!\int_{\T_L^3} \xi(t)\psi_n(|v_\varepsilon|)v_\varepsilon(\ddiv S)_\varepsilon \;\dd x\,\dd t = -\int_0^T\!\!\!\int_{\T_L^3}\xi(t)S:\nabla(\psi_n(|v|)v) \;\dd x\,\dd t. \]
    What is left is to show that $R_1,R_2\to 0$ as $\varepsilon\to 0$.
  
    To do that, we use the following commutator lemmas (see e. g. Lemma 2.3 in \cite{lions1}):
    \begin{lem}\label{comm1}
        Let $f\in (W^{1,p}(\R^d))^d$ and $g\in L^q(\R^d)$ with $\frac{1}{p}+\frac{1}{q}=\frac{1}{r}<1$. Then
        \[ \|(\ddiv(fg))_\varepsilon-\ddiv(fg_\varepsilon)\|_{L^r} \leq C\|f\|_{W^{1,p}}\|g\|_{L^q} \]
        for some $C>0$ independent of $\varepsilon$, and 
        \[ (\ddiv(fg))_\varepsilon - \ddiv(fg_\varepsilon) \to 0 \quad \text{in} \quad L^r(\R^d). \]
    \end{lem}
    Analogously with respect to time, we also have
    \begin{lem}\label{comm2}
        Let $f_t\in L^p((0,T))$ and $g\in L^q(0,T)$ with $\frac{1}{p}+\frac{1}{q}=\frac{1}{r}<1$. Then
        \[ \|(\partial_t(fg))_\varepsilon - \partial_t(fg_\varepsilon)\|_{L^r} \leq C\|f_t\|_{L^p}\|g\|_{L^q} \]
        and
        \[ (\partial_t(fg))_\varepsilon - \partial_t(fg_\varepsilon) \to 0 \quad \text{in} \quad L^r((0,T)).\]
    \end{lem}

    Now we apply the above lemmas to $R_1$ and $R_2$. By Sobolev embedding, $v\in L^2(0,T;L^6)$. For $R_2$, we have
    \[ \nabla(\varrho u)=\varrho^{1/2}\nabla\varrho^{1/4}\otimes\varrho^{1/4}u + \sqrt{\varrho}\sqrt{\varrho}\nabla u \in L^2(0,T;L^{3/2}). \]
    Therefore since $\psi_n(|v_\varepsilon|)|v_\varepsilon|\leq C_n|v_\varepsilon|^{1/3}$, from Lemma \ref{comm1} we get
    \[\begin{aligned} \Big|\int_{\T_L^3}\psi_n(|v_\varepsilon|)v_\varepsilon & \big((\ddiv(\varrho u\otimes v))_\varepsilon-\ddiv(\varrho u\otimes v_\varepsilon)\big)\dd x\Big| \\
    &\leq C_n\|v_\varepsilon\|_{L^2(\T_L^3)}^{1/3}\|(\ddiv(\varrho u\otimes v))_\varepsilon-\ddiv(\varrho u\otimes v_\varepsilon)\|_{L^{6/5}(\T_L^3)} \\
    &\leq C_n\|v\|_{L^\infty(0,T;L^2)}^{1/3}\|\nabla(\varrho u)\|_{W^{1,3/2}(\T_L^3)}\|v\|_{L^6(\T_L^3)}
    \end{aligned}\]
    and the right hand side is integrable in time (note that the $L^\infty(0,T;L^2)$ estimate on $v$ follows from the same regularity of $\sqrt{\varrho}u$). Thus from the Dominated Convergence Theorem
    \[ |R_2|\leq \|\xi\|_{L^\infty(0,T)}\int_0^T\!\!\!\int_{\T_L^3}\left|\psi_n(|v_\varepsilon|)v_\varepsilon\big((\ddiv(\varrho u\otimes v))_\varepsilon-\ddiv(\varrho u\otimes v_\varepsilon)\big)\right|\dd x\,\dd t \to 0 \]
as $\varepsilon\to 0$.
    For $R_1$, first note that
    \[ v = \varrho^{-1/4}\phi(\varrho)\varrho^{1/4}u \in L^4((0,T)\times\T_L^3). \]
    Moreover,
    \[ \partial_t\varrho = 4\sqrt{\varrho}\nabla\varrho^{1/4}\varrho^{1/4}u + \sqrt{\varrho}\sqrt{\varrho}\ddiv u \in L^{3/2}((0,T)\times\T_L^3). \]
    Then similarly as before we have
    \[\begin{aligned} \left|\int_0^T\psi_n(|v_\varepsilon|)v_\varepsilon\big((\partial_t(\varrho v))_\varepsilon-\partial_t(\varrho v_\varepsilon)\big)\dd t\right| &\leq C_n\|v_\varepsilon\|_{L^4(0,T)}^{1/3}\|(\partial_t(\varrho v))_\varepsilon-\partial_t(\varrho v_\varepsilon)\|_{L^{12/11}(0,T)} \\
    &\leq C_n\|v\|_{L^4(0,T)}^{4/3}\|\partial_t\varrho\|_{L^{3/2}(0,T)} \end{aligned}\]
    and from the Dominated Convergence Theorem
    \[ |R_1| \leq \|\xi\|_{L^\infty(0,T)}\int_0^T\!\!\!\int_{\T_L^3}\left|\psi_n(|v_\varepsilon|)v_\varepsilon\big((\partial_t(\varrho v))_\varepsilon - \partial_t(\varrho v_\varepsilon)\big)\right|\dd x\,\dd t \to 0. \]
Note that in this case we use the estimates in Lebesgue spaces on $(0,T)\times\T_L^3$ instead of Bochner spaces with different exponents over time and space, which allows us to change the order of integration.

    In consequence, when $\varepsilon\to 0$ we derive (\ref{renorm1})
 for any $\xi\in C_0^\infty(0,+\infty)$.
\end{proof}

The next step of the proof of Lemma \ref{MV_lem} is based on application of the Weak Gronwall's Lemma (Lemma \ref{weak_gronwall}). Let us rewrite (\ref{renorm1}) as
\begin{multline*} -\int_0^T\!\!\!\int_{\T_L^3} \xi'(t)\varrho F_n(|v|) \;\dd x\,\dd t 
= -\int_0^T\!\!\!\int_{\T_L^3}\xi(t)\psi_n(|v|)v\cdot\phi(\varrho)\varrho\nabla(K_L\ast\varrho) \;\dd x\,\dd t - \int_0^T\xi(t)b(t) \;\dd t, \end{multline*}
where $b(t)$ contains the rest of the terms from $G$ and $S$, i. e.
\[\begin{aligned} b(t) =& \int_{\T_L^3}\psi_n(|v|)v\cdot\Big[\varrho^2u\phi'(\varrho)\ddiv u + \varrho\nabla\phi(\varrho)\D u + \phi(\varrho)\varrho\nabla(K_L\ast\varrho) \\
&+ r_0u\phi(\varrho) + r_1\varrho|u|^2u\phi(\varrho) + \kappa\sqrt{\varrho}\nabla\phi(\varrho)\Delta\sqrt{\varrho} + 2\kappa\phi(\varrho)\nabla\sqrt{\varrho}\Delta\sqrt{\varrho}\Big]\;\dd x \\
&+ \int_{\T_L^3} \varrho\phi(\varrho)\left(\D u+\kappa\frac{\Delta\sqrt{\varrho}}{\sqrt{\varrho}}\mathbb{I}\right):\nabla(\psi_n(|v|)v)\;\dd x. \end{aligned} \]
 We first focus on the nonlocal term. From the definition of $K_L$, we have
\[\begin{aligned} |\nabla(K_L\ast\varrho)(x)| =& \left|\frac{\phi_L(\cdot)}{|\cdot|^\alpha}\ast\nabla\varrho + \frac{1}{2}\int_{\R^3}\nabla\big(|x-y|^2\phi_L(x-y)\big)\varrho(y)\;\dd y\right| \\
\leq & \left|\frac{\phi_L(\cdot)}{|\cdot|^\alpha}\ast\nabla\varrho\right| + C\int_{|x-y|<L}|x-y|\varrho(y)\;\dd y,
\end{aligned}\]
where in the last term we used (\ref{prop:phi}). Therefore we get
\[\begin{aligned} &\Bigg|\int_0^T\!\!\!\int_{\T_L^3}\xi\psi_n(|v|)v  \phi(\varrho)\varrho\nabla(K_L\ast\varrho)\;\dd x\,\dd t\Bigg| \\
\leq & \int_0^T\xi\int_{\T_L^3}\psi_n(|v|)|v|\varrho \left|\frac{\phi_L(\cdot)}{|\cdot|^\alpha}\ast\nabla\varrho\right| \;\dd x\,\dd t + C\int_0^T\xi\iint_{\T_L^{3\times 3}}\psi_n(|v(x)|)|v(x)|\varrho(x)\varrho(y)|x-y| \;\dd x\,\dd y\,\dd t \\
=& A_1 + A_2.
\end{aligned}\]
To estimate $A_1$, we use the estimates on Riesz potentials. For $f\in L^1(\T_L^3)$, let
\[ I_{3-\alpha}(f)=\int_{\T_L^3}\frac{f(y)}{|x-y|^\alpha}\;\dd y. \]
Then, in particular
\[ \|I_{3-\alpha}(f)\|_{L^{p^*}(\T_L^3)}\leq C\|f\|_{L^p(\T_L^3)} \quad \text{for} \quad p^*=\frac{3p}{3-(3-\alpha)p}. \]
From (\ref{BD3}), we have
\[ \|\nabla\varrho\|_{L^\infty(0,T;L^{3/2})} \leq 2\|\sqrt{\varrho}\|_{L^\infty(0,T;L^6)}\|\nabla\sqrt{\varrho}\|_{L^\infty(0,T;L^2)}\leq C \]
and thus
\[ \|I_{3-\alpha}(\nabla\varrho)\|_{L^\infty(0,T;L^q)} \leq C\|\nabla\varrho\|_{L^\infty(0,T;L^{3/2})} \quad \text{for} \quad q=\frac{3\cdot 3/2}{3-(3-\alpha)\cdot 3/2} = \frac{3}{\alpha-1}, \]
if $\alpha\leq 1$, then $q<\infty$.
Since $\psi_n(|v|)|v|\leq C+C|v|^{1+\delta}$, the integral in $A_1$ is estimated by
\[\begin{aligned} \int_{\T_L^3}|v|^{1+\delta}\varrho\left|\frac{\phi_L(\cdot)}{|\cdot|^\alpha}\ast\nabla\varrho\right|\;\dd x &= \int_{\T_L^3}|\sqrt{\varrho}v|^{1+\delta}\varrho^{\frac{1-\delta}{2}}|I_{3-\alpha}(\nabla\varrho)|\;\dd x \\
&\leq \|\sqrt{\varrho}v\|_{L^\infty(0,T;L^2)}^{1+\delta}\|\varrho\|_{L^\infty(0,T;L^3)}^{\frac{1-\delta}{2}}\|I_{3-\alpha}(\nabla\varrho)\|_{L^\infty(0,T;L^{\frac{3}{1-\delta}})} \\
&\leq C\|\sqrt{\varrho}v\|_{L^\infty(0,T;L^2)}^{1+\delta}\|\varrho\|_{L^\infty(0,T;L^3)}^{\frac{1-\delta}{2}}\|\nabla\varrho\|_{L^\infty(0,T;L^{3/2})},
\end{aligned}\]
provided that $\frac{3}{1-\delta}\leq \frac{3}{\alpha-1}$ (which is vaild for $\alpha<2$ and sufficiently small $\delta$). In the end, we get $|A_1|\leq C$,
where $C$ depends on the right hand sides of (\ref{energy3}) and (\ref{BD3}) (in particular it does not depend on $n,m,k,r_0,r_1$ and $\kappa$). 

For the term $A_2$, we use the following generalized Young inequality for convex functions:
\begin{equation}\label{young} ab \leq F(a)+F^*(b), \quad a,b\in\R, \end{equation}
where $F^*$ is a convex conjugate of $F$, given by
\[ F^*(s) = \sup\{sz-F(z): z\in\mathbb{R}\}. \]
The proof of (\ref{young}) is elementary, since straight from the definition of $F^*$
\[ ab-F(a) \leq \sup\{bz-F(z): \; z\in\R\}=F^*(b). \]
Applying this inequality to $A_2$, we get 

\[\begin{aligned}
    A_2 =& C\int_0^T\xi(t)\iint_{\T_L^3\times\T_L^3} F_n'(|v(x)|)|x-y|\varrho(x)\varrho(y)\;\dd x\,\dd y\dd t \\
    \leq & C\int_0^T\xi(t)\iint_{\T_L^3\times\T_L^3} F_n^*(F_n'(|v(x)|))\varrho(x)\varrho(y)\;\dd x\,\dd y\dd t \\
    &+ C\int_0^T\xi(t)\iint_{\T_L^3\times\T_L^3} F_n(|x-y|)\varrho(x)\varrho(y)\;\dd x\,\dd y\dd t.
\end{aligned}\]

To further simplify the estimate, we use the following Proposition:
\begin{prop}\label{convex_prop}
If $F\in C^1(\R)$ is strictly convex and such that 
\[ zF'(z) \leq aF(z) \]
for some $a>1$, then
\[ F^*(F'(z)) \leq (a-1)F(z). \]
\end{prop}
\begin{proof}
    Fix $s\in F'(\R)$ and let $g(z)=sz-F(z)$. Then
    \[ g'(z) = s-F'(z) \]
    and as $F'$ is increasing, $g$ attains a maximum at $z^*=(F')^{-1}(s)$. In consequence, $F^*$ is explicitly given by
    \[ F^*(s) = g(z^*) = s(F')^{-1}(s) - F((F')^{-1}(s)). \]
    Therefore
    \[ F^*(F'(z)) = F'(z)(F')^{-1}(F'(z)) - F((F')^{-1}(F'(z))) = zF'(z)-F(z) \leq (a-1)F(z), \]
    which finishes the proof.
\end{proof}
One can check that $zF_n'(z)\leq 4F_n(z)$ for sufficiently large $n$ and thus it satisfies the assumptions of Proposition \ref{convex_prop}. Therefore finally we derive
\[ A_2 \leq C\int_0^T\xi(t)\iint_{\T_L^3\times\T_L^3}\varrho F_n(|v|)\;\dd x\,\dd t + \int_0^T\xi(t)\iint_{\T_L^3\times\T_L^3} F_n(|x-y|)\varrho(x)\varrho(y)\;\dd x\,\dd y\dd t. \]

To close the estimate, we need to control the second term. To do this, we compute its derivative using the continuity equation and applying again the Young inequality. From the continuity equation, we have
\[\begin{aligned}
\frac{\dd}{\dd t}\iint_{\T_L^3\times\T_L^3} F_n(|x-y|) & \varrho(x)\varrho(y)\;\dd x\,\dd y \\
=& \iint_{\T_L^3\times\T_L^3} F_n(|x-y|)(\partial_t\varrho(x)\varrho(y)+\varrho(x)\partial_t\varrho(y))\;\dd x\,\dd y \\
=& -\iint_{\T_L^3\times\T_L^3}F_n(|x-y|)\ddiv_x(\varrho u)(x)\varrho(y)\;\dd x\,\dd y \\
=& 2\iint_{\T_L^3\times\T_L^3} F_n'(|x-y|)\frac{x-y}{|x-y|}\cdot u(x)\varrho(x)\varrho(y)\;\dd x\,\dd y \\
=& 2\iint_{\T_L^3\times\T_L^3} F_n'(|x-y|)\frac{x-y}{|x-y|}v(x)\varrho(x)\varrho(y)\;\dd x\,\dd y \\
&+ 2\iint_{\T_L^3\times\T_L^3} F_n'(|x-y|)\frac{x-y}{|x-y|}((1-\phi(\varrho))u)(x)\varrho(x)\varrho(y) \;\dd x\,\dd y. \end{aligned}\]
Therefore applying Young inequality and Proposition \ref{convex_prop}, we obtain
\[\begin{aligned}
\frac{\dd}{\dd t}\iint_{\T_L^3\times\T_L^3} F_n(|x-y|) & \varrho(x)\varrho(y)\;\dd x\,\dd y \\
\leq & 2\iint_{\T_L^3\times\T_L^3} F_n^*(F_n'(|x-y|))\varrho(x)\varrho(y)\;\dd x\,\dd y \\
&+ 2\|\varrho\|_{L^1(\T_L^3)}\int_{\R^3}\varrho F_n(|v|)\;\dd x \\
&+ 2\iint_{\T_L^3\times\T_L^3} F_n'(|x-y|)(1-\phi(\varrho(x)))|u(x)|\varrho(x)\varrho(y)\;\dd x\,\dd y \\
\leq & C\iint_{\T_L^3\times\T_L^3} F_n(|x-y|)\varrho(x)\varrho(y)\;\dd x\,\dd y + C\int_{\T_L^3}\varrho F_n(|v|) \;\dd x \\
&+ C_n\iint_{\T_L^3\times\T_L^3}|x-y|^\delta(1-\phi(\varrho(x)))|u(x)|\varrho(x)\varrho(y)\;\dd x\,\dd y.
\end{aligned} \]

In consequence we obtain
\begin{multline*} -\int_0^T\xi'(t)f(t)\;\dd t \leq C\int_0^T\xi(t)f(t)\;\dd t \\
+ \int_0^T\xi(t)\left(-b(t)+ C+C_n\iint_{\T_L^3\times\T_L^3}|u(x)||x-y|^\delta(1-\phi(\varrho(x)))\varrho(x)\varrho(y)\;\dd x\,\dd y\right), \end{multline*}
where
\[ f(t) = \int_{\T_L^3} \varrho F_n(|v|)\;\dd x + \iint_{\T_L^3\times\T_L^3}F_n(|x-y|)\varrho(x)\varrho(y)\;\dd x\,\dd y \]
and
\begin{equation}\label{b}
\begin{aligned}
    b(t) &= \int_{\T_L^3} \psi_n(|v|)v\cdot\Big(\varrho^2u\phi'(\varrho)\ddiv u + \varrho\nabla\phi(\varrho)\D u + r_0u\phi(\varrho) + r_1\varrho|u|^2u\phi(\varrho) \\
    &+ \kappa\sqrt{\varrho}\nabla\phi(\varrho)\Delta\sqrt{\varrho} + 2\kappa\phi(\varrho)\nabla\sqrt{\varrho}\Delta\sqrt{\varrho}\Big) \;\dd x + \int_{\T_L^3} \varrho\phi(\varrho)\left(\D u+\kappa\frac{\Delta\sqrt{\varrho}}{\sqrt{\varrho}}\mathbb{I}\right):\nabla(\psi_n(|v|)v)\;\dd x \\
&= J_1(t) +J_2(t).
\end{aligned}\end{equation}

Now applying weak Gronwall's lemma (Lemma \ref{weak_gronwall}) and using the continuity in time of $\sqrt{\varrho}$ and $\sqrt{\varrho}u$, we get for a. e. $t\in [0,T]$

\begin{equation}\label{MV_appr}
\begin{aligned} \int_{\T_L^3} \varrho F_n(|v|)\;\dd x +& \iint_{\T_L^3\times\T_L^3}F_n(|x-y|)\varrho(x)\varrho(y)\;\dd x\,\dd y \\
\leq & e^{CT}\left(\int_{\T_L^3}\varrho_0 F_n(|v_0|)\;\dd x + \iint_{\T_L^3\times\T_L^3}F_n(|x-y|)\varrho_0(x)\varrho_0(y)\;\dd x\,\dd y\right) \\
&- e^{CT}\int_0^T b(t)\dd t + CTe^{CT} \\ 
&+ C_n e^{CT}\int_0^T\iint_{\T_L^3\times\T_L^3}|u(x)||x-y|^\delta(1-\phi(\varrho(x)))\varrho(x)\varrho(y)\;\dd x\,\dd y\dd t, \end{aligned}\end{equation}
where the constant $C$ depends on $L$, but does not depend on $n,m,k,\kappa, r_0$ and $r_1$.

\subsection{Limit passage with $m\to\infty$}
We now pass to the limit with $m$ in (\ref{MV_appr}), i. e. remove the truncation of $\varrho$ at zero. 
Obviously
\[ v_m=\phi_m^0(\varrho)\phi_k^\infty(\varrho)u \to \phi_k^\infty(\varrho)u \quad \text{a. e.} \]
and $|v_m|\leq |u|$. 
Since $\varrho F_n(|u|)\leq C_n\varrho|u|^{1+\delta}$ is integrable, from the dominated convergence theorem we have
\[ \int_{\T_L^3}\varrho F_n(|v_m|)\;\dd x\to \int_{\T_L^3}\varrho F_n(\phi_k^\infty(\varrho)|u|) \;\dd x \]
as $m\to \infty$ and similarly
\[ \int_{\T_L^3}\varrho_0 F_n(|v_{0,m}|)\;\dd x \to \int_{\T_L^3}\varrho_0 F_n(\phi_k^\infty(\varrho_0)|u_0|) \;\dd x. \]
Since for $\delta<1$ the term $|u(x)||x-y|^\delta\varrho(x)\varrho(y)$ is integrable on $[0,T]\times\T_L^3\times\T_L^3$ by virtue of energy estimate (\ref{energy3}), the last term in (\ref{MV_appr}) converges to 
\[ \int_0^T\iint_{\T_L^3\times\T_L^3}|u(x)||x-y|^\delta(1-\phi_k^\infty(\varrho))\varrho(x)\varrho(y)\;\dd x\,\dd y\dd t. \]
Now we deal with the terms $J_1$ and $J_2$ of $b(t)$ (as in (\ref{b})). The convergence in all the terms will be a consequence of the following Proposition:
\begin{prop}\label{m_conv}
    If $\|a_m\|_{L^\infty((0,T)\times\Omega)}\leq C$, $a_m\to a$ a. e. and $f\in L^1((0,T)\times\Omega)$, then
    \[ \int_0^T\!\!\!\int_{\Omega}\phi_m^0(\varrho)a_m f \;\dd x\,\dd t \to \int_0^T\!\!\!\int_{\Omega} af\;\dd x\,\dd t \mbox{ \ \ 
    and \ \ }
     \int_0^T\!\!\!\int_{\Omega}|\varrho(\phi_m^0)'(\varrho)a_m f|\;\dd x\,\dd t\to 0 \]
    as $m\to \infty$.
\end{prop}
\begin{proof}
   Note that $\phi_m^0(\varrho)\to 1$ a. e. as $m\to\infty$. Since $|\phi_m^0(\varrho)f-f|\leq 2|f|$ and $|a_mf-af|\leq |f|(\|a_m\|_{L^\infty}+\|a\|_{L^\infty})$, by Dominated Convergence Theorem, 
    \[\int_0^T\!\!\!\int_\Omega |\phi_m^0(\varrho)f-f|\;\dd x\,\dd t \to 0 \quad \text{and} \quad \int_0^T\!\!\!\int_\Omega |a_mf-af|\;\dd x\,\dd t\to 0. \]
    Therefore
    \begin{multline*}\left|\int_0^T\!\!\!\int_\Omega \phi_m^0(\varrho)a_mf\;\dd x\,\dd t - \int_0^T\!\!\!\int_\Omega af\;\dd x\,\dd t\right| \\
    \leq \|a_m\|_{L^\infty}\int_0^T\!\!\!\int_\Omega |\phi_m^0(\varrho)f-f|\;\dd x\,\dd t + \int_0^T\!\!\!\int_\Omega|a_mf-af|\;\dd x\,\dd t \to 0. \end{multline*}
    
    For the second part of the Proposition, it is enough to notice that $|(\varrho\phi_m^0)'(\varrho)|\leq C$ and $(\phi_m^0)'(\varrho)\to 0$ a. e. Then again from the dominated convergence theorem,
    
    \[ \int_0^T\!\!\!\int_\Omega |\varrho(\phi_m^0)'(\varrho)a_mf|\;\dd x\,\dd t \to 0. \]
\end{proof}

We apply the above Proposition to each of the terms in $b(t)$. First, note that 
\[\begin{aligned} \nabla(\psi_n(|v_m|)v_m) &= \frac{\psi_n'(|v_m|)}{|v_m|}v_m\otimes v_m\nabla v_m + \psi_n(|v_m|)\nabla v_m \\
&= F_n''(|v_m|)\Big(\nabla\phi_m^0(\varrho)\otimes(\phi_k^\infty(\varrho)u_m) + \phi_m^0(\varrho)\big(\nabla\phi_k^\infty(\varrho)\otimes u_m + \phi_k^\infty(\varrho)\nabla u_m\big)\Big) \end{aligned}\]
and therefore
\begin{multline*} \int_0^T J_2(t)\;\dd t = \int_0^T\!\!\!\int_{\R^3} \varrho\phi(\varrho)\left(\D u+\kappa\frac{\Delta\sqrt{\varrho}}{\sqrt{\varrho}}\mathbb{I}\right):\nabla(\psi_n(|v_m|)v_m)\;\dd x\,\dd t \\
= \int_0^T\!\!\!\int_{\T_L^3} \phi_m^0(\varrho)a_mf_1\;\dd x\,\dd t + \int_0^T\!\!\!\int_{\T_L^3} \varrho(\phi_m^0)'(\varrho)a_mf_2 \;\dd x\,\dd t, \end{multline*}
where 
$a_m = \phi_m^0(\varrho)F_n''(|v_m|)$
and
\[\begin{aligned}
    f_1 &= \varrho\phi_k^\infty(\varrho)\left(\D u+\kappa\frac{\Delta\sqrt{\varrho}}{\sqrt{\varrho}}\mathbb{I}\right): \big(\nabla\phi_k^\infty(\varrho)\otimes u + \phi_k^\infty(\varrho)\nabla u\big), \\
    f_2 &= \phi_k^\infty(\varrho)\left(\D u+\kappa\frac{\Delta\sqrt{\varrho}}{\sqrt{\varrho}}\mathbb{I}\right):(\nabla\varrho\otimes(\phi_k^\infty(\varrho)u)).
\end{aligned} \]
By virtue of Proposition \ref{m_conv}, 
\begin{multline*} \lim_{m\to\infty}\int_0^T\!\!\!\int_{\T_L^3} \varrho\phi_{m,k}(\varrho)\left(\D u+\kappa\frac{\Delta\sqrt{\varrho}}{\sqrt{\varrho}}\mathbb{I}\right):\nabla(\psi_n(|v_m|)v_m)\;\dd x\,\dd t \\
= \int_0^T\!\!\!\int_{\T_L^3} \varrho\phi_k^\infty(\varrho)\left(\D u+\kappa\frac{\Delta\sqrt{\varrho}}{\sqrt{\varrho}}\mathbb{I}\right):\nabla(\psi_n(|\phi_k^\infty(\varrho)u|)\phi_k^\infty(\varrho)u)\;\dd x\,\dd t. \end{multline*}
Now we deal with $J_1$. From (\ref{b}), we see that
\[\begin{aligned}
J_1 =& \int_{\T_L^3} \psi_n(|v_m|)v_m\cdot\Big(\varrho^2u\phi'(\varrho)\ddiv u + \varrho\nabla\phi(\varrho)\D u \\
&+ r_0u\phi(\varrho) + r_1\varrho|u|^2u\phi(\varrho) + \kappa\sqrt{\varrho}\nabla\phi(\varrho)\Delta\sqrt{\varrho} + 2\kappa\phi(\varrho)\nabla\sqrt{\varrho}\Delta\sqrt{\varrho}\Big) \;\dd x \\
=&  \int_{\T_L^3} \psi_n(|v_m|)\phi_m^0(\varrho)\cdot\phi_k^\infty(\varrho)u\cdot\Big(\varrho^2u\phi'(\varrho)\ddiv u + \varrho\nabla\phi(\varrho)\D u \\
&+ r_0u\phi(\varrho) + r_1\varrho|u|^2u\phi(\varrho) + \kappa\sqrt{\varrho}\nabla\phi(\varrho)\Delta\sqrt{\varrho} + 2\kappa\phi(\varrho)\nabla\sqrt{\varrho}\Delta\sqrt{\varrho}\Big) \;\dd x.
\end{aligned} \]
We  group all the terms in $J_1$ with respect to $\phi_m^0(\varrho)$ and $\varrho(\phi_m^0)'(\varrho)$. Let
\[ b_m = \phi_m^0(\varrho)\psi_n(|v_m|), \quad
    g_1 = \phi_k^\infty(\varrho)^2\left(\varrho|u|^2\ddiv u + u\cdot\D u\cdot\nabla\varrho + \kappa u \cdot\frac{1}{\sqrt{\varrho}}\Delta\sqrt{\varrho}\right) \]
and
\[\begin{aligned}
    g_2 =& \phi_k^\infty(\varrho)(\phi_k^\infty)'(\varrho)\varrho^2|u|^2\ddiv u + \phi_k^\infty(\varrho)(\phi_k^\infty)'(\varrho)\varrho u\cdot\D u\cdot\nabla\varrho + r_0\phi_k^\infty(\varrho)^2|u|^2 \\
&+ r_1\phi_k^\infty(\varrho)^2\varrho|u|^4 + \kappa\phi_k^\infty(\varrho)u\cdot(\phi_k^\infty)'(\varrho)\sqrt{\varrho}\nabla\varrho\Delta\sqrt{\varrho} + 2\kappa\phi_k^\infty(\varrho)^2u\cdot\nabla\sqrt{\varrho}\Delta\sqrt{\varrho}.
\end{aligned}\]
Then
\[\begin{aligned}
  \int_0^T J_1(t)\;\dd t
    =& \int_0^T\!\!\!\int_{\T_L^3}\varrho(\phi_m^0)'(\varrho)b_mg_1 \;\dd x\,\dd t + \int_0^T\!\!\!\int_{\T_L^3}\phi_m^0(\varrho)b_mg_2 \;\dd x\,\dd t
\end{aligned}\]
and therefore from Proposition \ref{m_conv}
\[ \begin{aligned}
\lim_{m\to \infty}\int_0^T J_1(t)\;\dd t =& \int_0^T\!\!\!\int_{\T_L^3} \psi_n(|\phi_k^\infty(\varrho)u|)\phi_k^\infty(\varrho)u\cdot\Big(\varrho^2u(\phi_k^\infty)'(\varrho)\ddiv u + \varrho\nabla\phi_k^\infty(\varrho)\D u \\
    &+ r_0u\phi_k^\infty(\varrho) + r_1\varrho|u|^2u\phi_k^\infty(\varrho) \\
&+ \kappa\sqrt{\varrho}\nabla\phi_k^\infty(\varrho)\Delta\sqrt{\varrho} + 2\kappa\phi_k^\infty(\varrho)\nabla\sqrt{\varrho}\Delta\sqrt{\varrho}\Big) \;\dd x\,\dd t.
\end{aligned} \]
Combining $J_1$ and $J_2$, in the end we get
\[\begin{aligned}
    \lim_{m\to\infty}\int_0^T b(t)\dd t =& \int_0^T\!\!\!\int_{\T_L^3} \psi_n(|\phi_k^\infty(\varrho)u|)\phi_k^\infty(\varrho)u\cdot\Big(\varrho^2u(\phi_k^\infty)'(\varrho)\ddiv u + \varrho\nabla\phi_k^\infty(\varrho)\D u \\
    &+ r_0\phi_k^\infty(\varrho)u + r_1\phi_k^\infty(\varrho)\varrho|u|^2u \\
&+ \kappa\sqrt{\varrho}\nabla\phi_k^\infty(\varrho)\Delta\sqrt{\varrho} + 2\kappa\phi_k^\infty(\varrho)\nabla\sqrt{\varrho}\Delta\sqrt{\varrho}\Big) \;\dd x\,\dd t \\
    &+ \int_0^T\!\!\!\int_{\T_L^3} \varrho\phi_k^\infty(\varrho)\left(\D u+\kappa\frac{\Delta\sqrt{\varrho}}{\sqrt{\varrho}}\mathbb{I}\right):\nabla(\psi_n(|\phi_k^\infty(\varrho)u|)\phi_k^\infty(\varrho)u)\;\dd x\, \dd t.
\end{aligned}\]
Since 
\[ \int_0^T\!\!\!\int_{\T_L^3} \psi_n(|\phi_k^\infty(\varrho)u|)(\phi_k^\infty)^2(\varrho)\left(r_0|u|^2+r_1\varrho|u|^4\right) \;\dd x\,\dd t \geq 0, \]
after taking $m\to\infty$ we finally obtain the following estimate:
\begin{equation}\label{after_m}
\begin{aligned} \int_{\T_L^3} \varrho F_n(|v_k|) & \;\dd x + \iint_{\T_L^3\times\T_L^3}F_n(|x-y|)\varrho(x)\varrho(y)\;\dd x\,\dd y \\
\leq & e^{CT}\left(\int_{\T_L^3}\varrho_0 F_n(|\phi_k^\infty(\varrho_0)u_0|)\;\dd x + \iint_{\T_L^{3\times 3}}F_n(|x-y|)\varrho_0(x)\varrho_0(y)\;\dd x\,\dd y\right) \\
&- e^{CT}\int_0^T b(t)\dd t + CTe^{CT} \\ 
&+ C_n e^{CT}\int_0^T\iint_{\T_L^{3\times 3}}|u(x)||x-y|^\delta(1-\phi_k^\infty(\varrho(x)))\varrho(x)\varrho(y)\;\dd x\,\dd y\dd t \end{aligned}\end{equation}
for $v_k=\phi_k^\infty(\varrho)u$, where
\begin{equation}\label{b2}
\begin{aligned}
    b(t) =& \int_{\T_L^3} \psi_n(|v_k|)v_k\cdot\Big(\varrho^2u(\phi_k^\infty)'(\varrho)\ddiv u + \varrho\nabla\phi_k^\infty(\varrho)\D u \\
&+ \kappa\sqrt{\varrho}\nabla\phi_k^\infty(\varrho)\Delta\sqrt{\varrho} + 2\kappa\phi_k^\infty(\varrho)\nabla\sqrt{\varrho}\Delta\sqrt{\varrho}\Big) \;\dd x \\
    &+ \int_{\T_L^3} \varrho\phi_k^\infty(\varrho)\left(\D u+\kappa\frac{\Delta\sqrt{\varrho}}{\sqrt{\varrho}}\mathbb{I}\right):\nabla(\psi_n(|v_k|)v_k)\;\dd x \\
&= \tilde J_1+\tilde J_2.
\end{aligned}\end{equation}

\subsection{Limit passage with $\kappa\to 0$ and $k\to \infty$}

We choose the parameters $\kappa$ and $k$, so that we can pass to the limit in (\ref{after_m}) and (\ref{main2}) with both of them at the same time. Fix $\delta\leq 2/3$ and let $k=\kappa^{-2/\delta}$. We have the following lemma:
\begin{lem}
    If $(\varrho_\kappa,u_\kappa)$ is the solution to (\ref{main2}) and $v_\kappa=\phi_k^\infty(\varrho_\kappa)u_\kappa$, then $(\varrho_\kappa,u_\kappa)$ converges to a solution $(\varrho,u)$ to (\ref{main2}) with $\kappa=0$, and the limit satisfies for a. e. $\in[0,T]$
   \begin{multline}\label{MV_n}
        \int_{\T_L^3}\varrho F_n(|u|)\;\dd x + \iint_{\T_L^3\times\T_L^3} F_n(|x-y|)\varrho(x)\varrho(y)\;\dd x\,\dd y \\
        \leq C\left(\int_{\T_L^3}\varrho_0 F_n(|u_0|)\;\dd x + \iint_{\T_L^3\times\T_L^3} F_n(|x-y|)\varrho_0(x)\varrho_0(y)\;\dd x\,\dd y\right) + C,
    \end{multline}
    where $C$ depends on $T$ and on the right hand sides of (\ref{energy3}) and (\ref{BD3}).
\end{lem}
\begin{proof}

Since $\|\partial_t(\varrho_\kappa,u_\kappa)\|_{L^2(0,T;W^{-1,4})}\leq C$, by Lemma \ref{convergence1}, we have
\[ \varrho_\kappa\to \varrho \quad \text{in} \quad \text{in} \quad C(0,T;L^{3/2}) 
\text{ \ and \ }
 \varrho_\kappa u_\kappa\to \varrho u \quad \text{in} \quad L^2(0,T;L^{3/2}). \]
In consequence
    \[ \varrho_\kappa u_\kappa \to \varrho u \quad \text{a. e.} \]
    and for a. e. $(t,x)$ such that $\varrho(t,x)\neq 0$, we get
    \[ u_\kappa = \frac{\varrho_\kappa u_\kappa}{\varrho_\kappa} \to u. \]
    Therefore $v_\kappa(t,x) \to u(t,x)$
    as well. On the other hand, for a. e. $(t,x)$ where $\varrho(t,x)=0$ we have
    \[ \varrho_\kappa F_n(|v_\kappa|) \leq C_n \varrho_\kappa^{1-\delta}|\varrho_\kappa u_\kappa|^{\delta} \to 0 \]
    as $\kappa\to 0$. In consequence $\varrho_\kappa F_n(|v_\kappa|)\to \varrho F_n(|u|)$ a. e. Then Fatou's Lemma yields
    \[ \int_{\T_L^3}\varrho F_n(|u|)\;\dd x \leq \liminf_{\kappa\to 0}\int_{\T_L^3}\varrho_\kappa F_n(|v_\kappa|)\;\dd x. \]

    Let us pass to the limit with the terms on the right hand side of (\ref{after_m}) one by one. Similarly as before, 
    \[ \int_{\T_L^3} \varrho_0 F_n(|\phi_k^\infty(\varrho_0)u_0|) \;\dd x \to \int_{\T_L^3} \varrho_0 F_n(|u_0|) \;\dd x. \]
    Since, similarly as in the previous limit passage,
    \[ |u_\kappa(x)||x-y|^\delta(1-\phi_k^\infty(\varrho_\kappa(x)))\varrho_\kappa(x)\varrho_\kappa(y) \]
    is uniformly bounded in $L^p((0,T)\times\T_L^3\times\T_L^3)$ for some $p>1$, and convergent to $0$ a. e., we have
    \[ \lim_{\kappa\to 0}\int_0^T\iint_{\T_L^3\times\T_L^3}|u_\kappa(x)||x-y|^\delta(1-\phi_k^\infty(\varrho_\kappa(x)))\varrho_\kappa(x)\varrho_\kappa(y)\;\dd x\,\dd y\dd t = 0. \]
What is left is to estimate the terms $\tilde J_1,\tilde J_2$ defined in (\ref{b2}).
    For $\tilde J_1$, we respectively have the following bounds:

    \[\begin{aligned}
        \Big|\int_0^T\!\!\!\int_{\T_L^3} \psi_n(|\phi_k^\infty & (\varrho_\kappa) u_\kappa|)\phi_k^\infty(\varrho_\kappa)(\phi_k^\infty)'(\varrho_\kappa)\varrho_\kappa^2|u_\kappa|^2\ddiv u_\kappa \;\dd x\,\dd t\Big| \\
        &\leq C_n k^{-\delta/4}\int_0^T\!\!\!\int_{\T_L^3}\varrho_\kappa^{1/4}|\varrho^{1/4}u_\kappa|^{1+\delta}\sqrt{\varrho_\kappa}\ddiv u_\kappa \;\dd x\,\dd t \\
        &\leq C(n,T)k^{-\delta/4}\|\varrho_\kappa\|_{L^\infty(0,T;L^{\frac{1}{1-\delta}})}^{1/4}\|\varrho_\kappa^{1/4}u_\kappa\|_{L^4(0,T;L^4)}^{1+\delta}\|\sqrt{\varrho_\kappa}\ddiv u_\kappa\|_{L^2(0,T;L^2)},
    \end{aligned}\]

    \[\begin{aligned}
        &\Big|\int_0^T \int_{\T_L^3}\psi_n(|\phi_k^\infty(\varrho_\kappa)u_\kappa|)\phi_k^\infty(\varrho_\kappa)u_\kappa\cdot\varrho_\kappa\nabla\phi_k^\infty(\varrho_\kappa)\D u_\kappa \;\dd x\,\dd t\Big| \\
        &\leq C_n k^{-\delta/4}\kappa^{-1/4}\int_0^T\!\!\!\int_{\T_L^3}\varrho^{1/4}|\varrho_\kappa^{1/4}u_\kappa|^\delta|\kappa^{1/4}\nabla\varrho_\kappa^{1/4}||\sqrt{\varrho_\kappa}\D u_\kappa|\;\dd x\,\dd t \\
        &\leq C(n,T) \kappa^{1/4}\|\varrho_\kappa\|_{L^\infty(0,T;L^{\frac{1}{1-\delta}})}^{1/4}\|\varrho_\kappa^{1/4}u_\kappa\|_{L^4(0,T;L^4)}^\delta\|\kappa^{1/4}\nabla\varrho_\kappa^{1/4}\|_{L^4(0,T;L^4)}\|\sqrt{\varrho_\kappa}\D u_\kappa\|_{L^2(0,T;L^2)},
    \end{aligned}\]
    
    \[\begin{aligned}
        \kappa\Big|\int_0^T \int_{\T_L^3} \psi_n(|\phi_k^\infty(\varrho_\kappa)u_\kappa|)\phi_k^\infty(\varrho_\kappa) & u_\kappa\cdot \sqrt{\varrho_\kappa}\nabla\phi_k^\infty(\varrho_\kappa)\Delta\sqrt{\varrho_\kappa}\;\dd x\,\dd t \Big| \\
        \leq & C_n\frac{\kappa^{1/4}}{k^{\delta/4}}\int_0^T\!\!\!\int_{\T_L^3}\varrho_\kappa^{1/4}|\varrho_\kappa^{1/4}u_\kappa|^\delta|\kappa^{1/4}\nabla\varrho_\kappa^{1/4}||\kappa^{1/2}\Delta\sqrt{\varrho_\kappa}| \;\dd x\,\dd t \\
        \leq & C(n,T)\frac{\kappa^{1/4}}{k^{\delta/4}}\|\varrho_\kappa\|_{L^\infty(0,T;L^{\frac{1}{1-\delta}})}^{1/4}\|\varrho_\kappa^{1/4}u_\kappa\|_{L^4(0,T;L^4)}^\delta \times\\
        &\times \|\kappa^{1/4}\nabla\varrho_\kappa^{1/4}\|_{L^4(0,T;L^4)}\|\kappa^{1/2}\Delta\sqrt{\varrho_\kappa}\|_{L^2(0,T;L^2)}
    \end{aligned}\]
    and

    \[\begin{aligned}
        \kappa\Big|\int_0^T\!\!\!\int_{\T_L^3} & \psi_n(|v_\kappa|)v_\kappa\cdot\phi_k^\infty(\varrho_\kappa)\nabla\sqrt{\varrho_\kappa}\Delta\sqrt{\varrho_\kappa}\;\dd x\,\dd t\Big| \\
        &\leq C_n\kappa^{1/4}\int_0^T\!\!\!\int_{\T_L^3}|\varrho_\kappa^{1/4}u_\kappa|^\delta\varrho_\kappa^{1/4}|\kappa^{1/4}\nabla\varrho_\kappa^{1/4}||\kappa^{1/2}\Delta\sqrt{\varrho_\kappa}|\;\dd x\,\dd t \\
        &\leq C(n,T)\kappa^{1/4}\|\varrho_\kappa^{1/4}u_\kappa\|_{L^4(0,T;L^4)}^\delta\|\kappa^{1/4}\nabla\varrho_\kappa^{1/4}\|_{L^4(0,T;L^4)}\|\kappa^{1/2}\Delta\sqrt{\varrho_\kappa}\|_{L^2(0,T;L^2)}.
    \end{aligned}\]
    By the estimates in (\ref{energy3}) and (\ref{BD3}), all above terms converge to $0$ with $k\to\infty,\kappa\to 0$ if $\delta\leq 2/3$. To estimate $\tilde J_2$, we further write
\[\begin{aligned} \tilde J_2 &= \int_{\T_L^3}\varrho_\kappa\phi_k^\infty(\varrho_\kappa)\D u_\kappa:\nabla(\psi_n(|v_k|)v_k)\;\dd x + \kappa\int_{\T_L^3}\varrho_\kappa\phi_k^\infty(\varrho_\kappa)\frac{\Delta\sqrt{\varrho_\kappa}}{\sqrt{\varrho_\kappa}}\ddiv(\psi_n(|v_k|)v_k)\;\dd x \\
&= S_1 + S_2. \end{aligned} \]
For $S_1$, we have
    \[\begin{aligned}
    \int_0^T S_1(t)\;\dd t =& \int_0^T\!\!\!\int_{\T_L^3}\varrho_\kappa\phi_k^\infty(\varrho_\kappa)\D u_\kappa:\nabla(\psi_n(|v_\kappa|)v_\kappa) \;\dd x\,\dd t \\
=& \int_0^T\!\!\!\int_{\T_L^3}\frac{\psi_n'(|v_\kappa|)}{|v_\kappa|}\varrho_\kappa(\phi_k^\infty)^2(\varrho_\kappa)\D u_\kappa:(v_\kappa\otimes v_\kappa\nabla v_\kappa)\;\dd x\,\dd t \\
    &+ \int_0^T\!\!\!\int_{\T_L^3} \psi_n(|v_\kappa|)\varrho_\kappa\phi_k^\infty(\varrho_\kappa)\D u_\kappa:\nabla v_\kappa\;\dd x\,\dd t \\
    =& A_1 + A_2.
    \end{aligned}\]

    The term $A_1$ is estimated as follows:
    \[\begin{aligned}
        |A_1| \leq & \int_0^T\!\!\!\int_{\T_L^3}|\psi_n'(|v_\kappa|)||v_\kappa|\varrho_\kappa(\phi_k^\infty)^2(\varrho_\kappa)|\nabla u_\kappa|^2\;\dd x\,\dd t \\
        &+ \int_0^T\!\!\!\int_{\T_L^3}|\psi_n'(|v_\kappa|)||v_\kappa|^2\varrho_\kappa\phi_k^\infty(\varrho_\kappa)|\D u_\kappa||\nabla\phi_k^\infty(\varrho_\kappa)|\;\dd x\,\dd t \\
        =& A_{1,1}+A_{1,2}.
    \end{aligned}\]
    Since $|\psi_n'(|v_\kappa|)||v_\kappa|\leq 2$ independently of $n$, we have 
    \[ A_{1,1} \leq 2\int_0^T\!\!\!\int_{\T_L^3}\varrho_\kappa|\nabla u_\kappa|^2\;\dd x\,\dd t, \]
     which is bounded uniformly by virtue of (\ref{energy3}) and (\ref{BD3}).

For $A_{1,2}$, from the definition of $\psi$ we have
\[ \psi_n'(z)z^2 = \left\{\begin{aligned} 
\frac{2z^3}{1+z^2}, \quad z\leq n, \\
-n\ln(1+z^2) + \frac{4nz^2-2z^3(1-n^2)}{(1+z^2)^2}, \quad z>n \end{aligned} \right. \]
Therefore 
\[ |\psi_n'(v_\kappa|)||v_\kappa|^2\leq C_n+C_n|v_\kappa|^\delta \]
 for some $C_n>0$ and thus
\[ A_{1,2} \leq C_n\int_0^T\!\!\!\int_{\T_L^3} \phi_k^\infty(\varrho_\kappa)(\phi_k^\infty)'(\varrho_\kappa)|\sqrt{\varrho_\kappa}\D u_\kappa||\nabla\varrho_\kappa^{1/4}|\left(\varrho_\kappa^{5/4}+\varrho_\kappa^{5/4-\delta/4}|\varrho_\kappa^{1/4}u_\kappa|^\delta\right) \;\dd x\,\dd t. \]
In consequence
\begin{multline*} A_{1,2} \leq C_n\|\sqrt{\varrho_\kappa}\D u_\kappa\|_{L^2((0,T)\times\T_L^3)}\|\kappa^{1/4}\nabla\varrho_\kappa^{1/4}\|_{L^4((0,T)\times\T_L^3)}\Big(k^{-1/2}\kappa^{-1/4}\|\varrho_\kappa^{3/4}\|_{L^4((0,T)\times\T_L^3)} \\
+k^{-\delta/4}\kappa^{-1/4}\|\varrho_\kappa\|_{L^{\frac{1}{1-\delta}}((0,T)\times\T_L^3)}^{1/4}\|\varrho_\kappa^{1/4}u_\kappa\|_{L^4((0,T)\times\T_L^3)}\Big), \end{multline*}
    which converges to $0$ with $k\to\infty,\kappa\to 0$.

    For $A_2$, we write
    \[\begin{aligned}
        A_2 =& \int_0^T\!\!\!\int_{\T_L^3}\psi_n(|v_\kappa|)\varrho_\kappa\phi_k^\infty(\varrho_\kappa)|\D u_\kappa|^2\;\dd x\,\dd t \\
        &+ \int_0^T\!\!\!\int_{\T_L^3}\psi_n(|v_\kappa|)\varrho_\kappa\phi_k^\infty(\varrho_\kappa)\D u_\kappa:(\nabla\phi_k^\infty(\varrho_\kappa)\otimes u_\kappa)\;\dd x\,\dd t \\
        =& A_{2,1} + A_{2,2}.
    \end{aligned}\]
    The term $A_{2,1}$ is positive and thus we can estimate it in (\ref{after_m}) by $0$. The term $A_{2,2}$ however, is estimated by 
    \[\begin{aligned}
        |A_{2,2}| \leq & \int_0^T\!\!\!\int_{\T_L^3}\psi_n(|v_\kappa|)|v_\kappa|\varrho_\kappa|\D u_\kappa||\nabla\phi_k^\infty(\varrho_\kappa)|\;\dd x\,\dd t \\
        \leq & C(n,T)\frac{\kappa^{-1/4}}{k^{\delta/4}}\|\varrho_\kappa^{1/4}u_\kappa\|_{L^4(0,T;L^4)}^\delta\|\kappa^{1/4}\nabla\varrho_\kappa^{1/4}\|_{L^4(0,T;L^4)} \times \\
        &\times \|\sqrt{\varrho_\kappa}\D u_\kappa\|_{L^2(0,T;L^2)}\|\varrho_\kappa\|_{L^\infty(0,T;L^{\frac{1}{1-\delta}})}^{1/4}
    \end{aligned}\]
    which converges to $0$ as well.
    
    What is left is the term
    \[\begin{aligned} \int_0^T S_2(t)\;\dd t =& \kappa\int_0^T\!\!\!\int_{\T_L^3}\phi_k^\infty(\varrho_\kappa)\sqrt{\varrho_\kappa} \Delta\sqrt{\varrho_\kappa}\mathbb{I}:\nabla(\psi_n(|v_\kappa|)v_\kappa)\;\dd x\,\dd t \\
    =& \kappa\int_0^T\!\!\!\int_{\R^3}\phi_k^\infty(\varrho_\kappa)\sqrt{\varrho_\kappa}\Delta\sqrt{\varrho_\kappa}\frac{\psi_n'(|v_\kappa|)}{|v_\kappa|}\mathbb{I}:(v_\kappa\otimes v_\kappa\nabla v_\kappa)\;\dd x\,\dd t \\
    &+ \kappa\int_0^T\!\!\!\int_{\R^3}\psi_n(|v_\kappa|)\phi_k^\infty(\varrho_\kappa)\sqrt{\varrho_\kappa}\Delta\sqrt{\varrho_\kappa}\ddiv v_\kappa \;\dd x\,\dd t \\
    =& B_1+B_2.
    \end{aligned}\]
    Similarly as before, 
   \[\begin{aligned}
       |B_1| \leq & \kappa^{1/2}\|\kappa^{1/2}\Delta\sqrt{\varrho_\kappa}\|_{L^2(0,T;L^2)}\|\sqrt{\varrho_\kappa}\D u_\kappa\|_{L^2(0,T;L^2)} \\
       &+\kappa^{1/4}\|\kappa^{1/2}\Delta\sqrt{\varrho_\kappa}\|_{L^2(0,T;L^2)}\|\varrho_\kappa^{1/4}u_\kappa\|_{L^4(0,T;L^4)}\|\kappa^{1/4}\nabla\varrho_\kappa^{1/4}\|_{L^4(0,T;L^4)}
   \end{aligned}\]
   and 
   \[\begin{aligned}
       |B_2| \leq & C_n\kappa^{1/2}\|\kappa^{1/2}\Delta\sqrt{\varrho_\kappa}\|_{L^2(0,T;L^2)}\|\sqrt{\varrho_\kappa}\ddiv u_\kappa\|_{L^2(0,T;L^2)} \\
       &+ C_n\kappa^{1/4}\|\kappa^{1/2}\Delta\sqrt{\varrho_\kappa}\|_{L^2(0,T;L^2)}\|\kappa^{1/4}\nabla\varrho_\kappa^{1/4}\|_{L^4(0,T;L^4)}\|\varrho_\kappa^{1/4}u_\kappa\|_{L^4(0,T;L^4)},
   \end{aligned}\]
   which means that both $B_1,B_2\to 0$.

    In conclusion, after performing the limit passage $\kappa\to 0$ in (\ref{after_m}), we end up with the estimate (\ref{MV_n}).
\end{proof}

With the estimate (\ref{MV_n}) at hand, we are ready to finish the proof of Lemma \ref{MV_lem}. Since $F_n\nearrow F$, simply taking the limit $n\to\infty$ in $(\ref{MV_n})$, by monotone convergence theorem we derive (\ref{MV}). In order to show that the limit $(\varrho,u)$ satisfies (\ref{main2}) with $\kappa=0$, we focus only on $\varrho_\kappa\D u_\kappa$, since the limit passage in the remaining terms is performed in the same way as before. We will show that 
\[ \sqrt{\varrho_\kappa} u_\kappa \to \sqrt{\varrho} u \quad \text{in} \quad L^2((0,T)\times\T_L^3). \]
The proof is similar to the proof of Lemma \ref{convergence2}. From the weak convergence of $\varrho_\kappa^{1/4}u_\kappa$ in $L^4((0,T)\times\T_L^3)$ it follows that
\[ \int_0^T\!\!\!\int_{\T_L^3}\varrho|u|^4\;\dd x\,\dd t \leq \liminf_{\kappa\to 0}\int_0^T\!\!\!\int_{\T_L^3} \varrho_\kappa u_\kappa \;\dd x\,\dd t \leq C. \]
From the pointwise convergence of $\varrho_\kappa u_\kappa$ it follows that $\sqrt{\varrho_\kappa}T_M(u_\kappa)\to \sqrt{\varrho}T_M(u)$ a.e. for the truncation $T_M$ defined as in (\ref{truncation}). Since $\sqrt{\varrho_\kappa}T_M(u_\kappa)$ is also uniformly bounded with respect to $\kappa$ in $L^\infty(0,T;L^6)$, in particular it also converges in $L^2((0,T)\times\T_L^3)$ and therefore
\[ \begin{aligned} \|\sqrt{\varrho_\kappa}u_\kappa-\sqrt{\varrho}u\|_{L^2((0,T)\times\T_L^3)}\leq & \|\sqrt{\varrho_\kappa}T_M(u_\kappa)-\sqrt{\varrho}T_M(u)\|_{L^2((0,T)\times\T_L^3)} \\
&+ 2\|\sqrt{\varrho_\kappa}|u_\kappa|\mathbbm{1}_{|u_\kappa|>M}\|_{L^2((0,T)\times\T_L^3)} + 2\|\sqrt{\varrho}|u|\mathbbm{1}_{|u|>M}\|_{L^2((0,T)\times\T_L^3)}. \end{aligned} \]
The last two terms are estimated as follows:
\[ \int_0^T\!\!\!\int_{\T_L^3}\varrho_\kappa|u_\kappa|^2\mathbbm{1}_{|u_\kappa|>M}\;\dd x\,\dd t \leq \frac{1}{M^3}\int_0^T\!\!\!\int_{\T_L^3}\varrho_\kappa|u_\kappa|^4\;\dd x\,\dd t \leq \frac{C}{M^3} \]
and for the limit analogously. In conclusion,
\[ \limsup_{\kappa\to 0}\|\sqrt{\varrho_\kappa}u_\kappa-\sqrt{\varrho}u\|_{L^2((0,T)\times\T_L^3)} \leq \frac{C}{M^3} \]
and the convergence follows by taking $M\to\infty$.

With the above convergence, we pass to the limit using the relation 
\[ \varrho_\kappa\D u_\kappa = \D(\sqrt{\varrho_\kappa}\sqrt{\varrho_\kappa}u_\kappa) - \nabla\sqrt{\varrho_\kappa}\otimes\sqrt{\varrho_\kappa}u_\kappa \]
and weak convergence of $\nabla\sqrt{\varrho_\kappa}$, coming from (\ref{BD3}).
\end{proof}

\begin{rem}\label{Du_rem3}
Note that at this point we cannot use (\ref{Du}) to define $\sqrt{\varrho}\nabla u$ as in the previous section. However, from the uniform estimates we still have the convergence
\[ \sqrt{\varrho_\kappa}\nabla u_\kappa \rightharpoonup \overline{\sqrt{\varrho}\nabla u} \quad \text{in} \quad L^2((0,T)\times\T_L^3). \]
Using the strong convergence of $\sqrt{\varrho_\kappa}$, we can then pass to the limit in the relation
\[ \sqrt{\varrho_\kappa}\sqrt{\varrho_\kappa}\nabla u_\kappa = \nabla(\varrho_\kappa u_\kappa) - \nabla\sqrt{\varrho_\kappa}\otimes\sqrt{\varrho_\kappa}u_\kappa \]
and obtain
\begin{equation}\label{Du2} \sqrt{\varrho}\overline{\sqrt{\varrho}\nabla u} = \nabla(\varrho u)-\nabla\sqrt{\varrho}\otimes\sqrt{\varrho} u. \end{equation}
In particular, from (\ref{Du2}) it follows that $\nabla(\varrho u)\in L^2(0,T;L^1)$. Similarly as before, we will drop the bars and define other differential operators of $u$ analogously.
\end{rem}

\section{Limit passage with $r_0,r_1\to 0$}\label{r_sect}
In the previous section, we constructed the solutions to 
\begin{equation}\label{main3}
    \begin{aligned}
    \partial_t\varrho + \ddiv(\varrho u) &= 0 \\
    \partial_t(\varrho u) + \ddiv(\varrho u\otimes u) -\ddiv(\varrho\D u) + \varrho\nabla(K_L\ast\varrho) &= -r_0u - r_1\varrho|u|^2u,
    \end{aligned}
\end{equation}
defined on the torus $\T_L^3$, satisfying the estimates (\ref{energy3}) and (\ref{BD3}) with $\kappa=0$, together with (\ref{MV}).

Our next goal is to perform the last limit passage $r_0,r_1\to 0$ and in consequence obtain the solutions to (\ref{main}) on the torus.
The main tool to do so is the following lemma:
\begin{lem}\label{r_lem}
    Let $\Omega=\T^3$ or $\R^3$. Assume the sequence $(\varrho_n,u_n)$ satisfies uniformly the following estimates:
    \begin{equation}\label{main_energy2}
        \sup_{t\in[0,T]}\int_\Omega \varrho_n|u_n|^2\;\dd x\,\dd t + \int_0^T\!\!\!\int_\Omega \varrho_n|\D u_n|^2\;\dd x\,\dd t \leq C,
    \end{equation}

    \begin{equation}\label{main_BD2}
        \sup_{t\in[0,T]}\int_\Omega |\nabla\sqrt{\varrho_n}|^2\;\dd x + \int_0^T\!\!\!\int_\Omega \varrho_n|\nabla u_n-\nabla^Tu_n|^2\;\dd x\,\dd t \leq C,
    \end{equation}

    \begin{equation}\label{MV2}
        \sup_{t\in [0,T]}\int_\Omega \varrho_nF(|u_n|)\;\dd x \leq C,
    \end{equation}
    and
    \[ \|\partial_t\varrho_n\|_{L^\infty(0,T;W^{-1,3/2})}, \|\partial_t(\varrho_nu_n)\|_{L^2(0,T;W^{-2,4/3})} \leq C.
    \]
    Then up to a subsequence we have
    \[\begin{aligned}
        \varrho_n\to\varrho &\quad \text{in} \quad C(0,T;L^{3/2}_{\loc})
    \end{aligned} \]
    (in consequence also in particular $\sqrt{\varrho_n}\to \sqrt{\varrho}$ in $C(0,T;L_{\loc}^2)$). Moreover, there exists a function $m$ such that 
    \[ \varrho_nu_n\to m \quad \text{in} \quad L^2(0,T;L^p_{\loc}), \quad p<{3/2}, \]
    and $m(t,x)=0$ a. e. on $\{(t,x): \varrho(t,x)=0\}$. In conclusion, there exists also a function $u$ (defined uniquely on the set $\{(t,x):\varrho(t,x)\neq 0\}$) such that $m=\varrho u$, and moreover
    \[ \sqrt{\varrho_n}u_n\to \sqrt{\varrho}u \quad \text{in} \quad L^2_{\loc}([0,T]\times\Omega).\]
\end{lem}
Lemma \ref{r_lem} combines together the consecutive limit passages performed in \cite{mellet-vasseur} in order to show stability of solutions to (\ref{main}) without the nonlocal term. Because of that, we skip the proof and instead refer to \cite{mellet-vasseur}.

In our case, let $r=r_0=r_1$ and $(\varrho_r,u_r)$ be the corresponding solution to (\ref{main3}). From (\ref{energy3}), (\ref{BD3}) and (\ref{MV}), we have the required uniform estimates. For the time regularity, from the continuity equation, we have
\[ \|\partial_t\varrho_r\|_{W^{-1,3/2}(\T_L^3)} \leq \|\sqrt{\varrho_r}\|_{L^6(\T_L^3)}\|\sqrt{\varrho_r}u_r\|_{L^2(\T_L^3)}. \]
For $\partial_t(\varrho_r u_r)$, the term of the highest order is $\ddiv(\varrho_ru_r\otimes u_r)$. Since for $\varphi\in W^{2,4}(\Omega)$
\[ \int_\Omega \ddiv(\varrho_ru_r\otimes u_r)\cdot\varphi\;\dd x = -\int_\Omega (\sqrt{\varrho_r}u_r\otimes\sqrt{\varrho_r}u_r):\nabla\varphi\;\dd x \leq \|\varrho_r|u_r|^2\|_{L^1(\Omega)}\|\nabla\varphi\|_{L^\infty(\Omega)}, \]
we get the bound
\[ \|\partial_t(\varrho_ru_r)\|_{L^2(0,T;W^{-2,4/3})} \leq C. \]
By Lemma \ref{r_lem}, we can pass to the limit in all terms in the weak formulation of (\ref{main3}). In the convective term we use strong convergence of $\sqrt{\varrho_r}u_r$ and to pass to the limit in the viscous stress tensor we use the relation
\[ \varrho\D u = \D(\varrho u) - 2\nabla\sqrt{\varrho}\otimes\sqrt{\varrho}u. \]
For the nonlocal term, similarly as before we use the fact that $K_L\in L^p(\T_L^3)$ for $p<3/\alpha$. Since $\nabla\varrho_r$ is bounded in $L^\infty(0,T;L^{3/2})$, up to a subsequence we get
\[ K_L\ast\nabla\varrho_r \to K_L\ast\nabla\varrho \quad \text{in} \quad L^\infty(0,T;L^{\frac{3}{\alpha-1}}). \]
The convergence of $\varrho_r\nabla(K_L\ast\varrho_r)$ follows thus from the strong convergence of $\varrho_r$.
For the rest of the terms, we have
\[ r_0\left|\int_0^T\!\!\!\int_{\T_L^3}u_r\varphi \;\dd x\,\dd t\right| \leq \sqrt{r_0}\|\sqrt{r_0}u_r\|_{L^2([0,T]\times\T_L^3)}\|\varphi\|_{L^2([0,T]\times\T_L^3)} \to 0 \]
and
\[ r_1\left|\int_0^T\!\!\!\int_{\T_L^3}\varrho_r|u_r|^2u_r\varphi\;\dd x\,\dd t\right| \leq r_1^{1/4}\|\varrho_r\|_{L^\infty(0,T;L^3)}^{1/4}\|r_1^{1/4}\varrho_r^{1/4}u_r\|_{L^4(0,T;L^4)}^3\|\varphi\|_{L^\infty(0,T;L^6)} \to 0. \]

Using weak lower semicontinuity of the norm, Fatou's lemma and (\ref{MV}), the limit solution also satisfies the inequalities (\ref{energy3}), (\ref{BD3}) and (\ref{MV}) with $r_0=r_1=\kappa=0$.

Note that up to this point we obtained the solutions for initial conditions satisfying 
\[ \tilde\varrho_{0,L}=\varrho_{0,L}+\frac{1}{m_1} \quad \text{and} \quad \sqrt{\tilde\varrho_{0,L}}u_{0,L}\in L^\infty(\T_L^3). \]
 However, Lemma \ref{r_lem} in particular provides the sequential stability of solutions. Repeating the above reasoning, we are able to pass to the limit with $m_1\to\infty$ and with the truncation of initial data. In conclusion, we constructed weak solutions to the system

\begin{equation}\label{last_eq}
\begin{aligned}
        \partial_t\varrho + \ddiv(\varrho u) &= 0, \\
    \partial_t(\varrho u) + \ddiv(\varrho u\otimes u) - \ddiv(\varrho\D u) + \varrho\nabla(K_L\ast\varrho)
    &= 0
\end{aligned}
\end{equation}
on the torus $\T_L^3$, with the initial data $(\varrho_{0,L},u_{0,L})$ defined as in Section \ref{trunc_sect}, satisfying the energy estimate

\begin{equation}\label{final_energy_torus}
    \sup_{t\in[0,T]} E(\varrho,u) + \int_0^T\!\!\!\int_{\T_L^3}\varrho|\D u|^2\dd x\,\dd t
    \leq E(\varrho_0,u_0)
\end{equation}
for
\[ E(\varrho,u)=\frac{1}{2}\int_{\T_L^3} \varrho|u|^2+\varrho(K_L\ast\varrho)\;\dd x, \]
the Bresch--Desjardins estimates
\begin{multline}\label{final_BD_torus}
        \sup_{t\in[0,T]}\int_{\T_L^3} |\nabla\sqrt{\varrho}|^2\;\dd x + \frac{1}{8}\int_0^T\!\!\!\int_{\T_L^3}\varrho|\nabla u-\nabla^Tu|^2\dd x \\
        \leq 3E(\varrho_{0,L},u_{0,L}) + \int_{\T_L^3} |\nabla\sqrt{\varrho_{0,L}}|^2\dd x +CT\|\varrho_{0,L}\|_{L^1(\T_L^3)}^2
\end{multline}
and
\begin{multline}\label{final_MV_torus} 
    \sup_{t\in[0,T]}\Bigg(\int_{\T_L^3}\varrho F(|u|) \;\dd x +\iint_{\T_L^3\times\T_L^3} F(|x-y|)\varrho(x)\varrho(y)\;\dd x\,\dd y\Bigg) \\
    \leq C + C\left(\int_{\T_L^3}\varrho_{0,L} F(|u_{0,L}|) + \iint_{\T_L^3\times\T_L^3} F(|x-y|)\varrho_{0,L}(x)\varrho_{0,L}(y)\;\dd x\,\dd y\right),
    \end{multline}
    where $C$ depends on $E(\varrho_{0,L},u_{0,L})$ and the right hand side of (\ref{final_BD_torus}).

 \section{Expansion of the torus}\label{blowup_sect}

 Having the solutions to (\ref{last_eq}) defined on the torus $\T_L^3$, together with estimates (\ref{final_energy_torus})--(\ref{final_MV_torus}), we can now pass to the limit with $L\to\infty$ and in consequence obtain the solutions on the whole space $\R^3$. Let $(\varrho_L,u_L)$ be the solutions to (\ref{last_eq}) and by $(\tilde\varrho_L,\tilde u_L)$ we will denote $(\varrho_L,u_L)$ extended by zero outside the torus. Using the properties of $\varrho_{0,L}$ described in Lemma \ref{initial}, we have the following estimates, uniform in $L$:
\begin{multline}\label{energy4}
    \sup_{t\in[0,T]} \frac{1}{2}\int_{\R^3}\left(\tilde\varrho_L|\tilde u_L|^2 +\tilde\varrho_L(K_L\ast\tilde\varrho_L)\right)\dd x + \int_0^T\!\!\!\int_{[-L,L]^3}\tilde\varrho_L|\D \tilde u_L|^2\dd x\,\dd t \\
    \leq \frac{1}{2}\int_{\R^3} \varrho_0|u_0|^2+\varrho_0(K\ast\varrho_0)\;\dd x
\end{multline}
and
\begin{multline}\label{BD4}
        \sup_{t\in[0,T]}\int_{[-L,L]^3}|\nabla\sqrt{\tilde\varrho_L}|^2\;\dd x + \frac{1}{8}\int_0^T\!\!\!\int_{[-L,L]^3}\tilde\varrho_L|\nabla \tilde u_L-\nabla^T\tilde u_L|^2\dd x \\
        \leq \int_{\R^3} \varrho_0|u_0|^2+\varrho_0(K\ast\varrho_0)\;\dd x + \int_{\R^3} |\nabla\sqrt{\varrho_0}|^2\dd x +\frac{C}{L^2}\|\varrho_0\|_{L^1(\R^3)}^{1/2} +CT\|\varrho_0\|_{L^1(\R^3)}^2.
\end{multline}

From Lemma \ref{initial} it also follows that the right hand sides of (\ref{final_energy_torus}) and (\ref{final_BD_torus}) are independent of $L$. As a consequence, from (\ref{final_MV_torus}) we also get the uniform bound
\begin{equation}\label{MV4}
    \sup_{t\in[0,T]}\Bigg(\int_{\R^3}\tilde\varrho_L F(|\tilde u_L|) \;\dd x + \iint_{\R^3\times\R^3} F(|x-y|)\tilde\varrho_L(x)\tilde\varrho_L(y)\;\dd x\,\dd y\Bigg) \leq C
\end{equation}
for $C$ depending on $T$ and $\varrho_0,u_0$.

Now, let $V\subset\R^3$ be compact. Then $V\subset [-L,L]^3$ for sufficiently large $L$ and the estimates (\ref{energy4}), (\ref{BD4}) and (\ref{MV4}) provide the uniform estimates on $(\tilde\varrho_L,\tilde u_L)$ needed in Lemma \ref{MV_lem} for $\Omega=V$. Moreover, using analogous arguments as in the previous section, 
\[ \|\partial\tilde\varrho_L\|_{L^\infty(0,T;W^{-1,3/2}(V))}, \|\partial_t(\tilde\varrho_L,\tilde u_L)\|_{L^2(0,T;W^{2,4/3}(V))} \leq C. \]
Therefore up to a subsequence
we get the convergence from Lemma \ref{MV_lem} for $\Omega=V$. By the arbitrary choice of $V$ and applying the diagonal method, we finally get
\[\begin{aligned}
        \sqrt{\tilde\varrho_L}\to\sqrt{\varrho} &\quad \text{in} \quad C(0,T;L^2_{\loc}), \\
        \tilde\varrho_L\to\varrho &\quad \text{in} \quad C(0,T;L^{3/2}_{\loc}), \\
        \tilde\varrho_L\tilde u_L\to \varrho u &\quad \text{in} \quad L^2(0,T;L^{3/2}_{\loc}), \\
        \sqrt{\tilde\varrho_L}u_L\to \sqrt{\varrho}u &\quad \text{in} \quad L^2_{\loc}([0,T]\times\R^3).
    \end{aligned}\]
 Similarly from the Banach-Alaoglu theorem we get
 \[ \nabla\sqrt{\tilde\varrho_L}\mathbbm{1}_{[-L,L]^d}\rightharpoonup^* \nabla\sqrt{\varrho} \quad \text{in} \quad L^\infty(0,T;L^2_{\loc}). \]
Then for the term $\tilde\varrho_L\D u_L$ the relation (\ref{Du2}) again provides convergence in the sense of distributions.

\subsection{Convergence of the nonlocal term}

The obtained convergence allows us to pass to the limit with $(\tilde\varrho_L,\tilde u_L)$ in the weak formulation of (\ref{last_eq}) in all terms except the nonlocal one. Note that since $K$ is unbounded, as $L\to\infty$ we lose any compactness properties of $K_L$. However, due to (\ref{energy4}) we are able to show that

\begin{lem}\label{nonlocal_lem}
   \[ \tilde\varrho_L (\nabla K_L\ast\tilde\varrho_L) \to \varrho (\nabla K\ast\varrho) \quad \text{in} \quad L^1([0,T]\times\R^3). \]
\end{lem}

\begin{proof}
First, note that from the strong convergence of $\tilde\varrho_L$ it follows that up to a subsequence
\[ \tilde\varrho_L \to \varrho \quad \text{a. e. in} \quad [0,T]\times\R^3. \]
From (\ref{energy4}) in particular we have that
\[ \sup_{t\in[0,T]}\iint_{\R^3\times\R^3}\tilde\varrho_L(t,x)\tilde\varrho_L(t,y)|x-y|^2\phi_L(x-y)\;\dd x\,\dd y \leq C. \]
Therefore by Fatou's lemma
\begin{multline}\label{moment_lim} \iint_{\R^3\times\R^3}\varrho(t,x)\varrho(t,y)|x-y|^2\;\dd x\,\dd y \\
\leq \liminf_{L\to\infty}\iint_{\R^3\times\R^3}\tilde\varrho_L(t,x)\tilde\varrho_L(t,y)|x-y|^2\phi_L(x-y) \;\dd x\,\dd y \leq C. \end{multline}

Since $\nabla\sqrt{\varrho_L}$ is bounded in $L^\infty(0,T;L^2_{loc})$, we also have the bound on $\|\varrho_L\|_{L^\infty(0,T;L^3_{loc})}$ and in consequence
\[ \tilde\varrho_L \to \varrho \quad \text{in} \quad C(0,T;L^p_{\loc}) \]
for any $3/2<p<3$. 

Now fix $R>0$ and denote by $B_R$ the ball of radius $R$ centered in zero. From the definition of $K_L$ and the strong convergence of $\tilde\varrho_L$ we know that $\nabla K_L\to K$ and $\tilde\varrho_L\to\varrho$ a. e. Moreover, $\nabla K_L\mathbbm{1}_{B_R}$ is uniformly bounded with respect to $L$ in $L^p(\R^3)$ for $p<\frac{3}{\alpha+1}$. Therefore we also have
\[ \|\tilde\varrho_L(\nabla K_L\mathbbm{1}_{B_R}\ast\varrho_L)\|_{L^\infty(0,T;L^q)}\leq \|\nabla K_L\mathbbm{1}_{B_R}\|_{L^p(\R^3)}\|\tilde\varrho_L\|_{L^\infty(0,T;L^3)}^2 \]
for $q<\frac{3}{\alpha}$ and in consequence
\[ \tilde\varrho_L(\nabla K_L\mathbbm{1}_{B_R}\ast\tilde\varrho_L) \to \varrho(\nabla K\mathbbm{1}_{B_R}\ast\varrho) \quad \text{in} \quad L^1([0,T]\times\R^3). \]
We  now estimate the rest. We have
\begin{equation*} I=\left|\int_{\R^3} \tilde\varrho_L\left((\nabla K_L\left((1-\mathbbm{1}_{B_R}\right)\ast\tilde\varrho_L\right) \;\dd x\right|
\leq \iint_{|x-y|>R}\tilde\varrho_L(t,x)\tilde\varrho_L(t,y)|\nabla K_L(x-y)|\;\dd x\,\dd y. \end{equation*}
Using the definition of $K_L$ and $\phi_L$, we have
\[\begin{aligned} |\nabla K_L(x-y)| \leq & \phi_L(x-y)\left(\frac{1}{|x-y|^{\alpha+1}} + |x-y|\right) \\
&+ |\nabla\phi_L(x-y)|\left(\frac{1}{|x-y|^\alpha} + \frac{1}{2}|x-y|^2\right) \\
\leq & \frac{1}{|x-y|^{\alpha+1}}+\phi_L(x-y)|x-y| + \frac{C}{L}\left(\frac{1}{|x-y|^\alpha} + \frac{1}{2}|x-y|^2\right). \end{aligned}\]
Therefore
\[\begin{aligned} I
\leq & \frac{1}{R^{\alpha+1}}\iint_{\R^3\times\R^3} \tilde\varrho_L(t,x)\tilde\varrho_L(t,y)\;\dd x\,\dd y + \iint_{|x-y|>R} \phi_L(x-y)\tilde\varrho_L(t,x)\tilde\varrho_L(t,y)|x-y|\;\dd x\,\dd y \\
&+ \frac{C}{LR^\alpha}\iint_{\R^3\times\R^3}\tilde\varrho_L(t,x)\tilde\varrho_L(t,y)\;\dd x\,\dd y + \frac{C}{L}\iint_{|x-y|>R}\tilde\varrho_L(t,x)\tilde\varrho_L(t,y)|x-y|^2\dd x\,\dd y.
\end{aligned}\]
Let us now estimate all terms on the right hand side. For the first and the third term, we have
\[ \iint_{\R^3\times\R^3}\tilde\varrho_L(t,x)\tilde\varrho_L(t,y)\;\dd x\,\dd y = \|\varrho_L\|_{L^1(\T_L^3)}^2=\|\varrho_{0,L}\|_{L^1(\T_L^3)}^2 \leq \|\varrho_0\|_{L^1(\R^3)}^2. \]
For the second term, by (\ref{energy4}) we get
\begin{multline*} \iint_{|x-y|>R}\phi_L(x-y)\tilde\varrho_L(t,x)\tilde\varrho_L(t,x)|x-y|\;\dd x\,\dd y \\
\leq \frac{1}{R}\iint_{\R^3\times\R^3}\phi_L(x-y)\tilde\varrho_L(t,x)\tilde\varrho_L(t,y)|x-y|^2\;\dd x\,\dd y \leq \frac{C}{R}. \end{multline*}
Finally we estimate the last term using (\ref{MV4}) as
\begin{multline*} \iint_{\R^{3\times 3}}\tilde\varrho_L(t,x)\tilde\varrho_L(t,y)|x-y|^2\mathbbm{1}_{|x-y|>R}\;\dd x\,\dd y \\
\leq \frac{1}{\ln(1+R^2)}\iint_{\R^{3\times 3}}\tilde\varrho_L(t,x)\tilde\varrho_L(t,y)F(|x-y|)\;\dd x\,\dd y \leq \frac{C}{\ln(1+R^2)}. \end{multline*}

In consequence,
\[ I \leq \frac{C}{R}+\frac{C}{L\ln(1+R^2)}. \]
Doing analogous estimates for $K$, using (\ref{moment_lim}), we get
\begin{multline*} \left|\iint_{|x-y|>R}\varrho(t,x)\varrho(t,y)\nabla K(x-y)\;\dd x\,\dd y\right| \\
\leq \frac{1}{R^2}\iint_{\R^{3\times 3}} \varrho(t,x)\varrho(t,y)\;\dd x\,\dd y + \iint_{|x-y|>R}\varrho(t,x)\varrho(t,y)|x-y|\;\dd x\,\dd y \leq \frac{C}{R}. \end{multline*}
Combining the above estimates, we obtain
\[ \liminf_{L\to\infty}\int_0^T\!\!\!\int_{\R^3} \left|\tilde\varrho_L\nabla K_L\ast\tilde\varrho_L-\varrho\nabla K\ast\varrho)\right| \;\dd x\,\dd t \leq \frac{C}{R}. \]
Finally, we end the proof by taking $R\to\infty$.
\end{proof}

Lemma \ref{nonlocal_lem} finishes the limit passage in the weak formulation of (\ref{last_eq}). In consequence, we obtain the weak solution to (\ref{main}) on $[0,T]\times\R^3$. Since $\nabla\sqrt{\tilde\varrho_L} \rightharpoonup \nabla\sqrt{\varrho}$ in $L^2(0,T;L^2_{loc})$, 
\[ \int_{\R^3}|\nabla\sqrt{\varrho}|^2\;\dd x = \lim_{R\to\infty}\int_{B(0,R)}|\nabla\sqrt{\varrho}|^2\;\dd x \leq \lim_{R\to\infty}\liminf_{L\to\infty}\int_{B(0,R)}|\nabla\sqrt{\tilde\varrho_L}|^2\;\dd x \leq C(T). \]
Doing analogously with $\sqrt{\varrho}\nabla u$, we finally show that the solution satisfies the estimates (\ref{main_energy}), (\ref{main_BD}) and (\ref{main_MV}).

\subsection{Mass preservation}

In the end, let us conclude that

\begin{lem}
    The total mass is conserved, i.e. 
    $\displaystyle \int_{\R^3}\varrho \;\dd x = \int_{\R^3}\varrho_0 \;\dd x. $
\end{lem}
\begin{proof}
    By Fatou's lemma, we have
    \[ \int_{\R^3}\varrho \;\dd x \leq \liminf_{L\to\infty}\int_{\R^3}\varrho_L\;\dd x = \liminf_{L\to\infty}\int_{\T_L^3}\varrho_{L,0}\dd x \leq \int_{\R^3}\varrho_0\;\dd x. \]
    
    On the other hand, for a smooth function $\phi_R$ such that $\phi_R(x)=1$ for $|x|<R/2$, $\phi_R\in (0,1)$ and $\mathrm{supp}\phi_R\subset B_R$, $|\nabla\phi_R|\leq\frac{C}{R}$, for $L$ large enough we have
    \[\begin{aligned} \int_{\R^3}\tilde\varrho_L(t,x)\phi_R(x)\;\dd x =& \int_{\T_L^3}\varrho_L(0,x)\phi_R(x)\;\dd x + \int_{\T_L^3}\varrho_L(t,x)u_L(t,x)\cdot\nabla\phi_R(x)\;\dd x \\
    \geq & \int_{\T_L^3}\varrho_{0,L}(x)\phi_R(x)\;\dd x - \frac{C}{R}.
    \end{aligned}\]
    
    Therefore

     \[\begin{aligned} \int_{\R^3}\varrho \;\dd x \geq \int_{|x|\leq R}\varrho \;\dd x = \lim_{L\to\infty}\int_{|x|\leq R}\tilde\varrho_L\dd x \geq \lim_{L\to\infty}\int_{\T_L^3}\varrho_{L,0}\phi_R\;\dd x - \frac{C}{R} \geq \int_{\R^3}\varrho_0\;\dd x - \frac{C}{R}, \end{aligned}\]
    and taking  $R\to +\infty$ we conclude the proof.
\end{proof}

\begin{appendices}
    \section{Proof of Bresch-Desjardins estimates}\label{BD_app}
    Below we show how to derive the inequality (\ref{BD}). We need to compute
    \begin{equation}\label{rach1}
    \begin{aligned} \frac{\dd}{\dd t}\int_{\T_L^3}\left(\frac{1}{2}\varrho|u+\nabla\log\varrho|^2 + \varrho(K_L\ast\varrho)+\frac{\delta}{2}|\nabla\Delta\varrho|^2+\frac{\kappa}{2}|\nabla\sqrt{\varrho}|^2+\frac{\eta}{7}\varrho^{-6} \right)\dd x \\
    = \frac{\dd}{\dd t} E(\varrho,u) + \frac{\dd}{\dd t}\int_{\T_L^3}\varrho u\cdot\nabla\log\varrho \;\dd x + \frac{\dd}{\dd t}\int_{\T_L^3}\varrho|\nabla\log\varrho|^2\;\dd x. \end{aligned}\end{equation}
    For the first term on the right hand side of (\ref{rach1}), we use the energy inequality (\ref{energy}). For the second term, we have
    \begin{equation}\label{bd2}
    \begin{aligned}
    \int_{\T_L^3}\varrho u\cdot\partial_t\nabla\log\varrho \;\dd x &= - \int_{\T_L^3}\ddiv(\varrho u)\frac{1}{\varrho}\partial_t\varrho \;\dd x \\
    &= \int_{\T_L^3}\frac{1}{\varrho}(\ddiv(\varrho u))^2\;\dd x - \varepsilon\int_{\T_L^3}\frac{1}{\varrho}\Delta\varrho\ddiv(\varrho u) \;\dd x
    \end{aligned}\end{equation}
    and from the momentum equation
    \[\begin{aligned} \int_{\T_L^3}\partial_t(\varrho u)\cdot\nabla\log\varrho \;\dd x =& -\int_{\T_L^3}\ddiv(\varrho u\otimes u)\cdot\nabla\log\varrho\;\dd x + \int_{\T_L^3}\ddiv(\varrho\D u)\cdot\nabla\log\varrho\;\dd x \\
    &- \int_{\T_L^3}\nabla(K_L\ast\varrho)\cdot\nabla\log\varrho \;\dd x + \kappa\int_{\T_L^3}\varrho\nabla\left(\frac{\Delta\sqrt{\varrho}}{\sqrt{\varrho}}\right)\cdot\nabla\log\varrho\;\dd x \\
    &-r_0\int_{\T_L^3}u\cdot\nabla\log\varrho\;\dd x - r_1\int_{\T_L^3}\varrho|u|^2u\cdot\nabla\log\varrho \;\dd x \\
    &-\varepsilon\int_{\T_L^3}\nabla\varrho\cdot\nabla u\cdot\nabla\log\varrho \;\dd x - \nu\int_{\T_L^3}\Delta^2u\cdot\nabla\log\varrho \\
    &+\eta\int_{\T_L^3}\nabla\varrho^{-6}\cdot\nabla\log\varrho\;\dd x + \delta\int_{\T_L^3}\varrho\nabla\Delta^3\varrho\cdot\nabla\log\varrho\;\dd x \\
    =& -\int_{\T_L^3}\ddiv(\varrho u\otimes u)\cdot\nabla\log\varrho\;\dd x - \int_{\T_L^3}\varrho\D u:\nabla^2\log\varrho\;\dd x \\
    &- \int_{\T_L^3}\nabla(K_L\ast\varrho)\cdot\nabla\log\varrho\;\dd x -\frac{\kappa}{2}\int_{\T_L^3}\varrho|\nabla^2\log\varrho|^2\;\dd x \\
    &-r_0\int_{\T_L^3}u\cdot\frac{1}{\varrho}\nabla\varrho\;\dd x - r_1\int_{\T_L^3}|u|^2u\cdot\nabla\varrho \;\dd x \\
    &-\varepsilon\int_{\T_L^3}\nabla\varrho\cdot\nabla u\cdot\nabla\log\varrho \;\dd x -\nu\int_{\T_L^3}\Delta u\cdot\nabla\Delta\log\varrho\;\dd x \\
    &- \frac{2}{3}\eta\int_{\T_L^3}|\nabla\varrho^{-3}|^2\;\dd x -\delta\int_{\T_L^3}|\Delta^2\varrho|^2\;\dd x.
    \end{aligned}\]

    Note that
    \[\begin{aligned} -\int_{\T_L^3}\varrho\D u:\nabla^2\log\varrho\;\dd x =& -
    \frac{1}{2}\sum_{i,j}\int_{\T_L^3}\varrho(\partial_{x_i}u_j+\partial_{x_j}u_i)\partial_{x_ix_j}\log\varrho \;\dd x \\
    =& \sum_{i,j}\int_{\T_L^3}\Big(\partial_{x_j}(\varrho\partial_{x_i}u_j)\partial_{x_i}\log\varrho + \partial_{x_i}(\varrho\partial_{x_j}u_i)\partial_{x_j}\log\varrho\Big)\dd x \\
    =& \int_{\T_L^3} \nabla u:(\nabla\varrho\otimes\nabla\log\varrho)\;\dd x + \int_{\T_L^3}\varrho\nabla\ddiv u\cdot\nabla\log\varrho\;\dd x \\
    =& \int_{\T_L^3} \nabla u:(\nabla\varrho\otimes\nabla\log\varrho)\;\dd x - \int_{\T_L^3}\Delta\varrho\ddiv u\;\dd x,
    \end{aligned}\]
which gives
\begin{equation}\label{bd3}\begin{aligned} \int_{\T_L^3}\partial_t(\varrho u)\cdot\nabla\log\varrho \;\dd x =& -\int_{\T_L^3}\ddiv(\varrho u\otimes u)\cdot\nabla\log\varrho\;\dd x + \int_{\T_L^3}\nabla u:(\nabla\varrho\otimes\nabla\log\varrho)\;\dd x \\
&-\int_{\T_L^3}\Delta\varrho\ddiv u\;\dd x \\
    &- \int_{\T_L^3}\nabla(K_L\ast\varrho)\cdot\nabla\log\varrho\;\dd x -\frac{\kappa}{2}\int_{\T_L^3}\varrho|\nabla^2\log\varrho|^2\;\dd x \\
    &-r_0\int_{\T_L^3}u\cdot\frac{1}{\varrho}\nabla\varrho\;\dd x - r_1\int_{\T_L^3}|u|^2u\cdot\nabla\varrho \;\dd x \\
    &-\varepsilon\int_{\T_L^3}\nabla\varrho\cdot\nabla u\cdot\nabla\log\varrho \;\dd x -\nu\int_{\T_L^3}\Delta u\cdot\nabla\Delta\log\varrho\;\dd x \\
    &- \frac{2}{3}\eta\int_{\T_L^3}|\nabla\varrho^{-3}|^2\;\dd x -\delta\int_{\T_L^3}|\Delta^2\varrho|^2\;\dd x.
\end{aligned} \end{equation}

To compute the last term on the right hand side of (\ref{rach1}), we see that from the continuity equation
\[\begin{aligned} \partial_t\frac{|\nabla\log\varrho|^2}{2} =& \nabla\log\varrho\cdot\partial_t\nabla\log\varrho \\
=& \nabla\log\varrho\cdot\nabla (-\ddiv((\log\varrho) u)+(\log\varrho-1)\ddiv u +\varepsilon\frac{1}{\varrho}\Delta\varrho) \\
    =& -u\cdot\nabla^2\log\varrho\cdot\nabla\log\varrho - \nabla u:(\nabla\log\varrho\otimes\nabla\log\varrho) - \nabla\log\varrho\cdot\nabla\ddiv u \\
    &+ \varepsilon\nabla\log\varrho\cdot\nabla\left(\frac{1}{\varrho}\Delta\varrho\right).
    \end{aligned}\]
The above calculations are justified, since $u\in L^2(0,T;H^2)$, $\varrho\in L^\infty(0,T;H^3)$ and $\varrho$ is bounded away from zero. Using the above calculations, we derive

\begin{equation}\label{bd4}
    \begin{aligned}
    \frac{\dd}{\dd t}\int_{\T_L^3}\varrho\frac{|\nabla\log\varrho|^2}{2}\dd x =& \int_{\T_L^3}\varrho\partial_t\frac{|\nabla\log\varrho|^2}{2}\dd x - \int_{\T_L^3}\frac{|\nabla\log\varrho|^2}{2}\ddiv(\varrho u)\dd x \\
    &+ \varepsilon\int_{\T_L^3}\frac{|\nabla\log\varrho|^2}{2}\Delta\varrho\;\dd x \\
    =& -\int_{\T_L^3}\varrho u\cdot\nabla^2\log\varrho\cdot\nabla\log\varrho)\;\dd x - \int_{\T_L^3}\varrho\nabla u:(\nabla\log\varrho\otimes\nabla\log\varrho)\;\dd x \\
    &-\int_{\T_L^3}\varrho\nabla\log\varrho\cdot\nabla\ddiv u\;\dd x + \int_{\T_L^3}\varrho u\cdot\nabla\frac{|\nabla\log\varrho|^2}{2}\dd x \\
    &+\varepsilon\int_{\T_L^3}\varrho\nabla\log\varrho\cdot\nabla\left(\frac{1}{\varrho}\Delta\varrho\right)\;\dd x +\varepsilon\int_{\T_L^3}\frac{|\nabla\log\varrho|^2}{2}\Delta\varrho\;\dd x \\
    =& -\int_{\T_L^3}\varrho\nabla u:(\nabla\log\varrho\otimes\nabla\log\varrho)\dd x + \int_{\T_L^3}\Delta\varrho\ddiv u\;\dd x \\
    & -\varepsilon\int_{\T_L^3}\frac{|\Delta\varrho|^2}{\varrho}\;\dd x +\varepsilon\int_{\T_L^3}\frac{|\nabla\log\varrho|^2}{2}\Delta\varrho\;\dd x
    \end{aligned}
    \end{equation}

Since $\nabla\varrho\cdot\nabla\log\varrho = \varrho|\nabla\log\varrho|^2$, 
and $\nabla\varrho\otimes\nabla\log\varrho = \varrho\nabla\log\varrho\otimes\nabla\log\varrho,$
combining (\ref{bd2}), (\ref{bd3}) and (\ref{bd4}), we get
\[\begin{aligned} \frac{\dd}{\dd t}\int_{\T_L^3}\varrho\frac{|\nabla\log\varrho|^2}{2} + \varrho u\cdot\nabla\log\varrho \dd x =& \int_{\T_L^3}\frac{1}{\varrho}(\ddiv(\varrho u))^2\;\dd x - \varepsilon\int_{\T_L^3}\frac{1}{\varrho}\Delta\varrho\ddiv(\varrho u) \;\dd x  \\
&-\int_{\T_L^3}\ddiv(\varrho u\otimes u)\cdot\nabla\log\varrho\;\dd x \\
&-\varepsilon\int_{\T_L^3}\frac{|\Delta\varrho|^2}{\varrho}\;\dd x  + \varepsilon\int_{\T_L^3}\frac{|\nabla\log\varrho|^2}{2}\Delta\varrho\;\dd x \\
&- \int_{\T_L^3}\nabla(K_L\ast\varrho)\cdot\nabla\log\varrho\;\dd x -\frac{\kappa}{2}\int_{\T_L^3}\varrho|\nabla^2\log\varrho|^2\;\dd x \\
    &-r_0\int_{\T_L^3}u\cdot\frac{1}{\varrho}\nabla\varrho\;\dd x - r_1\int_{\T_L^3}|u|^2u\cdot\nabla\varrho \;\dd x \\
    &-\varepsilon\int_{\T_L^3}\nabla\varrho\cdot\nabla u\cdot\nabla\log\varrho \;\dd x -\nu\int_{\T_L^3}\Delta u\cdot\nabla\Delta\log\varrho\;\dd x \\
    &- \frac{2}{3}\eta\int_{\T_L^3}|\nabla\varrho^{-3}|^2\;\dd x -\delta\int_{\T_L^3}|\Delta^2\varrho|^2\;\dd x. \end{aligned}\]
We have the relations 
\[\begin{aligned}
-\nabla\log\varrho\cdot\ddiv(\varrho u\otimes u) &= -\frac{1}{\varrho}\sum_{i,j}\partial_{x_i}\varrho\Big(\partial_{x_j}\varrho u_iu_j + \varrho\partial_{x_j}u_iu_j + \varrho u_i\partial_{x_j}u_j\Big) \\
&= -\frac{1}{\varrho}\Big((u\cdot\nabla\varrho)^2 + \varrho u\cdot(\nabla u\nabla\varrho) + \varrho\ddiv u u\cdot\nabla\varrho\Big)
\end{aligned}\]
and
\[ \frac{1}{\varrho}(\ddiv(\varrho u))^2 = \frac{1}{\varrho}\Big(\varrho^2(\ddiv u)^2 + 2\varrho\ddiv u u\cdot\nabla\varrho + (u\cdot\nabla\varrho)^2\Big), \]
therefore
\[\begin{aligned}\int_{\T_L^3}-\nabla\log\varrho\cdot\ddiv(\varrho u\otimes u) + \frac{1}{\varrho}(\ddiv(\varrho u))^2\;\dd x =& \int_{\T_L^3} \ddiv u u\cdot\nabla\varrho - u\cdot(\nabla u\nabla\varrho) + \varrho(\ddiv u)^2 \;\dd x \\
=& \sum_{i,j}\int_{\T_L^3}\varrho\partial_{x_j}u_i\partial_{x_i}u_j\;\dd x
\end{aligned}\]
Subtracting $\int_{\T_L^3}\varrho|\D u|^2\;\dd x$, we get
\[\begin{aligned} \int_{\T_L^3}\varrho\sum_{i,j}\left(-\left(\frac{\partial_{x_j}u_i+\partial_{x_i}u_j}{2}\right)^2+\partial_{x_j}u_i\partial_{x_i}u_j\right)\;\dd x &= -\int_{\T_L^3}\varrho\sum_{i,j}\left(\frac{\partial_{x_j}u_i-\partial_{x_i}u_j}{2}\right)^2 \;\dd x \\
&= -\frac{1}{4}\int_{\T_L^3}\varrho|\nabla u-\nabla^Tu|^2\;\dd x. \end{aligned}\]

Combining the above calculations with (\ref{energy}) and integrating in time, we finally obtain
\[
\begin{aligned}
        E_{BD} & (\varrho,u) + \frac{2}{3}\eta(1+\varepsilon)\int_0^T\!\!\!\int_{\T_L^3}|\nabla\varrho^{-3}|^2\dd x\,\dd t + \delta(1+\varepsilon)\int_0^T\!\!\!\int_{\T_L^3}|\Delta^2\varrho|^2\dd x\,\dd t \\
        &+ \frac{1}{4}\int_0^T\!\!\!\int_{\T_L^3}\varrho|\nabla u-\nabla^Tu|^2\dd x\,\dd t + \nu\int_0^T\!\!\!\int_{\T_L^3}|\Delta u|^2\dd x\,\dd t \\
        &+ r_0\int_0^T\!\!\!\int_{\T_L^3}|u|^2\;\dd x\,\dd t + r_1\int_0^T\!\!\!\int_{\T_L^3}\varrho|u|^4\;\dd x\,\dd t \\
        &+ \frac{\kappa(1+\varepsilon)}{2}\int_0^T\!\!\!\int_{\T_L^3}\varrho|\nabla^2\log\varrho|^2\dd x\,\dd t + \varepsilon\int_0^T\!\!\!\int_{\T_L^3}\frac{|\Delta\varrho|^2}{\varrho}\dd x + \int_0^T\!\!\!\int_{\T_L^3}\nabla(K_L\ast\varrho)\cdot\nabla\varrho \;\dd x \\
        \leq & E_{BD}(\varrho_0,u_0) + C\varepsilon T\left(\int_{\T_L^3}\varrho_0\;\dd x\right)^2 \\
        &+ \varepsilon\int_0^T\!\!\!\int_{\T_L^3}\left(\nabla\varrho\cdot\nabla u\cdot\nabla\log\varrho + \Delta\varrho\frac{|\nabla\log\varrho|^2}{2} - \ddiv(\varrho u)\frac{1}{\varrho}\Delta\varrho\right)\dd x \\
        &- \nu\int_0^T\!\!\!\int_{\T_L^3}\Delta u\cdot\nabla\Delta\log\varrho \;\dd x - r_1\int_0^T\!\!\!\int_{\T_L^3}|u|^2u\nabla\varrho\;\dd x - r_0\int_0^T\!\!\!\int_{\T_L^3}\frac{u\cdot\nabla\varrho}{\varrho}\;\dd x
        \end{aligned} \]
for
\[ E_{BD}(\varrho,u) = \int_{\T_L^3} \left(\frac{1}{2}\varrho\left|u+\frac{1}{\varrho}\nabla\varrho\right|^2 + \varrho(K_L\ast\varrho) + \frac{\delta}{2}|\nabla\Delta\varrho|^2 + \frac{\kappa}{2}|\nabla\sqrt{\varrho}|^2 + \frac{\eta}{7}\varrho^{-6}\right)\dd x. \]

\section{Weak Gronwall's lemma}
    Below, let us present the weak version of Gronwall's Lemma, which becomes useful in Section \ref{MV_sect}:

\begin{lem}[Weak version of Gronwall's lemma]\label{weak_gronwall}
    Let $f\in L^1(0,T)$ satisfy
    \[ -\int_0^T \xi'(s)f(s)\dd s \leq \int_0^T \xi(s)(af(s) + b(s))\dd s \]
    for any $\xi\in C^\infty_0(0,T)$, $\xi\geq 0$, a constant $a\geq 0$ and nonnegative function $b\in L^1(0,T)$. Then for almost all $0\leq s<t<T$ we have
    \[ f(t) \leq f(s)e^{a(t-s)}+\int_s^t e^{a(t-\tau)}b(\tau) \;\dd\tau \]
\end{lem}
\begin{proof}
    Let $f_\varepsilon=f\ast\eta_\varepsilon$, where $\eta_\varepsilon$ is a standard mollifier. Fix $t\in (0,T)$ and let $\xi(s)=\eta_\varepsilon(t-s)$. Then $f$ satisfies
    \[ \int_0^T f(s)\eta_\varepsilon'(t-s)\dd s \leq a\int_0^T f(s)\eta_\varepsilon(t-s)\dd s + \int_0^T b(s)\eta_\varepsilon(t-s)\;\dd s, \]
    which is equivalent to
    \[ f_\varepsilon'(t) \leq af_\varepsilon(t) + b_\varepsilon(t). \]
    Then from Gronwall inequality on $f_\varepsilon$, we get
    \[ f_\varepsilon(t) \leq f_\varepsilon(s)e^{a(t-s)} + \int_s^t e^{a(t-\tau)}b_\varepsilon(\tau)\dd\tau, \]
    for any $0\leq s<t<T$. Choosing $s,t$ such that $f_\varepsilon\to f$ pointwise in $s,t$ and passing to the limit with $\varepsilon\to 0$, we get
    \[ f(t)\leq f(s)e^{a(t-s)}+\int_s^te^{a(t-\tau)}b(\tau)\dd\tau. \]
\end{proof}

\end{appendices}

\bigbreak

\noindent
\textbf{Acknowledgement:} The first author (PBM) has been partly supported by the Polish National Science Centre’s Grant No. 2018/30/M/ST1/00340 (HARMONIA). The work of M. S. was supported by the National Science Centre grant no. 2022/45/N/ST1/03900 (Preludium). The work of E.Z. was supported by the EPSRC Early Career Fellowship no. EP/V000586/1.

\bibliographystyle{abbrv}
\bibliography{biblio.bib}

\end{document}